\documentclass[a4paper,12pt]{amsart}

\usepackage{hyperref}
\usepackage{tikz}
\usepackage{amscd}
\usepackage{amsmath}
\usepackage{amsthm}
\usepackage{graphicx}
\usepackage{txfonts}
\usepackage{dsfont}
\usepackage{mathabx} 
\usepackage{tikz-cd}
\usepackage[titletoc]{appendix}
\usepackage{shuffle}
\usepackage{geometry}
\usepackage{enumitem}

\theoremstyle{definition}
\newtheorem{Defn}{Definition}[section]

\theoremstyle{plain}
\newtheorem{Lemma}[Defn]{Lemma}
\newtheorem{prop}[Defn]{Proposition}
\newtheorem{theorem}[Defn]{Theorem}
\newtheorem{Corollary}[Defn]{Corollary}

\newtheorem{THM}{Theorem}
\newtheorem{COR}[THM]{Corollary}

\theoremstyle{remark}
\newtheorem{remark}[Defn]{Remark}
\newtheorem{Example}[Defn]{Example}

\newif\ifshowcomments
\showcommentstrue

\title{On analytic exponential functors on free groups}
\author{Minkyu Kim \& Christine Vespa}
\date{}

\begin{document}

\maketitle
\begin{abstract}
This paper concerns exponential contravariant functors on free groups.
We obtain an equivalence of categories between analytic, exponential contravariant functors on free groups and conilpotent cocommutative Hopf algebras. 
This result explains how equivalences of categories obtained previously by Pirashvili and by Powell interact.
Moreover, we obtain an equivalence between the categories of outer, exponential contravariant functors on free groups and bicommutative Hopf algebras. 
We also go further by introducing a subclass of analytic, contravariant functors on free groups, called {\it primitive} functors; and prove an equivalence between primitive, exponential contravariant functors and primitive cocommutative Hopf algebras.
\end{abstract}

\tableofcontents

%Functors on the category of finitely generated free groups $\mathsf{gr}$ are closely related to the theory of linear representations of the groups $Aut(\mathsf{F}_n)$ where $\mathsf{F}_n$ is the free group generated by $n$ elements. 
%%Let $M$ be a functor from the category $\mathsf{gr}$ to the category $\mathsf{Mod}_\mathds{k}$ of $\mathds{k}$-modules (for $\mathds{k}$ a commutative ring).
%%By evaluation on $F_n$, we can associate to $M$ a family of representations of the groups $Aut(F_n)$. These representations satisfy relations of compatibility given by the functoriality. 
%So the study of functors from the category $\mathsf{gr}$ to the category of $\mathds{k}$-modules should be viewed as a way to study global properties of compatible families of representations of the groups $Aut(F_n)$.

\section{Introduction}

Functors on the category $\mathsf{gr}$ of finitely generated free groups and group homomorphisms  appear naturally in different contexts. For example, they give rise to natural examples of coefficients for the stable homology of the family of automorphisms groups of free groups $\mathsf{Aut}(\mathsf{F_n})$ (see \cite{DV, D, KV}), where $\mathsf{F}_n$ is the free group generated by $n$ elements. By \cite{PV}, Hochschild-Pirashvili homology for a wedge of circle gives rise to interesting functors on the category $\mathsf{gr}$. Since the linearization of $\mathsf{gr}^{\mathsf{op}}$ is a subcategory of the linear category of Jacobi diagrams in handlebodies introduced by Habiro and Massuyeau in \cite{HM}, a linear functor on this linear category gives rise, by restriction, to a functor on $\mathsf{gr}^{\mathsf{op}}$.  This is exploited in \cite{Katada1, Katada2, V-Jacobi} in order to study functors associated to (beaded) open Jacobi diagrams.

To a functor from the category $\mathsf{gr}$  to the category of $\mathds{k}$-modules (for $\mathds{k}$ a commutative ring), we can associate a family of  linear representations of the groups $\mathsf{Aut}(\mathsf{F}_n)$, for $n\in \mathbb{N}$. 
%%Let $M$ be a functor from the category $\mathsf{gr}$ to the category $\mathsf{Mod}_\mathds{k}$ of $\mathds{k}$-modules (for $\mathds{k}$ a commutative ring).
%%By evaluation on $F_n$, we can associate to $M$ a family of representations of the groups $Aut(F_n)$. These representations satisfy relations of compatibility given by the functoriality. 
Hence, a functor from the category $\mathsf{gr}$ to the category of $\mathds{k}$-modules is called a $\mathsf{gr}$-module and gives a compatible family of representations of the groups $\mathsf{Aut}(\mathsf{F_n})$, where the compatibilities are given by the functoriality.

In the applications cited above, the functors on the category $\mathsf{gr}$ (resp. $\mathsf{gr}^{\mathsf{op}}$) satisfy further properties. One of the advantages of considering functors instead of families of linear representations lies in the exploitation of these properties. 

This paper deals with the functors on the category $\mathsf{gr}^{\mathsf{op}}$ satisfying the exponential property  (or, equivalently, the fact that the functor is symmetric monoidal) and the analyticity property obtained from the polynomial property in the sense of Eilenberg-Mac Lane.

% and the outer properties corresponding to the fact that the action of the inner automorphisms is trivial on the evaluations of the functor.

On the one hand, by the work of Pirashvili \cite{Pirashvili} (see also \cite[Remark 2.1, Theorem 2.2]{MR3765469} and \cite{Hab}), the category of exponential functors on $\mathsf{gr}^{\mathsf{op}}$ is equivalent to the category of cocommutative Hopf algebras. On the other hand, Powell considers, in \cite{powell2021analytic}, the category of analytic functors from $\mathsf{gr}^{\mathsf{op}}$ to  the category of $\mathds{k}$-vector spaces (for $\mathds{k}$ a field of characteristic $0$). He proves an equivalence between this category and the category $\mathcal{F}_{\mathfrak{Lie}}$ of linear functors on the linear PROP associated to the Lie operad $\mathfrak{Lie}$. 

This paper answers the following question: how do these two equivalences of categories interact or more precisely can we describe the category of analytic exponential functors on 
$\mathsf{gr}^{\mathsf{op}}$?

Let $\mathcal{F}^\mathsf{exp}_\omega ( \mathsf{gr^{\mathsf{op}}} ; \mathds{k} )$ be the category of exponential and analytic functors on $\mathsf{gr^{\mathsf{op}}}$ and $\mathsf{Hopf}^\mathsf{cc,conil}_\mathds{k}$ the category of cocommutative and conilpotent Hopf algebras.
We obtain the following result, which is valid for any commutative ring $\mathds{k}$:

\begin{THM}
\label{202211301648}
(Theorem \ref{202207212109})
The evaluation on ${\mathsf{F}_1}$ induces an equivalence:
$$
\mathcal{F}^\mathsf{exp}_\omega ( \mathsf{gr^{\mathsf{op}}} ; \mathds{k} ) \simeq \mathsf{Hopf}^\mathsf{cc,conil}_\mathds{k} .
$$
\end{THM}

The proof of this result is independent of Powell's result and is based on Theorem \ref{202207141636} implying that the polynomial filtration of an exponential functor on $\mathsf{gr^{\mathsf{op}}}$ and the coradical filtration coincide. Theorem \ref{202207141636}  extends a result obtained by Touz\'e \cite[Theorem 11.10]{touze} concerning functors on an additive category. Recall that the category $\mathsf{gr^{\mathsf{op}}}$ is not additive. 

For $\mathds{k}$ a field of characteristic zero, using the equivalence of categories:
$$P : \mathsf{Hopf}^\mathsf{cc,conil}_\mathds{k} \stackrel{\simeq}{\longrightarrow} \mathsf{Lie}_\mathds{k}$$
given by the primitive elements of Hopf algebras, we obtain:

\begin{COR}(Corollary \ref{202211071547})
For $\mathds{k}$ a field of characteristic zero, the composition of the evaluation on ${\mathsf{F}_1}$ with the functor $P$ induces an equivalence of categories:
$$\mathcal{F}^\mathsf{exp}_\omega ( \mathsf{gr^{\mathsf{op}}} ; \mathds{k} ) \simeq \mathsf{Lie}_\mathds{k}.$$
\end{COR}

Another important property which could be satisfied by a functor on the category $\mathsf{gr}^{\mathsf{op}}$ is the \textit{outer} property introduced in \cite{PV}. A functor $F: \mathsf{gr^{\mathsf{op}}} \to \mathsf{Mod}_\mathds{k}$ is an outer functor if the inner automorphisms act trivially on $F(\mathsf{F}_n)$ for all $n\in \mathbb{N}$. Outer functors appear naturally in \cite{PV} and \cite{Katada1}. For $\mathds{k}$ a field of characteristic zero, Powell describes in \cite{P-21} outer analytic functors. More precisely he gives an equivalence of categories between outer analytic functors on $\mathsf{gr^{\mathsf{op}}}$ and a subcategory of $\mathcal{F}_{\mathfrak{Lie}}$ denoted by $\mathcal{F}_{\mathfrak{Lie}}^{\mu}$. 

Let $\mathcal{F}^\mathsf{exp}_\mathsf{out} ( \mathsf{gr}^{\mathsf{op}} ; \mathds{k} )$ be the full subcategory of $\mathcal{F}^\mathsf{exp} ( \mathsf{gr}^{\mathsf{op}} ; \mathds{k} )$ consisting of outer $\mathsf{gr}^{\mathsf{op}}$-modules and $\mathsf{Hopf}^\mathsf{bc}_{\mathds{k}}$ be the category of bicommutative Hopf algebras.
We obtain the following result:

\begin{THM}(Theorem \ref{202207311733})
\label{202211301647}
The evaluation on ${\mathsf{F}_1}$ induces an equivalence of categories:
$$
\mathcal{F}^\mathsf{exp}_\mathsf{out} ( \mathsf{gr}^{\mathsf{op}} ; \mathds{k} ) 
\simeq
\mathsf{Hopf}^\mathsf{bc}_{\mathds{k}}.
$$
\end{THM}

This result is deduced from the fact (see Theorem \ref{202207311733}) that the composition by the abelianization functor $\mathfrak{a} : \mathsf{gr} \to \mathsf{ab}$ (where $\mathsf{ab}$ is the category consisting of the free abelian groups $\mathds{Z}^n$ for $n \in \mathds{N}$ and group homomorphisms)  induces an equivalence of categories:
$$
\mathcal{F}^\mathsf{exp} ( \mathsf{ab}^{\mathsf{op}} ; \mathds{k} )
\simeq
\mathcal{F}^\mathsf{exp}_\mathsf{out} ( \mathsf{gr}^{\mathsf{op}} ; \mathds{k} ).
$$

Let $\mathcal{F}^\mathsf{exp}_{\omega, \mathsf{out}} ( \mathsf{gr^{\mathsf{op}}} ; \mathds{k} )$ be the full subcategory of $\mathcal{F}^\mathsf{exp} ( \mathsf{gr}^{\mathsf{op}} ; \mathds{k} )$ consisting of analytic and outer $\mathsf{gr}^{\mathsf{op}}$-modules. For $\mathds{k}$ a field of characteristic zero, we obtain:

\begin{COR}(Corollary \ref{202211071547})
For $\mathds{k}$ a field of characteristic zero, the composition of the evaluation on ${\mathsf{F}_1}$ with the functor $P$ induces an equivalence of categories:
$$\mathcal{F}^\mathsf{exp}_{\omega, \mathsf{out}} ( \mathsf{gr^{\mathsf{op}}} ; \mathds{k} )  \simeq \mathsf{Lie}^\mathsf{ab}_\mathds{k}.$$
\end{COR}

For $\mathds{k}$ a field of characteristic zero, our results and previous results obtained by Pirashvili and Powell can be summarized in the following commutative diagram:

$$
\begin{tikzcd}[column sep=tiny, row sep=small]
& & \mathcal{F}^\mathsf{exp}_\omega ( \mathsf{gr^{\mathsf{op}}} )\ar[ddd, "\simeq"', near start] \ar[dll, hookrightarrow] \ar[rr] & & \mathcal{F}_\omega ( \mathsf{gr^{\mathsf{op}}}  ) \ar[ddd, "\simeq"']  & & \\
\mathcal{F}^\mathsf{exp} ( \mathsf{gr^{\mathsf{op}}})\ar[ddd, "\simeq"'] & & & \mathcal{F}^\mathsf{exp}_{\omega,\mathsf{out}} ( \mathsf{gr^{\mathsf{op}}} ) \ar[ul, hookrightarrow] \ar[rrr, crossing over] \ar[dll, crossing over, hookrightarrow] \ar[ddd, "\simeq"'] & & & \mathcal{F}_{\omega,\mathsf{out}} ( \mathsf{gr^{\mathsf{op}}})  \ar[ddd, "\simeq"] \ar[ull, hookrightarrow] \ar[dl, hookrightarrow] \\
& \mathcal{F}^\mathsf{exp}_\mathsf{out} ( \mathsf{gr^{\mathsf{op}}} ) \ar[ul, hookrightarrow] \ar[rrrr, crossing over] & & & & \mathcal{F}_\mathsf{out} ( \mathsf{gr^{\mathsf{op}}} )& \\
& & \mathsf{Lie}_\mathds{k} \ar[dll, hookrightarrow] \ar[rr] & & \mathcal{F}_{\mathfrak{Lie}} & & \\
\mathsf{Hopf}^\mathsf{cc}_\mathds{k} & & & \mathsf{Lie}^\mathsf{ab}_\mathds{k} \ar[ul, hookrightarrow] \ar[dll, hookrightarrow]  \ar[rrr] & & & \mathcal{F}^{\mu}_{\mathfrak{Lie}} \ar[ull, hookrightarrow, shift right] \\
& \mathsf{Hopf}^\mathsf{bc}_\mathds{k} \ar[ul, hookrightarrow] \ar[from=uuu, crossing over, "\simeq"', near start] & & & & &
\end{tikzcd}
$$
where we omit the field $\mathds{k}$ in the notation of the categories of functors. 
For example, we write $\mathcal{F}^\mathsf{exp} ( \mathsf{gr^{\mathsf{op}}} )$ instead of $\mathcal{F}^\mathsf{exp} ( \mathsf{gr^{\mathsf{op}}} ; \mathds{k})$.

For $\mathds{k}$ a field of positive characteristic, in Corollary \ref{202211301650}, we obtain  similar statements to Theorems \ref{202211301648}, \ref{202211301647} by replacing the polynomial filtration with the {\it primitive} filtration introduced in section \ref{202211081022}. 

The above results concern $\mathsf{gr}^{\mathsf{op}}$-modules. These are consequences of more general results obtained in this paper and concerning functors on a small category $\mathcal{C}$ with finite products and a null object. 

In part \ref{202312081751}, we study the polynomial filtration of functors on $\mathcal{C}$.
In section \ref{202207311622}, we recall the polynomial filtration of functors.
Section \ref{202207211508} concerns the coradical filtration of coaugmented coalgebras. One of the main ingredients of this paper is the description of the coradical filtration based on idempotents and mimicking the notion of cross-effect for functors (see Proposition \ref{202208302218}). We also give examples of computations of the coradical filtration using the description given in Proposition \ref{202208302218}.
The proof of this result is based on  Proposition \ref{202208302218}.
In section \ref{202211251159} (see Theorem \ref{202207141636})  we prove that the polynomial filtration of an exponential functor $F$ on $\mathcal{C}$, evaluated on $X$, coincides with the coradical filtration of the  bialgebra $F(X)$ for any $X \in \mathcal{C}$. 

In part \ref{202312081754}, we give applications of part \ref{202312081751} to $\mathcal{C} = \mathsf{gr}^{\mathsf{op}}$.
In section \ref{202212161644}, we give an overview of the operad $\mathsf{Cat}_{\mathfrak{Ass}^u}$ of unital associative algebras and related constructions.
The goal of this section is to recall Powell's assertion \cite{powell2021analytic} that $\mathsf{Cat}_{\mathfrak{Ass}^u}$ induces a bilinear functor ${}_\Delta \mathsf{Cat}_{\mathfrak{Ass}^u} : \mathsf{Cat}_{\mathfrak{Lie}}^{\mathsf{op}} \times \mathds{k} \mathsf{gr}^{\mathsf{op}} \to \mathsf{Mod}_{\mathds{k}}$.
In section \ref{202312081800}, we introduce a left ideal $\mathcal{I}$ of the linear category $\mathds{k}\mathsf{gr}^{\mathsf{op}}$ which is used to introduce the primitive filtration of $\mathsf{gr}^{\mathsf{op}}$-modules.
In section \ref{202211081004}, we apply the results of part \ref{202312081751} to study the polynomial filtration of exponential $\mathsf{gr}^{\mathsf{op}}$-modules.
In section \ref{202211081022}, we introduce the primitive filtration of $\mathsf{gr}^{\mathsf{op}}$-modules.
In section \ref{202312081805}, we study primitive exponential $\mathsf{gr}^{\mathsf{op}}$-modules.

\section*{Acknowledgements}

We thank Geoffrey Powell for valuable comments on previous version of this paper.
We thank Kawazumi Sensei for introducing the authors to each other and giving the opportunity to start this project.
The first author appreciates Leading Graduate Course for Frontiers of Mathematical Sciences and Physics; and a KIAS Individual Grant MG093701 at Korea Institute for Advanced Study.
The second author would like to thank the JSPS (The Japan Society for the Promotion of Science) for its support through the invitational fellowship for research in Japan (long-term) ID L21510 and the university of Tokyo for its hospitality.

\section{Notation} \label{202208181550}

We denote by $\mathds{k}$ a commutative unital ring unless otherwise specified.

We denote by $\mathds{N}$ the set of natural numbers $0, 1, 2, \cdots$.

We denote by $\mathcal{C}$ a small category with finite products $\times$ and a null object $\ast$.
We use the notation $\varepsilon_X : X \to \ast$ and $\eta_X : \ast \to X$ to denote the unique morphisms and we use the notation $p_X := \eta_X \circ \varepsilon_X: X \to X$. We can view $\mathcal{C}$  as a symmetric monoidal category of the form $(\mathcal{C}, \times, \ast)$.
The example of main interest here is the category $\mathcal{C}_{H}$ generated by a coaugmented cocommutative coalgebra $H$ (see section \ref{202211251159}).
The product is given by the tensor product.
Another example arises from the opposite category of the category of finitely generated free groups and group homomorphisms.
The product is given by the free product $\ast$ of groups and the trivial group is the null object.
We choose a skeleton of the category which we denote by $\mathsf{gr}^{\mathsf{op}}$.
Its objects consist of natural numbers $n$ which we identify with $\mathsf{F}_n$ the free group generated by $n$ elements $x_1, x_2 , \cdots , x_n$ (with the unit $e$).
Using this notation, the free product $\ast$ is given by the sum of natural numbers.
A morphism from $n$ to $m$ is given by a group homomorphism from $\mathsf{F}_m$ to $\mathsf{F}_n$.

We introduce the notation of morphisms in the {\it opposite} category $\mathsf{gr}^{\mathsf{op}}$.
For a group homomorphism $\rho : \mathsf{F}_m \to \mathsf{F}_n$, we denote by $[\rho (x_1) | \cdots | \rho (x_m) ]_n : n \to m$ the corresponding morphism in $\mathsf{gr}^{\mathsf{op}}$.
In particular, we set $\Delta {:=} [x_1|x_1]_1$ for the fold map, $\nabla {:=} [x_1x_2]_2$ for the coproduct and $\gamma {:=} [x_1^{-1}]_1$ for the antipode respectively.
We denote by $\eta : 0 \to 1$, $\epsilon: 1 \to 0$ the trivial morphisms.

We denote by $[ X_i ]^{n}_{i=1}$ a finite sequence of objects of $\mathcal{C}$.

For a category $\mathcal{C}$ and objects $X,Y \in \mathcal{C}$, we denote by $\mathcal{C} (X,Y)$ the set of morphisms from $X$ to $Y$.

For functors $F,G$ from $\mathcal{C}$ to $\mathsf{Mod}_{\mathds{k}}$, we denote by $\mathrm{Nat} ( F, G)$ the $\mathds{k}$-module consisting of natural transformations from $F$ to $G$.

We denote by $\Sigma$ the category induced by the symmetric groups.
In other words, its objects are given by $n \in \mathds{N}$ and the morphism set $\Sigma (n,m)$  is given by  $\Sigma_n$ the $n$-th symmetric group if $n=m$; and is empty if $n \neq m$.

We denote by $\underline{n} {:=} \{1, \ldots, n \}$ for $n > 0$, and $\underline{0} {:=} \emptyset$.

\part{The general case}
\label{202312081751}

\section{Polynomial filtration of $\mathcal{C}$-modules}
\label{202207311622}

In this section, we recall the notions of cross-effects and polynomial filtration of $\mathcal{C}$-modules. These notions were introduced by Eilenberg and Mac Lane in \cite{EML} for functors between categories of modules and extended to functors on a monoidal category where the unit is a null object in \cite{HPV}.
The aim of this section is to give an explicit way to construct the polynomial approximations.

\subsection{Cross-effects of $\mathcal{C}$-modules}
\label{202208131840}

Let $\mathds{k}$ be a commutative unital ring.
A {\it $\mathcal{C}$-module} is a functor from $\mathcal{C}$ to the category $\mathsf{Mod}_\mathds{k}$ of $\mathds{k}$-modules and homomorphisms.
We denote by $\mathcal{F} ( \mathcal{C} ; \mathds{k})$ the category of $\mathcal{C}$-modules and natural transformations.

Let $\mathcal{F}^{(n)}_\ast ( \mathcal{C} ; \mathds{k} )$ be the full subcategory of $\mathcal{F} ( \mathcal{C}^{\times n} ; \mathds{k} )$ consisting of $F \in \mathcal{F} ( \mathcal{C}^{\times n} ; \mathds{k} )$ such that $F ( X _1 , \cdots , X_n ) \cong 0$ if $X_i = \ast$ for some $i$. We denote by $\Delta_n: \mathcal{C} \to \mathcal{C}^{\times n}$ the diagonal functor.

\begin{Defn}
For $n \geq 1$, the {\it $n$-th cross-effect functor} is the left adjoint to the precomposition with the diagonal functor $\Delta_n^\ast : \mathcal{F}^{(n)}_\ast ( \mathcal{C} ; \mathds{k} ) \to \mathcal{F} ( \mathcal{C} ; \mathds{k} )$.
\end{Defn}

In order to give an explicit description of the $n$-th cross-effect functor, we introduce the following notation.
Let $m,n$ be natural numbers such that $1 \leq m \leq n$.
For $[ X_i ]^{n}_{i=1}$ a finite sequence of objects of $\mathcal{C}$ and $p_{X_m} = \eta_{X_m} \circ \varepsilon_{X_m}: X_m \to X_m$, we define the morphism:
$$\tau_m ( [ X_i ]^{n}_{i=1} ) {:=} \left( \prod^{m-1}_{i=1} \mathrm{id}_{X_i} \right) \times p_{X_m}  \times \left( \prod^{n}_{i=m+1} \mathrm{id}_{X_i} \right): \prod^{n}_{i=1} {X_i} \to  \prod^{n}_{i=1} {X_i}.$$

The morphisms $\tau_m( [ X_i ]^{n}_{i=1} )$ are pairwise commuting idempotents. 
Hence, for  $F \in \mathcal{F} ( \mathcal{C} ; \mathds{k} )$ and $X = \prod^{n}_{i=1} X_i$, the morphisms $e_m: F(X) \to F(X)$ defined by:
$$e_m= \mathrm{id}_{F(X)} - F(\tau_m ([ X_i ]^{n}_{i=1}))$$
are pairwise commuting idempotents. 

We define an endomorphism of $F( X )$ by:
$$
\chi_{F} ( [ X_i ]^{n}_{i=1} ) {:=} ( e_1 \circ  e_2 \circ \cdots \circ e_n ) : F(X) \to F(X) .
$$
For $n \geq 1$, we define $\mathrm{cr}^{n} (F) \in \mathcal{F}^{(n)}_\ast (  \mathcal{C}  ; \mathds{k} )$ by

\begin{equation}
\label{202211251145}
( \mathrm{cr}^{n}( F )) ( 
[ X_i ]^{n}_{i=1} )
{=}
\mathrm{Im} \left(  \chi_{F} ( [ X_i ]^{n}_{i=1} ) \right) .
\end{equation}

This induces a functor 
$$\mathrm{cr}^{n}: \mathcal{F} ( \mathcal{C} ; \mathds{k} ) \to \mathcal{F}^{(n)}_\ast (  \mathcal{C}  ; \mathds{k} ).$$

\begin{prop}
\label{202207211151}
The functor $\mathrm{cr}^{n}: \mathcal{F} ( \mathcal{C} ; \mathds{k} ) \to \mathcal{F}^{(n)}_\ast (  \mathcal{C}  ; \mathds{k} )$ is the $n$-th cross-effect functor.
\end{prop}

\begin{proof}
Let $F \in \mathcal{F} ( \mathcal{C} ; \mathds{k} )$ and $G \in \mathcal{F}^{(n)}_\ast ( \mathcal{C} ; \mathds{k} )$.
We give a direct construction of a natural isomorphism $\Phi : \mathrm{Nat} ( F , \Delta_n^\ast (G)) \to \mathrm{Nat} ( \mathrm{cr}^{n} (F) , G )$.
Let $\beta \in \mathrm{Nat} ( F , \Delta_n^\ast (G))$.
For an $n$-tuple $[ X_i ]^{n}_{i=1}$ in $\mathcal{C}$, let $X = \prod^{n}_{i=1} X_i$ and  $\pi_i : X \to X_i$ be the $i$-th projection.
We define $\Phi (\beta) :  ( \mathrm{cr}^{n} (F) ) \left( [X_i ]^{n}_{i=1} \right) \to G \left( [X_i  ]^{n}_{i=1} \right)$ to be the composition:
$$
( \mathrm{cr}^{n} (F) )\left( [ X_i  ]^{n}_{i=1} \right)
\hookrightarrow
F( X )
\stackrel{\beta_{X}}{\to}
(\Delta_n^\ast (G) ) ( X  ) =
G \left( [ X  ]^{n}_{i=1} \right)
\stackrel{G( [\pi_i]^{n}_{i=1})}{\longrightarrow}
G \left( [ X_i ]^{n}_{i=1} \right) .
$$
The inverse of $\Phi$ is given by $\Psi : \mathrm{Nat} ( \mathrm{cr}^{n} (F) , G ) \to \mathrm{Nat} ( F , \Delta_n^\ast (G))$ defined as follows.
Let $\beta^\prime \in \mathrm{Nat} ( \mathrm{cr}^{n} (F) , G )$.
For $Z \in \mathcal{C}$ with the $n$-times iterated diagonal map $\Delta^{(n)}_Z : Z \to Z^{\times n}$, we define $\Psi (\beta^\prime) :  F ( Z ) \to (\Delta_n^\ast (G))(Z)$ as the composition:
$$
F(Z)
\stackrel{F(\Delta^{(n)}_Z)}{\longrightarrow}
F(Z^{\times n})
\stackrel{\chi_{F} ( [ Z ]^{n}_{i=1} )}{\longrightarrow}
( \mathrm{cr}^{n}(F) ) ( [ Z ]^{n}_{i=1} )
\stackrel{\beta^\prime_{[ Z  ]^{n}_{i=1}}}{\longrightarrow}
(\Delta_n^\ast (G)) ( Z ) .
$$
\end{proof}

\begin{prop}
\label{202207281651}
Let $F \in \mathcal{F} ( \mathcal{C} ; \mathds{k})$.
Then we have a natural isomorphism:
$$
F ( \prod^{n}_{i=1} X_i ) \cong F(\ast ) \oplus \left( \bigoplus^{n}_{k=1} \bigoplus_{1 \leq i_1 < \cdots < i_k \leq n} (\mathrm{cr}^k (F))(X_{i_1} , \cdots , X_{i_k} ) \right) .
$$
\end{prop}
\begin{proof}
The previous idempotents $e_i$'s are simultaneously diagonalizable since they commute with each other.
For $1 \leq i_1 < \cdots < i_k \leq n$, the simultaneous eigenspace consisting of $v \in F(X)$ such that $e_i v = v$ only for $i = i_1, \cdots i_k$ coincides with $(\mathrm{cr}^k (F))(X_{i_1} , \cdots , X_{i_k} )$.
The eigenspace consisting of $v \in F(X)$ such that $e_i v = 0$ for any $i$ coincides with $F(\ast)$.
This observation leads to the natural isomorphism.
\end{proof}

\subsection{Polynomial and analytic functors}
\label{202211251500}

In this section, we give an overview of polynomial functors by using cross-effects.
We define the polynomial functor approximation as being the right adjoint of the inclusion.

\begin{Defn}
\label{Korea202306121414}
A functor $F\in \mathcal{F} ( \mathcal{C} ; \mathds{k})$ {\it has degree less than or equal to $n$}, written $\mathrm{deg} (F) \leq n$, if $\mathrm{cr}^{n+1} (F) \cong 0$; and $F\in \mathcal{F} ( \mathcal{C} ; \mathds{k})$ {\it has degree $n$} if $\mathrm{deg} (F) \leq n$ and $\mathrm{deg} (F) \not\leq (n-1)$.
A $\mathcal{C}$-module $F$ is a {\it polynomial functor} if there exists $n \in \mathbb{N}$ such that $\mathrm{deg} (F) \leq n$.
Denote by $\mathcal{F}_{\leq n} ( \mathcal{C} ; \mathds{k})$ the full subcategory of $\mathcal{F} ( \mathcal{C} ; \mathds{k})$ consisting of functors with degree less than or equal to $n$.
\end{Defn}

Note that, if $\mathrm{deg} ( F) \leq n$, then $\mathrm{deg} (F) \leq (n+1)$.

\begin{Defn}
\label{202207202245}
Let $i_n :  \mathcal{F}_{\leq n} ( \mathcal{C} ; \mathds{k} ) \to \mathcal{F} ( \mathcal{C} ; \mathds{k} )$ be the inclusion.
{\it The $n$-th polynomial approximation} $P_n: \mathcal{F} ( \mathcal{C} ; \mathds{k} ) \to \mathcal{F}_{\leq n} ( \mathcal{C} ; \mathds{k} ) $ is the right adjoint to the inclusion $i_n$.
\end{Defn}

\begin{prop}
\label{202207211024}
\begin{itemize}
\item
For $F\in \mathcal{F} ( \mathcal{C} ; \mathds{k})$, the counits of the adjunctions induce the following commutative diagram:
$$
\begin{tikzcd}
& & F  & \\
 \cdots \ar[r] \ar[urr, shift left] & i_n ( P_n (F) ) \ar[r] \ar[ur] & i_{n+1} ( P_{n+1} (F) ) \ar[r] \ar[u] & \cdots \ar[ul]
\end{tikzcd}
$$
\item
For $G \in \mathcal{F}_{\leq n} (\mathcal{C} ; \mathds{k})$, the unit $G \to P_n ( i_n ( G ))$ is an isomorphism.
\end{itemize}
\end{prop}
\begin{proof}
Let $j : \mathcal{F}_{\leq n} (\mathcal{C} ;\mathds{k}) \to \mathcal{F}_{\leq (n+1)} (\mathcal{C} ;\mathds{k})$ be the inclusion.
The unit $i_n ( P_n (F)) \to F$ induces a natural transformation $j ( P_n (F)) \to P_{n+1} (F)$ by the adjunction with respect to $i_{n+1}$.
We apply $i_{n+1}$ to this to obtain $i_n (P_n(F)) \to i_{n+1} ( P_{n+1} (F))$.
The commutativity in the first part follows from adjunction properties.

For the second part, it suffices to verify that the postcomposition $\mathrm{Nat} ( G^\prime , G) \to \mathrm{Nat} (G^\prime ,P_n (i_n (G)))$ gives a bijection.
Note that $i_n : \mathcal{F}_{\leq n} ( \mathcal{C} ; \mathds{k}) \to \mathcal{F} ( \mathcal{C} ; \mathds{k})$ is a fully faithful functor by definition.
Equivalently, we have a bijection $\mathrm{Nat} ( G^\prime , G) \to \mathrm{Nat} ( i_n (G^\prime) , i_n (G) )$.
The result follows from the adjunction $\mathrm{Nat} (G^\prime ,P_n (i_n (G))) \cong \mathrm{Nat} ( i_n (G^\prime) ,i_n (G))$.
\end{proof}

\begin{remark}
The existence of a right adjoint to the inclusion $i_n$ is proved in Section \ref{202207211510}.
It turns out that the natural transformations in the first part of Proposition \ref{202207211024} are inclusions.
In particular, the sequence of $i_n (P_n (F))$'s gives a filtration of $F$ called the {\it polynomial filtration} of $F$.
\end{remark}

\begin{Defn}
A $\mathcal{C}$-module $F$ is {\it analytic} if $F$ is the colimit of its polynomial filtration.
Denote by $\mathcal{F}_\omega ( \mathcal{C} ; \mathds{k} )$ the full subcategory of $\mathcal{F} ( \mathcal{C} ; \mathds{k})$ consisting of analytic $\mathcal{C}$-module.
\end{Defn}

\subsection{Construction of the polynomial filtration}
\label{202207211510}

In this section, we show that there exists a polynomial approximation as in Definition \ref{202207202245}.
We give an explicit description of the functor $P_n (F)$ and we prove that it corresponds to the right adjoint of the inclusion.

\begin{Defn}
\label{202207202244}
For $X$ an object of $\mathcal{C}$, we define $( P_n (F) ) (X )$ to be the kernel of the composite,
$$F(X) \stackrel{F(\Delta^{(n+1)}_X)}{\to} F( X^{\times (n+1)} ) \stackrel{\chi}{\longrightarrow} F( X^{\times (n+1)} )$$
where $\Delta^{(n+1)}_X$ is the $(n+1)$-times iterated diagonal map of $X$ and $\chi= \chi_{F} ( [ X ]^{n+1}_{i=1} )$.
Denote by $P_n (F) \in \mathcal{F} ( \mathcal{C} ; \mathds{k} )$ the induced functor.
\end{Defn}

\begin{Lemma}
\label{202207141317}
For $F \in \mathcal{F} ( \mathcal{C} ; \mathds{k} )$, we have $\mathrm{cr}^{n+1} ( P_n (F) ) \cong 0$.
In particular, $\mathrm{deg} ( P_n (F) ) \leq n$.
\end{Lemma}
\begin{proof}
We prove that $\mathrm{cr}^{n+1} ( P_n (F)) ( [ X_i ]^{n+1}_{i=1} ) = 0$ for a sequence $[ X_i ]^{n+1}_{i=1}$ of objects of $\mathcal{C}$.
It suffices to show that $\chi_{P_n(F)} ( [ X_i ]^{n+1}_{i=1} ) = 0$.
Let $X = \prod^{n+1}_{i=1} X_i$.
Let $\pi : X^{\times (n+1)} \to \prod^{n+1}_{i=1} X_i$ be $\prod^{n+1}_{i=1}  \pi_i$ where $\pi_i = \prod_{j \neq i} \varepsilon_j : X = X_1 \times \cdots \times X_{n+1} \to X_i$.
It is easy to check that the diagram below commutes for any $j$ where $\tau_j$ is introduced in section \ref{202208131840}:
$$
\begin{tikzcd}[row sep=scriptsize]
X^{\times (n+1)} \ar[rr, "\tau_j ({[ X ]}^{n+1}_{i=1})"] \ar[d, "\pi"] & & X^{\times (n+1)} \ar[d, "\pi"] \\
\prod^{n+1}_{i=1} X_i \ar[rr, "\tau_j ( {[ X_i ]}^{n+1}_{i=1} )"] & & \prod^{n+1}_{i=1} X_i
\end{tikzcd}
$$
By combining the above commutative diagrams iteratively, we obtain 
$$F( \pi ) \circ \chi_{F} ( [ X ]^{n+1}_{i=1}) = \chi_{F} ( [ X_i ]^{n+1}_{i=1}) \circ F (\pi) . $$
Moreover, $\pi \circ \Delta^{(n+1)}_X = \mathrm{id}_X$ implies $F( \pi ) \circ \chi_{F} ( [ X ]^{n+1}_{i=1}) \circ F( \Delta^{(n+1)}_X ) = \chi_{F} ( [ X_i ]^{n+1}_{i=1})$.
Now let $x \in ( P_n (F) ) ( X )$.
Then the above result implies 
\begin{align*}
(\chi_{P_n(F)} ( [ X_i ]^{n+1}_{i=1} ) ) (x) &= (\chi_{F} ( [ X_i ]^{n+1}_{i=1} ) ) (x) , \\
&= \left( F( \pi ) \circ \left( \chi_{F} ( [ X ]^{n+1}_{i=1}) \circ F( \Delta^{(n+1)}_X ) \right) \right) (x) , \\
&= 0  ~~~ ( \mathrm{since} ~~~ x \in ( P_n (F) ) ( X ) ) .
\end{align*}
\end{proof}

\begin{prop}
\label{202207202250}
The assignment of $P_n (F)$ to $F \in \mathcal{F} ( \mathcal{C} ; \mathds{k})$ gives a right adjoint to the inclusion $i_n$.
In particular, $P_n$ is the $n$-th polynomial approximation.
\end{prop}
\begin{proof}
Note that Lemma \ref{202207141317} implies that $P_n (F) \in \mathcal{F}_{\leq n} ( \mathcal{C} ; \mathds{k})$, so that we obtain a functor $P_n : \mathcal{F} ( \mathcal{C} ; \mathds{k}) \to \mathcal{F}_{\leq n} ( \mathcal{C} ; \mathds{k})$.
The inclusion $\epsilon_F : i_n ( P_n (F) ) \to F$ induces a natural transformation $\mathrm{Nat} ( G , P_n (F)) \to \mathrm{Nat} ( i_n (G) , F); ~ \phi \mapsto \epsilon_F \circ i_n (\phi)$ for $G \in \mathcal{F}_{\leq n} (  \mathcal{C} ; \mathds{k})$ and $F \in \mathcal{F} ( \mathcal{C} ; \mathds{k} )$.
We now prove that this gives a natural isomorphism.
It suffices to show that a natural transformation $\psi : i_n (G) \to F$ factors uniquely through $\epsilon_F$.
Let $X$ be an object of $\mathcal{C}$ and put $Y = X^{\times (n+1)}$.
Then the naturality of $\psi$ implies $( \chi_{F} ( [ X ]^{n+1}_{i=1}) \circ F(\Delta^{(n+1)}_X ) ) \circ \psi_{X} = \psi_{Y} \circ ( \chi_{G} ([ X ]^{n+1}_{i=1}) \circ G (\Delta^{(n+1)}_X ) )$, with the right hand side zero due to the assumption $\mathrm{cr}^{n+1} (G ) \cong 0$.
Hence, $( \chi_{F} ( [ X ]^{n+1}_{i=1}) \circ F(\Delta^{(n+1)}_X ) ) \circ \psi_{X} = 0$, and thus $\psi_X$ uniquely factors through the inclusion $\epsilon_F$.

\end{proof}

\begin{remark}
\label{202207201449}
Dually, one can consider a small category $\mathcal{D}$ with finite {\it coproducts} and a null object.
Let $j_n : \mathcal{F}_{\leq n} ( \mathcal{D} ; \mathds{k} ) \to \mathcal{F} ( \mathcal{D} ; \mathds{k} )$ be the inclusion.
In a similar fashion, one can construct a cofiltration of a $\mathcal{D}$-module $F$ by using a {\it left} adjoint $P^n$ to the inclusion $j_n$:
$$
\begin{tikzcd}
& & F \ar[d] \ar[dl] \ar[dll, shift right] \ar[dr] & \\
 \cdots \ar[r] & j_{n+1} ( P^{n+1} (F) ) \ar[r]  & j_{n} ( P^{n} (F) ) \ar[r] & \cdots
\end{tikzcd}
$$
\end{remark}

%%%%%%%%%%%%%%%%%%%%%%%%%%%%%%%%%%%

\section{Coradical filtration of coaugmented coalgebras}
\label{202207211508}

This section concerns the coradical filtration of coaugmented coalgebras. After recalling the usual definition of the coradical filtration, we give an equivalent definition based on idempotents by  mimicking the notion of cross-effects for functors recalled in section  \ref{202208131840}. 
The description of the $n$-th term of the coradical filtration as a kernel in Proposition \ref{202208302218} will be useful in section \ref{202211251159}. To illustrate Proposition \ref{202208302218}, we give in section \ref{202211292154} the coradical filtrations of the polynomial bialgebra, the shuffle bialgebra and finite-dimensional bialgebras.

\subsection{Coradical filtration}
\label{202212191823Japan}

Let $\mathds{k}$ be a commutative unital ring.
We recall that a {\it coaugmented coalgebra} (over $\mathds{k}$) is a quadruple $(H, \Delta  , \varepsilon  , \eta )$ such that $(H, \Delta, \varepsilon)$ is a coalgebra with the comultiplication $\Delta : H \to H \otimes H$ and counit $\varepsilon : H \to \mathds{k}$ ; and $\eta : \mathds{k} \to H$ is a coalgebra homomorphism.
We abbreviate $(H, \Delta  , \varepsilon  , \eta )$ to a coaugmented coalgebra $H$ if no confusion arises, and use notations $\Delta_H, \varepsilon_H, \eta_H$ for its structure maps.
For simplicity, we introduce the notation $1_H$ to denote $\eta (1_\mathds{k}) \in H$.
A morphism between coaugmented coalgebras is a coalgebra homomorphism preserving the coaugmentations.
We denote by $\mathds{k} \downarrow \mathsf{Coalg}_{\mathds{k}}$ the category of coaugmented coalgebras and their morphisms.
The tensor product of coalgebras induces a symmetric monoidal structure on $\mathds{k} \downarrow \mathsf{Coalg}_{\mathds{k}}$.

For a coaugmented coalgebra $H$, we introduce the endomorphism $e_{H} {:=} (\mathrm{id}_H - \eta_H \circ \varepsilon_H)$ of the underlying module of $H$.
It is an idempotent such that $\mathrm{Ker} (e_H) = \mathds{k} 1_H$ and $\mathrm{Im} ( e_H ) = \overline{H}$ where $\overline{H} {:=} \mathrm{Ker} ( \varepsilon_H )$.
In particular, we have an eigenspace decomposition $H \cong \mathds{k} 1_H \oplus \overline{H}$.

{\it The reduced comultiplication} $\overline{\Delta}_H : \overline{H} \to \overline{H} \otimes \overline{H}$ is defined by $\overline{\Delta}_H ( a) {:=} \Delta_H (a) - a \otimes 1_H - 1_H \otimes a$.
We define $\mathrm{Prim} (H) {:=} \mathrm{Ker} ( \overline{\Delta}_H )$ whose elements are called to be {\it primitive}.
It is well known that for a bialgebra $H$, $\mathrm{Prim} (H)$ associated with the underlying coaugmented coalgebra of $H$ becomes a Lie algebra by the commutator.

The reduced comultiplication is coassociative and thus it induces a well-defined $n$-times iterated reduced comultiplication $\overline{\Delta}_H^{(n)} : \overline{H} \to \overline{H}^{\otimes n}$.
This induces {\it the coradical filtration} consisting of submodules $P_n (H) {:=} \mathds{k} 1_H \oplus \mathrm{Ker} ( \overline{\Delta}_H^{(n+1)} ) \subset H$ for $n \in \mathds{N}$ (e.g. see  \cite[Section 1.2.4]{loday2012algebraic}):
\begin{align}
\label{202211251436}
P_0 (H) \subset \cdots \subset P_n (H) \subset P_{n+1} (H) \subset \cdots \subset H 
\end{align}

\begin{remark}
Note that, in this paper, we use the same notation $P_n$ to denote the coradical filtration of coaugmented coalgebras and also the polynomial approximation of functors (see Definition \ref{202207202245}).
We investigate their relation in section \ref{202211251159}.
\end{remark}

The coaugmented coalgebra $H$ {\it has degree less than or equal to $n$} if $P_n (H) = H$. We denote by $(\mathds{k} \downarrow \mathsf{Coalg}_{\mathds{k}})_{\leq n}$ the full subcategory of $\mathds{k} \downarrow \mathsf{Coalg}_{\mathds{k}}$ of coaugmented coalgebras of degree less than or equal to $n$.

The coaugmented coalgebra $H$ is {\it conilpotent} if $H = \bigcup_{n \in \mathds{N}} P_n (H)$, or, equivalently, for any $x \in \overline{H}$ there exists $n \in \mathds{N}$ such that $\overline{\Delta}^{(n)}_H (x) = 0$.

The following definition should be viewed as an analogue for coaugmented coalgebras of the explicit description of the cross-effect functor given in (\ref{202211251145}).
\begin{Defn}
\label{2022073111613}
Let $H$ be a coaugmented coalgebra and $n \in \mathds{N}$. 
For $a \in H$, we define a linear map $\delta^n_H : H \to H^{\otimes n}$ by $\delta^n_{H} (a) {:=} e_H^{\otimes n} \left( \Delta_H^{(n)} (a) \right) \in H^{\otimes n}$.
Here, $\Delta^{(n)}_H$ is the $n$-times iterated comultiplication of $H$ and $e_H^{\otimes n} : H^{\otimes n} \to H^{\otimes n}$ is the $n$-fold tensor product of $e_H$.
\end{Defn}

\begin{Example}
\label{202212181851}
For $n = 0,1,2$, we have
\begin{align*}
\delta^0_H (a) &= \varepsilon_H (a) , \\
\delta^1_H (a) &= e_H (a) = a - \varepsilon_H(a) 1_H ,\\
\delta^2_H (a) &= \Delta_H (a) - a \otimes 1_H - 1_H \otimes a + \varepsilon_H (a) 1_H^{\otimes 2} .
\end{align*}
\end{Example}

\begin{prop}
\label{202208302218}
Let $n \in \mathds{N}$.
\begin{enumerate}
\item
For $a \in \overline{H}$, we have $\delta^n_H (a) = \overline{\Delta}_H^{(n)} ( a )$.
\item
We have $P_n ( H) = \mathrm{Ker} ( \delta^{n+1}_H )$.
\end{enumerate}
\end{prop}
\begin{proof}
The linear map $\delta^n_H$ can be considered as an $n$-times iterated application of $\delta^ 2_H : H \to H \otimes H$ as follows.
Namely, we have $(\delta^2_H \otimes \mathrm{id}_H^{\otimes (n-1)} ) \circ \delta^n_H = \delta^{n+1}_H$, by the calculation below:
\begin{align*}
(\delta^2_H \otimes \mathrm{id}_H^{\otimes (n-1)} ) \circ e_H^{\otimes n} \circ \Delta_H^{(n)} =& (\delta^2_H \otimes \mathrm{id}_H^{\otimes (n-1)} ) \circ ( (\mathrm{id}_H - \eta_H \circ \varepsilon_H ) \otimes e_H^{\otimes (n-1)} ) \circ \Delta_H^{(n)} , \\
=& (\delta^2_H \otimes \mathrm{id}_H^{\otimes (n-1)} ) \circ (\mathrm{id}_H \otimes e_H^{\otimes (n-1)} ) \circ \Delta_H^{(n)} , \\
=& ((e_H^{\otimes 2} \circ \Delta_H) \otimes \mathrm{id}_H^{\otimes (n-1)} ) \circ (\mathrm{id}_H \otimes e_H^{\otimes (n-1)} ) \circ \Delta_H^{(n)}  \\
=& e_H^{\otimes (n+1)} \circ \Delta_H^{(n+1)} . 
\end{align*}
Here, the second equality follows from $\delta^2_H \circ \eta_H = 0$.
By Example \ref{202212181851} and using $a \in \overline{H}$, we have:
$$\delta^2_H (a) = \Delta_H (a) - a \otimes 1_H - 1_H \otimes a = \overline{\Delta}_H (a)$$ for every $a \in \overline{H}$.
Hence, we can inductively prove the first assertion by:
$$\delta^{n+1}_H (a) = ( (\delta^2_H \otimes \mathrm{id}_H^{\otimes (n-1)} ) \circ \delta^n_H ) (a)  = ( (\overline{\Delta}_H \otimes \mathrm{id}_H^{\otimes (n-1)} ) \circ \overline{\Delta}_H^{(n)} ) (a)  = \overline{\Delta}_H^{(n+1)} (a).$$
We prove the second assertion. 
It is obvious that $\delta^n (1_H) = 0$, and thus we have $\mathrm{Ker} ( \delta^n_H ) = \mathds{k} 1_H \oplus \mathrm{Ker} ( \overline{\Delta}_H^{(n)} ) = P_n (H)$, by the above result and the definition of $P_n (H)$.
\end{proof}

Let $n , k \geq 1$ be natural numbers and consider a subset $S = \{ m_1 , m_2 , \cdots , m_k \} \subset \underline{n}$ with $m_1 <m_2 < \cdots < m_k$.
We define the linear map $(-)_{S} : H^{\otimes k} \to H^{\otimes n}$ to be the linear extension of the assignment of $1^{\otimes (m_1 -1)} \otimes a_1 \otimes 1^{\otimes (m_2 -m_1 -1)} \otimes a_2 \otimes \cdots \otimes 1^{\otimes (m_k - m_{k-1} -1)} \otimes a_k \otimes 1^{\otimes (n-m_k)}$ to a tensor product $a_1 \otimes \cdots \otimes a_k \in H^{\otimes k}$.
In other words, it inserts the $j$-th tensor component into the $m_j$-th component.

The idempotent $e_H$ induces a direct sum decomposition of $H^{\otimes n}$ as follows.
Let $\tau_i (e_H) {:=} \mathrm{id}_H^{\otimes (i-1)} \otimes e_H \otimes \mathrm{id}_H^{\otimes (n-i)}$ for $1 \leq i \leq n$.
For a subset $S \subset \underline{n}$, let 
$$H_{n,S} {:=} \{ v \in H^{\otimes n } ~;~ (\tau_i (e_H))( v) = v ~( i \in S),~ (\tau_i (e_H) ) (v) = 0~ ( i \not\in S) \} . $$
It is obvious that $H_{n, S} = \{ w_S \in H^{\otimes n} ~;~ w \in \overline{H}^{\otimes k} \} \cong \overline{H}^{\otimes k}$ where $|S| = k$.
Furthermore, we have the decomposition below, since the $\tau_i (e_H)$'s are commuting idempotents.
\begin{align}
\label{202208301310}
H^{\otimes n} \cong \bigoplus_{S \subset \underline{n}} H_{n,S} .
\end{align}

\begin{prop}
\label{202207301503}
Let $H$ be a coaugmented coalgebra.
For $n \in \mathds{N}$ and $a \in H$, the $H_{n,S}$-component of $\Delta_H^{(n)} (a)$ in the decomposition (\ref{202208301310}) is $\delta^{|S|}_H (a)_S$.
In particular, the following equality holds.
\begin{align*}
\Delta^{(n)}_H (a) = \sum_{S \subset \underline{n}} \delta^{|S|}_H ( a)_{S} .
\end{align*}
\end{prop}
\begin{proof}
For $S \subset \underline{n}$ such that $|S|=k$, let $\pi_{n,S} : H^{\otimes n} \to H_{n,S}$ be the projection to the eigenspace.
Note that the composition of $\pi_{n,S}$ and the inverse of the isomorphism $(-)_S : \overline{H}^{\otimes k} \to H_{n,S}$ is equal to the tensor product $\pi^\prime_{n,S} = \bigotimes_{i \in S} e_H \otimes \bigotimes_{i \in \underline{n} \backslash S} \varepsilon_H : H^{\otimes n} \to \overline{H}^{\otimes k}$.
By using the counital property, we have $\pi^\prime_{n,S} ( \Delta_H^{(n)} (a) ) = \delta^{k}_H (a)$.
Hence, we obtain $\pi_{n,S} ( \Delta_H^{(n)} (a) ) = \delta^{k}_H (a)_S$.
\end{proof}

\begin{prop}
\label{202207282312}
Let $H_1$ and $H_2$ be coaugmented coalgebras.
If $a \in P_n ( H_1)$ and $b\in P_m (H_2)$, then we have $a \otimes b \in P_{n+m} (H_1 \otimes H_2 )$.
\end{prop}
\begin{proof}
It suffices to prove that $\delta^{n+m+1}_{H_1 \otimes H_2} ( a \otimes b) = 0$.
The iterated comultiplication $\Delta^{(n+m+1)} ( a \otimes b)$ corresponds to $\Delta^{(n+m+1)} (a) \otimes \Delta^{(n+m+1)} (b)$ up to the switching isomorphism 
$$\tau : H_1^{\otimes (n+m+1)} \otimes H_2^{\otimes (n+m+1)} \to (H_1 \otimes H_2)^{\otimes (n+m+1)} . $$
We apply the decomposition in Proposition \ref{202207301503} to $\Delta^{(n+m+1)} (a) , \Delta^{(n+m+1)} (b)$ respectively.
By Proposition \ref{202208302218}, $\delta^k_{H_1} (a) = 0$ for $k \geq n+1$ and $\delta^l_{H_2} (b) = 0$ for $l \geq m+1$.
Hence, $\Delta^{(n+m+1)} (a) \otimes \Delta^{(n+m+1)} (b)$ is the sum of $\delta^k_{H_1} (a)_{i_1, \cdots ,i_k} \otimes \delta^l_{H_2} (b)_{j_1, \cdots ,j_l}$ for $k \leq n$, $l \leq m$, $i_1, \cdots ,i_k \in \underline{(n+m+1)}$ and $j_1, \cdots ,j_l \in \underline{(n+m+1)}$.
Note that we have
$$\tau \left(  \delta^k_{H_1} (a)_{i_1, \cdots ,i_k} \otimes \delta^l_{H_2} (b)_{j_1, \cdots ,j_l} \right) \in ( H_1 \otimes H_2 )_{n+m+1,S} $$ 
where $S= \{ i_1, \cdots ,i_k , j_1, \cdots ,j_l \}$ and $|S| < n+m+1$.
Thus, by Proposition \ref{202207301503}, we see that $\delta^{n+m+1}_{H_1 \otimes H_2} ( a \otimes b) = 0$.
\end{proof}

\begin{prop}
\label{202207281554}
The tensor product of conilpotent coaugmented coalgebras is conilpotent.
\end{prop}
\begin{proof}
Let $H_1, H_2$ be conilpotent coaugmented cocommutative coalgebras.
Let $M = \mathrm{colim}_j P_j ( H_1 \otimes H_2 )$ and $u : M \to H_1 \otimes H_2$ be the canonical injection.
We prove that $u$ is surjective.
By Proposition \ref{202207282312}, we obtain a linear map $P_n ( H_1) \otimes P_m ( H_2 ) \to P_{n+m} (H_1 \otimes H_2 )$.
This map is compatible with the filtration, so that it induces $v : \mathrm{colim}_{n} P_n ( H_1 ) \otimes \mathrm{colim}_{m} P_ m ( H_2) \to M$ whose domain turns out to be $H_1 \otimes H_2$.
The map $u$ is surjective since $u \circ v$ is the identity.
\end{proof}

By replacing the tensor products in Proposition \ref{202207282312} with the multiplication of $H$, we also obtain the following statement:
\begin{prop} \label{202401101549}
Let $H$ be a bialgebra.
For $a \in P_n (H)$ and $b \in P_m (H)$, we have $a \cdot b \in P_{n+m} (H)$ .
\end{prop}

\subsection{Coradical filtration: the case of a field $\mathds{k}$}
\label{202211251457}

In this section, we assume that the ground ring $\mathds{k}$ is a field.
Under this hypothesis, we prove that the coradical filtration lifts to a coaugmented coalgebra filtration.
Furthermore, the assignment of the $n$-th coradical filtration $P_n (H)$ to a coaugmented coalgebra $H$ could be described as a right adjoint.

\begin{Lemma}
\label{202207282250}
For linear maps $f: V_1 \to V_0$ and $g : W_1 \to W_0$, let $f \tilde{\otimes} g : V_1 \otimes W_1 \to ( V_0 \otimes W_1 ) \oplus (V_1 \otimes W_0)$ be $f \otimes \mathrm{id}_{W_1} \oplus \mathrm{id}_{V_1} \otimes g$.
Then we have $\mathrm{Ker} (f) \otimes \mathrm{Ker} (g) \cong \mathrm{Ker} (f \tilde{\otimes} g)$.
\end{Lemma}
\begin{proof}
Let $C_\bullet$ be the chain complex concentrated only on the zeroth and first components which are given by $V_0$ and $V_1$ respectively.
The only nontrivial differential is given by $f$.
Let $D_\bullet$ be the similar chain complex induced by $g$.
Then the claim follows from the K{\" u}nneth formula with respect to the zeroth homology.
\end{proof}

\begin{prop}
\label{202207311344}
The subspace $P_n (H )$ is a subcoalgebra of $H$ which inherits the coaugmentation.
\end{prop}
\begin{proof}
We prove that the composition $f : P_n (H) \hookrightarrow H \stackrel{\Delta}{\to} H \otimes H$ factors through $P_n(H) \otimes P_n (H) \hookrightarrow H \otimes H$.
Lemma \ref{202207282250} implies that $P_n(H) \otimes P_n (H) = \mathrm{Ker} ( \delta^{n+1}_{H} ) \otimes \mathrm{Ker} ( \delta^{n+1}_{H} ) = \mathrm{Ker} ( \delta^{n+1}_{H} \tilde{\otimes} \delta^{n+1}_{H} )$.
Hence, it suffices to show that $(\delta^{n+1}_{H} \tilde{\otimes} \delta^{n+1}_{H}) \circ f = 0$, in particular that $( \delta^{n+1}_{H} \otimes \mathrm{id}_H ) \circ f = 0$.
For $h \in P_n(H)$,  by the filtration (\ref{202211251436}) we have $h \in P_n(H) \subset P_{n+1}(H)$, so by Proposition \ref{202208302218} we have $\delta^{n+2}_H (h) = 0$.
Note that we have
$$\delta^{n+2}_{H} = \left( \delta^{n+1}_{H} \otimes ( \mathrm{id}_H - \eta \circ \varepsilon ) \right) \circ \delta_H . $$
Hence, we have $\left( ( \delta^{n+1}_{H} \otimes \mathrm{id}_H ) \circ f \right) (h) = \left( ( \delta^{n+1}_{H} \otimes ( \eta \circ \varepsilon) ) \circ f \right) (h)$. 
Moreover, for $h \in P_n (H)$,  we have
$$\left( ( \delta^{n+1}_{H} \otimes ( \eta \circ \varepsilon) ) \circ f \right) (h)= \delta^{n+1}_{H} (h) \otimes \eta (1) = 0 .$$
Therefore, $f$ factors as $\Delta_{P_n(H)} : P_n (H) \to P_n (H) \otimes P_n (H)$.
All that remain is to give the counit and coaugmentation of $P_n(H)$.
The restriction of the counit $\varepsilon$ is the counit of $P_n (H)$ ; and the coaugmentation $\eta : \mathds{k} \to H$ induces the coaugmentation of $P_n (H)$.
\end{proof}

\begin{Corollary}
The functor $P_n :  \left(\mathds{k} \downarrow \mathsf{Coalg}_\mathds{k} \right) \to \mathsf{Mod}_\mathds{k}$ lifts along the forgetful functor $\left(\mathds{k} \downarrow \mathsf{Coalg}_\mathds{k} \right)_{\leq n} \to \mathsf{Mod}_\mathds{k}$.
Moreover, the lift is a right adjoint of the embedding functor $i$:
\begin{equation}
\begin{tikzcd}
\left(\mathds{k} \downarrow \mathsf{Coalg}_\mathds{k} \right)_{\leq n} \ar[r, "i", shift left = 2] &  \left(\mathds{k} \downarrow \mathsf{Coalg}_\mathds{k} \right) \ar[l, "P_n", "\perp"', shift left = 2]
\end{tikzcd}
\end{equation}
\end{Corollary}
\begin{proof}
The first claim is immediate from Proposition \ref{202207311344}.
We sketch the proof of the second part.
Let $H$, $L$ be coaugmented coalgebras and $\varphi : L \to H$ be a morphism.
Suppose that $\mathrm{deg} (L ) \leq n$, i.e. $\delta^{n+1}_{L} =0$.
We have $\varphi^{\otimes (n+1)} \circ \delta^{n+1}_{H} = \delta^{n+1}_{L} \circ \varphi = 0$ since $\varphi$ preserves the structure maps.
Hence $\varphi$ uniquely factors through the inclusion $P_n (H) \hookrightarrow H$.
It proves $\mathrm{Nat} ( i(L) , H ) \cong \mathrm{Nat} ( L , P_n (H ) )$.
\end{proof}

\subsection{Examples of coradical filtrations}
\label{202211292154}

In this section, we give concrete examples of coradical filtrations.

\subsubsection{The polynomial bialgebra} \label{202211071544}

Let $\mathds{k}$ be a commutative unital ring and $\mathds{k}[t]$ be the polynomial algebra with the variable $t$ over $\mathds{k}$.
It is well-known that the algebra structure extends to a Hopf algebra structure with the comultiplication characterized by $\Delta (t) = 1 \otimes t + t \otimes 1$.
If $m \cdot 1_{\mathds{k}} \in \mathds{k}$ is not a zero divisor for any nonzero integer $m$, then the coradical filtration of $\mathds{k} [t]$ coincides with the degree filtration:
$$P_n ( \mathds{k} [t] ) = \bigoplus_{i \leq n} \mathds{k} t^i$$
For example, $\mathds{k}$ could be the integers $\mathds{Z}$ or a field of characteristic zero.

We go further by giving an explicit calculation of $P_n ( \mathds{k} [t] )$ in {\it positive characteristic}.
The precise meaning of positive characteristic for a general commutative ring $\mathds{k}$ is given in the statement of Theorem \ref{202211071539}.

Let $p$ be a prime.
For $n \in \mathds{N}$, let $( n_i )_{i \in \mathds{N}}$ be the representation of $n$ in base $p$.
In other words, $( n_i )_{i \in \mathds{N}}$ is a sequence of $\mathds{N}$ such that $n = \sum^{\infty}_{i=0} n_i p^i$ and $0 \leq n_i < p$.
We define the map $L_p : \mathds{N} \to \mathds{N}$ by $L_p (n) {:=} \sum^{\infty}_{i=0} n_i$.

We give a brief review of multinomial coefficients.
Let $m , r \in \mathds{N}$ such that $r \geq 1$.
For $k_1, k_2, \cdots, k_r \in \mathds{Z}$ such that $\sum^{r}_{j=1} k_j = m$, the {\it multinomial coefficient} $\begin{pmatrix} m \\ k_1 , k_2, \cdots , k_r \end{pmatrix} \in \mathds{N}$ is characterized by the equation,
$$(z_1 + z_2 + \cdots + z_r)^m = \sum_{\substack{k_i \in \mathds{Z} \\ \sum_{i}k_i = m}} \begin{pmatrix} m \\ k_1 , k_2, \cdots , k_r \end{pmatrix} z_1^{k_1} z_2^{k_2} \cdots z_r^{k_r} .$$
One can deduce the following lemma from elementary number theory, in particular the Lucas theorem on binomial coefficients.

\begin{Lemma}
\label{202209182244}
Let $p$ be a prime.
For $m , r \in \mathds{N}$, the following conditions are equivalent:
\begin{itemize}
\item
There exist positive integers $k_1, k_2, \cdots, k_r$ such that $m = \sum^{r}_{j=1} k_j$ and the multinomial coefficient $\begin{pmatrix} m \\ k_1 , k_2, \cdots , k_r \end{pmatrix}$ is coprime to $p$.
\item
$r \leq L_p (m)$.
\end{itemize}
\end{Lemma}

\begin{theorem}
\label{202211071539}
Suppose that $p > 0$ is a prime such that $p \cdot 1_{\mathds{k}} = 0$ and $m \cdot 1_{\mathds{k}} \in \mathds{k}$ is not a zero divisor for $m \in \mathds{Z}$ which is coprime to $p$.
Then the coradical filtration of $\mathds{k} [t]$ is given by,
$$P_n ( \mathds{k} [t] ) = \bigoplus_{L_p (i) \leq n} \mathds{k} t^i .$$
In particular, the underlying coalgebra of $\mathds{k} [t]$ is conilpotent.
\end{theorem}
\begin{proof}
For simplicity, write $H = \mathds{k} [t]$.
By Proposition \ref{202208302218}, it suffices to compute the kernel of $\delta^{n+1}_H$ (see Definition \ref{2022073111613}).
One can compute $\delta^{n+1}_H ( t^m ) = e^{\otimes (n+1)}_H ( \Delta^{(n+1)}_H ( t^m ) )$ as follows:
\begin{align*}
\delta^{n+1}_H ( t^m ) &= \sum_{\substack{0 < k_i \\ \sum_{i}k_i = m}} \begin{pmatrix} m \\ k_1, k_2, \cdots , k_{n+1} \end{pmatrix} t^{k_1} \otimes t^{k_2} \otimes \cdots \otimes t^{k_{n+1}} 
\end{align*}
By Corollary \ref{202209182244} and the assumption on $p$, $\delta^{n+1}_H (t^m ) \neq 0$ if and only if $L_p (m) > n$.
Furthermore, the assumption on $p$ implies that $\{ \delta^{n+1}_H (t^m ) ~;~ L_p (m) > n\}$ is linearly independent since the total degrees of $t^{k_1} \otimes t^{k_2} \otimes \cdots \otimes t^{k_{n+1}}$, i.e. $\sum k_i = m$, are different from each other.
\end{proof}

\begin{remark}
In section \ref{202211301655}, we give a review of primitive bialgebras which are automatically conilpotent.
One can go further by using the above calculation to prove that the polynomial bialgebra is primitive.
\end{remark}

\subsubsection{The shuffle bialgebra}
\label{202212111901}

For a $\mathds{k}$-module $V$, we give a quick review of the shuffle-deconcatenation bialgebra, or simply the shuffle bialgebra, $\mathrm{Sh} (V)$ whose multiplication is given by shuffles and comultiplication is given by deconcatenation.
Its underlying $\mathds{k}$-module is the tensor module $T(V) {:=} \bigoplus_{i \in \mathds{N}} V^{\otimes i}$.
For $v_1 , v_2 , \cdots , v_n \in V$, we introduce the notation $\left[ v_1 | v_2 | \cdots | v_i \right] {:=} v_1 \otimes v_2 \otimes \cdots \otimes v_n \in V^{\otimes n} \subset T(V)$.
For natural numbers $p,q$ such that $p+q = n$, a permutation $\sigma \in \Sigma_n$ is a $(p,q)$-shuffle if $\sigma (1) < \sigma (2) < \cdots < \sigma (p)$ and $\sigma (p+1) < \sigma ( p+2) < \cdots < \sigma ( n)$.
Denote by $\Sigma_{p,q}$ the subset of $(p,q)$-shuffles.
The multiplication, denoted by $\shuffle$, is characterized by the following assignment:
$$
\left[ v_1 | v_2 | \cdots | v_i \right] \shuffle \left[ v_{i+1} | v_{i+2} | \cdots | v_n \right] {:=} \sum_{\sigma \in \Sigma_{i,n-i}} \left[ v_{\sigma^{-1} (1)} | v_{\sigma^{-1} (2)} | \cdots | v_{\sigma^{-1} (n)} \right] .
$$
The comultiplication $\Delta = \Delta_{\mathrm{Sh} (V)}$ is defined by deconcatenation:
$$
\Delta ( \left[ v_1 | v_2 | \cdots | v_n \right] ) {:=} \left[ v_1 | \cdots | v_n \right] \otimes 1_\mathds{k} + \sum^{n-1}_{i = 1} \left( \left[ v_1 | v_2 | \cdots | v_i \right] \otimes \left[ v_{i+1} | v_{i+2} | \cdots | v_n \right] \right) + 1_\mathds{k} \otimes \left[ v_1 | \cdots | v_n \right].
$$
The above data yield a well-defined commutative bialgebra $\mathrm{Sh} (V)$.
For a free $\mathds{k}$-module $V$, $\mathrm{Sh} (V)$ is cocommutative if and only if the rank of $V$ does not exceed $1$.
The shuffle bialgebra is a Hopf algebra since $S( \left[ v_1 | v_2 | \cdots | v_n \right] ) = (-1)^n \left[ v_n | v_{n-1} | \cdots | v_1 \right]$ gives the antipode.

\begin{prop}
\label{202211291552}
The coradical filtration of $\mathrm{Sh} (V)$ is identified as follows:
$$
P_n ( \mathrm{Sh} (V)) = \bigoplus_{0 \leq i \leq n} V^{\otimes i}
$$
In particular, the underlying coalgebra of the shuffle bialgebra $\mathrm{Sh} (V)$ is conilpotent.
\end{prop}
\begin{proof}
Put $H = \mathrm{Sh} (V)$.
One can inductively prove that the $(n+1)$-times iterated reduced comultiplication acts on $\left[ v_1 | v_2 | \cdots | v_m \right] \in \mathrm{Sh} (V)$ as follows:
$$\delta^{n+1}_H ( \left[ v_1 | v_2 | \cdots | v_m \right] ) = \sum_{i_1, \cdots , i_{n}} \left( \left[ v_1 | v_2 | \cdots | v_{i_1} \right] \otimes \left[ v_{i_1+1} | v_{i_1+2} | \cdots | v_{i_2} \right] \otimes \cdots \left[ v_{i_{n}+1} | v_{i_{n}+2} | \cdots | v_{m} \right] \right)$$
Here, the indices $i_1, \cdots , i_{n}$ run all over $1\leq i_1 < i_2 < \cdots < i_{n} < m$.
Hence, the left hand side in the statement contains the right hand side.
Moreover, if $m > n$, then the above computation implies that $\delta^{n+1}_{H} |_{\mathrm{Sh}_m (V)}$ is injective where $\mathrm{Sh}_m (V)$ is the subspace of $\mathrm{Sh} (V)$ generated by $\left[ v_1 | v_2 | \cdots | v_m \right]$'s.
This completes the proof, since $\delta^{n+1}_H (\mathrm{Sh}_m (V)) \cap \delta^{n+1}_H (\mathrm{Sh}_{m^\prime} (V)) = 0$ for $m \neq m^\prime$.
\end{proof}

\subsubsection{Finite-dimensional bialgebras}

In this section, we give a computation of the coradical filtration of finite-dimensional bialgebras over a field $\mathds{k}$ of characteristic zero.
In particular, we prove that a finite-dimensional bialgebra is conilpotent if and only if it is trivial.
For this, in Lemma \ref{202209010909}, we examine a compatibility of the multiplication of a bialgebra $H$ and $\delta^n_H$ of Definition \ref{2022073111613} for general characteristic.

Below, we regard $\delta_H^n (a) \in H^{\otimes n}$ as an element of the shuffle algebra $\mathrm{Sh} (H)$ of the underlying module of $H$.
\begin{Lemma}
\label{202209010909}
For $a \in P_n (H)$ and $b \in P_m (H)$, we have $\delta_H^{n + m } ( ab) = \delta_H^n (a) \shuffle \delta_H^m (b)$.
\end{Lemma}
\begin{proof}
Let $N = n + m$.
By proposition \ref{202207301503}, we have 
$$\Delta_H^{(N)} ( ab) = \Delta_H^{(N)} (a) \Delta_H^{(N)} (b) = \left( \sum_{S \subset \underline{N}} \delta_H^{|S|} (a)_S \right) \left( \sum_{T \subset \underline{N}} \delta_H^{|T|} (b)_T \right) . $$
By the assumptions $a \in P_n (H)$ and $b \in P_m (H)$, this is equal to 
$$\left( \sum_{\substack{S \subset \underline{N} \\ |S| \leq n}} \delta_H^{|S|} (a)_S \right) \left( \sum_{\substack{T \subset \underline{N} \\ |T| \leq m}} \delta_H^{|T|} (b)_T \right) = \sum_{\substack{S,T \subset \underline{N} \\ |S| \leq n, |T| \leq m}} \delta_H^{|S|} (a)_S ~ \delta_H^{|T|} (b)_T . $$
Note that, by Proposition \ref{202207301503}, we have $\delta_H^{|S|} (a)_S \in H_{N,S}$.
In other words, $\delta_H^{|S|} (a)_S$ is a linear combination of vectors in $H^{\otimes N}$ whose $j$-th tensor component is $1_H$ for all $j \not\in S$.
$\delta_H^{|T|} (b)_T$ satisfies a similar property, so that their product $\delta_H^{|S|} (a)_S ~ \delta^{|T|} (b)_T$ is a linear combination of vectors in $H^{\otimes N}$ whose $j$-th tensor component is $1_H$ for all $j \not\in S \cup T$.
Hence, if $S \cup T \neq \underline{N}$, then we have $e_H^{\otimes N} \left( \delta_H^{|S|} (a)_S ~ \delta_H^{|T|} (b)_T \right) =0$.
This observation leads to
\begin{align*}
\delta_H^{N} (ab) = (e_H^{\otimes N} ) ( \Delta_H^{(N)} (ab)) = \sum_{\substack{S,T \subset \underline{N}, |S|=n \\ S\cap T= \emptyset}} \delta_H^{|S|} (a)_S ~ \delta_H^{|T|} (b)_T = \sum_{\substack{S,T \subset \underline{N}, |S|=n \\ S\cap T= \emptyset}} \delta_H^{n} (a)_S ~ \delta_H^{m} (b)_T
\end{align*}
This is equal to the shuffle multiplication $\delta_H^n (a) \shuffle \delta_H^m (b)$ by definitions.
\end{proof}

We recall a classical result on shuffle algebras.
For a field $\mathds{k}$ of characteristic zero, the shuffle algebra $\mathrm{Sh} (V)$ is isomorphic to a polynomial algebra based on certain multi-variables \cite{radford1979natural}.
Choose a basis $X$ of $V$ equipped with an order.
A Lyndon word on $X$ (equipped with the order) is a finite word on $X$ which is strictly smaller than all of its rotations with respect to the lexicographical order.
Denote by $\mathrm{Lyn} (X)$ the set of Lyndon words on $X$.
Then there exists an algebra isomorphism $\mathrm{Sh} (V) \cong \mathds{k} [ \mathrm{Lyn} (X) ]$.

\begin{Lemma}
\label{202209011041}
Let $H$ be a (not necessarily finite-dimensional) bialgebra over a field of characteristic zero.
For $a \in P_n (H)$ such that $a \not\in \mathds{k} 1_H$, the sequence $( a^k )_{k\in \mathds{N}}$ of the powers of $a$ is linearly independent.
\end{Lemma}
\begin{proof}
Without loss of generality, we may assume that $a \in P_n (H) \backslash P_{n-1} (H)$.
In other words, $\delta_H^{n+1}(a) =0$ but $\delta_H^{n} (a) \neq 0$.
Consider the linear map $\delta = \sum_{m \in \mathds{N}} \delta_H^{m} : H \to \bigoplus_{m \in \mathds{N}} H^{\otimes m} = \mathrm{Sh} (H)$.
It suffices to show that the induced sequence $( \delta( a^k ) )_{k\in \mathds{N}}$ is linearly independent.
Note that $\delta_H^m (a^k) = 0$ if $m > kn$, since $a^k \in P_{nk} (H)$ by Proposition \ref{202401101549}.
We claim that $\delta_H^m (a^k ) \neq 0$ if $m = kn$.
Then it is obvious that $( \delta ( a^k ) )_{k\in \mathds{N}}$ is linearly independent.

We now prove the above claim.
Indeed, an iterated application of Lemma \ref{202209010909} implies $\delta_H^{kn} (a^k ) = \delta_H^n (a)^{\shuffle k}$.
The shuffle algebra $\mathrm{Sh} (H)$ is isomorphic to a polynomial algebra since $\mathds{k}$ is a field of characteristic zero.
Hence, $\delta_H^{kn} (a^k ) = \delta_H^n (a)^{\shuffle k} \neq 0$ due to $\delta_H^n (a) \neq 0$.
\end{proof}

\begin{theorem}
\label{202211292056}
Let $\mathds{k}$ be a field of characteristic zero.
For a finite-dimensional bialgebra $H$, we have $P_n (H) \cong \mathds{k}$.
In particular, $H$ is conilpotent if and only if $H \cong \mathds{k}$.
\end{theorem}
\begin{proof}
It is immediate from Lemma \ref{202209011041}.
\end{proof}

\begin{Example}
Let $G$ be a finite group and $H = \mathds{k}G$ be the group algebra over $\mathds{k}$ a field of characteristic zero.
It is well known that $H$ is a bialgebra whose comultiplication $\Delta$ is characterized by $\Delta (g) = g \otimes g$ for $g \in G$.
By Theorem \ref{202211292056}, $H$ is conilpotent if and only if $G$ is trivial.
One can directly verify that this is true even for arbitrary ring $\mathds{k}$.
\end{Example}

\begin{Example}
\label{202212121114}
For $\mathds{k}$ a field of positive characteristic $p$, there exists a finite-dimensional conilpotent bialgebra $H$ with $\dim H > 1$.
Let $J$ be the ideal of $\mathds{k} [t]$ generated by $t^p$.
Then $J$ is a coideal of $\mathds{k} [t]$, i.e. $\Delta (J) \subset J \otimes \mathds{k} [t] + \mathds{k} [t] \otimes J$, so that $J$ induces the quotient Hopf algebra $H = \mathds{k} [t] / J$.
The Hopf algebra $H$ is conilpotent since the conilpotency is inherited via a quotient map.
This also follows from a stronger result that $H$ is primitive (see section \ref{202211301655}).
\end{Example}

\section{Isomorphism of polynomial and coradical filtrations}
\label{202211251159}

Let $\mathcal{C}$ be a category satisfying the conditions given in section \ref{202208181550}. In this section, we prove in Theorem \ref{202207141636}, that the polynomial filtration of an exponential functor on $\mathcal{C}$ and the coradical filtration coincide. This theorem extends a result obtained by Touz\'e \cite[Theorem 11.10]{touze} concerning functors on an additive category. Recall that our main example of interest, the category $\mathsf{gr^{\mathsf{op}}}$ is not additive but satisfies the conditions given in section \ref{202208181550}. Our proof is based on the alternative description of the coradical filtration of a coaugmented coalgebra  given in Proposition \ref{202208302218}.
Furthermore, we show that the analyticity of exponential $\mathcal{C}$-modules could be discussed using object-wise conilpotency.

We regard  the categories $\mathcal{C}$ and $\mathsf{Mod}_\mathds{k}$ as symmetric monoidal categories by using the product (given by the assumption in section \ref{202208181550}) and the tensor product respectively.
Let $\mathcal{F}^\mathsf{exp} ( \mathcal{C} ; \mathds{k} )$ be the category of exponential functors, equivalently symmetric monoidal functors, from $\mathcal{C}$ to $\mathsf{Mod}_\mathds{k}$.

\begin{prop}
\label{202312081109}
Let $\mathsf{Comon^{cc}} ( \mathcal{C} )$ be the category of cocommutative comonoids in $\mathcal{C}$.
The forgetful functor $( \ast \downarrow \mathsf{Comon^{cc}} ( \mathcal{C} ) ) \to \mathcal{C}$ gives an equivalence of symmetric monoidal categories.
\end{prop}
\begin{proof}
The universality of the product implies that the diagonal map is the unique comultiplication on each object.
\end{proof}

Based on Proposition \ref{202312081109}, we see that an exponential functor $F: \mathcal{C} \to \mathsf{Mod}_\mathds{k}$ induces an exponential functor $\widetilde{F} : \mathcal{C} \to ( \mathds{k} \downarrow \mathsf{Coalg}^\mathsf{cc}_{\mathds{k}})$.
In fact, it is given by  $\widetilde{F} : \mathcal{C} \simeq ( \ast \downarrow \mathsf{Comon^{cc}} ( \mathcal{C} ) ) \to  ( \mathds{k} \downarrow \mathsf{Comon^{cc}} ( \mathsf{Mod}_\mathds{k} ) ) =  (\mathds{k} \downarrow \mathsf{Coalg}^\mathsf{cc}_{\mathds{k}})$.
In particular, for $X \in \mathcal{C}$, $\widetilde{F} (X)$ is the coaugmented coalgebra whose underlying module is $F(X)$; and the comultiplication, counit and coaugmentation of $\widetilde{F} (X)$ are given by
\begin{align*}
F(X) &\stackrel{F(\Delta_X)}{\longrightarrow} F(X \times X) \cong F(X) \otimes F(X) , \\
\mathds{k} \cong F(\ast) &\stackrel{F(\eta_X)}{\longrightarrow} F(X) , \\
F(X) &\stackrel{F(\varepsilon_X)}{\longrightarrow} F(\ast) \cong \mathds{k} .
\end{align*}

For $X \in \mathcal{C}$, we have $\widetilde{F} (X ) \in (\mathds{k} \downarrow \mathsf{Coalg}^\mathsf{cc}_{\mathds{k}})$, so by section \ref{202212191823Japan}, we have the coradical filtration $P_n ( \widetilde{F} (X ) )$. In the following theorem, we compare the polynomial filtration $P_n(F)$ (see Definition \ref{202207202244}) applied to $X \in \mathcal{C}$ and the previous coradical filtration.

\begin{theorem}
\label{202207141636}
Let $F \in \mathcal{F}^\mathsf{exp} ( \mathcal{C} ; \mathds{k} )$ and $X \in \mathcal{C}$.
Under the identification $F(X) = \widetilde{F} (X)$, the polynomial filtration of $F$ evaluated at $X$ is equal to the coradical filtration of $\widetilde{F} (X)$.
In other words, we have the following commutative diagram for $n \in \mathds{N}$:
$$
\begin{tikzcd}[row sep=scriptsize]
F( X ) \ar[r , equal] & \widetilde{F} ( X ) \\
 ( P_n (F) ) ( X  ) \ar[u, hookrightarrow] \ar[r, equal] &  P_n ( \widetilde{F} ( X ) ) \ar[u, hookrightarrow]  
\end{tikzcd}
$$
\end{theorem}
\begin{proof}
The coaugmented coalgebra structure of $\widetilde{F} (X)$ induces the following commutative diagram.
The maps in the right two columns are induced by the exponential isomorphism of $F$.
Hence, the kernel of the composition of maps in the first row is equal to that in the second row.
$$
\begin{tikzcd}
F(X) \ar[d, equal] \ar[r, "\Delta^{(n+1)}_{\widetilde{F}(X)}"] & F(X)^{\otimes (n+1)} \ar[d, "\cong"]  \ar[r, "e_{\widetilde{F}(X)}^{\otimes (n+1)}"] & F(X)^{\otimes (n+1)} \ar[d, "\cong"] \\
F(X) \ar[r, "F(\Delta^{(n+1)}_X)"] & F(X^{\times (n+1)} ) \ar[r, "\chi_F ({[}X {]}^{n+1}_{i=1} )"] & F(X^{\times (n+1)} )
\end{tikzcd}
$$
\end{proof}

\begin{Corollary}
\label{202207211158}
Let $F \in \mathcal{F}^\mathsf{exp} ( \mathcal{C} ; \mathds{k} )$.
The $\mathcal{C}$-module $F$ is analytic if and only if the induced coaugmented coalgebra $\widetilde{F} (X)$ is conilpotent for all $X \in \mathcal{C}$.
\end{Corollary}
\begin{proof}
The $\mathcal{C}$-module $F$ is analytic, i.e. $\varinjlim_n ( P_n (F) ) ( X  ) = F(X)$ for any $X$, if and only if $\widetilde{F} ( X )$ is conilpotent, i.e. $\varinjlim_n P_n ( \widetilde{F} ( X ) )  = \widetilde{F} ( X )$, for any $X$ by Theorem \ref{202207141636}.
\end{proof}

We give an application of the theorem to coaugmented cocommutative coalgebras.

\begin{Defn}
Let $H$ be a coaugmented cocommutative coalgebra.
We denote by $\mathcal{C}_{H}$ the full subcategory of $( \mathds{k} \downarrow \mathsf{Coalg}_\mathds{k})$ consisting of any $n$-fold tensor product $H^{\otimes n}$ where $n \geq 0$.
\end{Defn}

\begin{prop} \label{202312072047}
The tensor product of coaugmented cocommutative coalgebras gives the product in the category $\mathcal{C}_H$.
\end{prop}
\begin{proof}
Let $X,Y$ be objects of $\mathcal{C}_H$.
We prove that $X \stackrel{\pi_X}{\leftarrow} X \otimes Y \stackrel{\pi_Y}{\to} Y$ gives the product of $X$ and $Y$ where $\pi_X = \mathrm{id}_X \otimes \varepsilon_Y$ and $\pi_Y = \varepsilon_X \otimes \mathrm{id}_Y$.
Let $f: Z \to X$ and $g : Z \to Y$ be morphisms of $\mathcal{C}_H$.
Then we define $h : Z \to X \otimes Y$ by $h = (f \otimes g) \circ \Delta_Z$.
Note that $\Delta_Z$ is a coalgebra homomorphism since $Z \in \mathcal{C}_H$ is cocommutative.
It turns out that $h$ gives a morphism in $\mathcal{C}_H$ such that $\pi_X \circ h = f$ and $\pi_Y \circ h = g$.
Moreover, if $h^\prime : Z \to X \otimes Y$ is a morphism of $\mathcal{C}_H$ such that $\pi_X \circ h^\prime = f$ and $\pi_Y \circ h^\prime = g$, then $h^\prime = h$.
This follows from
\begin{align*}
h^\prime &= (\pi_X \otimes \pi_Y) \circ \Delta_{X\otimes Y} \circ h^\prime , \\
&= (\pi_X \otimes \pi_Y) \circ (h^\prime \otimes h^\prime) \circ \Delta_Z , \\
&= (f \otimes g) \circ \Delta_Z .
\end{align*}
\end{proof}

\begin{remark}
In fact, the same proof implies that the tensor product of (coaugmented) coalgebras is the product in the category of cocommutative (coaugmented) coalgebras.
\end{remark}

By Proposition \ref{202312072047}, the category $\mathcal{C}_{H}$ is a small category with finite products and a null object as explained in section \ref{202208181550}.
The forgetful functor from $ \mathcal{C}_{H}$ to $\mathsf{Mod}_{\mathds{k}}$ gives an exponential functor $U_{H} \in \mathcal{F}^{\mathsf{exp}} ( \mathcal{C}_H ; \mathds{k})$.

%\begin{Lemma}
%\label{202211071052}
%Let $H$ be a coaugmented cocommutative coalgebra.
%Under the identification $U_H (H) = H$, the polynomial filtration of the functor $U_{H}$, evaluated at $H$, $\left(P_n ( U_{H} ) \right)(H )$ is equal to the coradical filtration $P_n (H)$ of $H$.
%In other words, we have the following commutative diagram:
%$$
%\begin{tikzcd}
%U_H (H) \ar[r, equal] & H \\
%\left(P_n ( U_{H} ) \right)(H ) \ar[r, equal] \ar[u, hookrightarrow] & P_n (H)\ar[u, hookrightarrow]
%\end{tikzcd}
%$$
%\end{Lemma}
%\begin{proof}
%We prove our claim by using the fact that $U_H$ is an exponential functor.
%Let $F = U_H$, and recall $\chi_{F}$ from section \ref{202208131840}.
%Under the exponential isomorphism $F ( \prod^{n+1}_{i=1} H_i ) \cong \bigotimes^{n+1}_{i=1} F( H_i )$, we have $\chi_{F} ( [ H_i ]^{n+1}_{i=1}) = \bigotimes^{n+1}_{i=1} ( 1_{F(H_i)} - F(p_{H_i}))$ where $H_i$'s are objects of $\mathcal{C}_H$.
%In particular, this implies that $\chi_{F} ( [ H ]^{n+1}_{i=1}) =e_{H}^{\otimes (n+1)}$.
%Hence, the filtration $(P_n ( U_{H} ) )(H )$ coincides with $\mathrm{Ker} ( \delta^{n+1}_{H})$.
%The result follows from Proposition \ref{202208302218}.
%\end{proof}

\begin{Corollary}
Let $H$ be a coaugmented cocommutative coalgebra.
The $\mathcal{C}_H$-module $U_H$ is analytic if and only if $H$ is conilpotent.
\end{Corollary}
\begin{proof}
We apply Theorem \ref{202207141636} to $\mathcal{C} = \mathcal{C}_H$ and $F = U_{H}$.
By Corollary \ref{202207211158}, it is clear that, if the $\mathcal{C}_{H}$-module $U_H$ is analytic, then $H$ is conilpotent since $\widetilde{U_H} ( H ) = H$.
Conversely, suppose that $H$ is conilpotent.
By Corollary \ref{202207211158} again, it suffices to prove that $\widetilde{U_H} (X)$ is analytic for any object $X \in \mathcal{C}_{H}$.
Note that any object $X$ is given by an $n$-fold tensor product $X = H^{\otimes n}$ for some $n$.
For such $X$, we have $\widetilde{U_H} (X) = H^{\otimes n}$ as coaugmented coalgebras by definitions.
By Proposition \ref{202207281554}, $H^{\otimes n}$ is conilpotent since so is $H$.
\end{proof}

\part{Application to $\mathsf{gr}^{\mathsf{op}}$}
\label{202312081754}

\section{The operad $\mathfrak{Ass}^u$ and its associated category}
\label{202212161644}

Every $\mathds{k}$-linear category $\mathcal{D}$ induces a bilinear functor $\mathcal{D}^{\mathsf{op}} \times \mathcal{D} \to \mathsf{Mod}_{\mathds{k}}$ given by the morphisms in $\mathcal{D}$.
This section concerns the case where $\mathcal{D}$ is the category $\mathsf{Cat}_{\mathfrak{Ass}^u}$ associated with the operad of unital associative algebras.
We recall the skeleton $\mathsf{gr}^{\mathsf{op}}$ of the opposite category of finitely generated free groups in section \ref{202208181550}.
The aim of this section is to recall Powell's assertion \cite{powell2021analytic} that $\mathsf{Cat}_{\mathfrak{Ass}^u}$ induces a bilinear functor ${}_\Delta \mathsf{Cat}_{\mathfrak{Ass}^u} : \mathsf{Cat}_{\mathfrak{Lie}}^{\mathsf{op}} \times \mathds{k} \mathsf{gr}^{\mathsf{op}} \to \mathsf{Mod}_{\mathds{k}}$ where $\mathsf{Cat}_{\mathfrak{Lie}}$ is the category associated with the Lie operad.
In \cite{powell2021analytic}, the ground ring is assumed to be a field of characteristic zero but the generalization to any commutative ring is straightforward as we sketch below.

\subsection{The operad of unital associative algebras} \label{202401101351}

Let $\mathds{k}$ be a commutative ring with unit.
For a small $\mathds{k}$-linear category $\mathcal{M}$, a {\it left (right, resp.) $\mathcal{M}$-module} is a $\mathds{k}$-linear functor from $ \mathcal{M}$ ($\mathcal{M}^{\mathsf{op}}$, resp.) to $\mathsf{Mod}_{\mathds{k}}$.
More generally, a {\it $(\mathcal{M} , \mathcal{N} )$-bimodule} is a $\mathds{k}$-bilinear functor $\mathcal{N}^\mathsf{op} \times \mathcal{M} \to \mathsf{Mod}_{\mathds{k}}$.

For a small category $\mathcal{D}$, let $\mathds{k} \mathcal{D}$ be the $\mathds{k}$-linear category obtained by the $\mathds{k}$-linearization of $\mathcal{D}$.
Then the category of $\mathds{k} \mathcal{D}$-modules and linear natural transformations is equivalent with $\mathcal{F} ( \mathcal{D} ; \mathds{k} )$ the category of $\mathcal{D}$-modules in the sense of section \ref{202208131840}.

For a $\mathds{k}$-linear operad $\mathfrak{O}$, denote by $\mathsf{Cat}_{\mathfrak{O}}$ the $\mathds{k}$-linear category associated with $\mathfrak{O}$: its objects consist of natural numbers and the morphism sets are given by
$$\mathsf{Cat}_{\mathfrak{O}} (n , m) = \bigoplus_{f: \underline{n} \to \underline{m}} \bigotimes_{i=1}^{m} \mathfrak{O} ( f^{-1} (i) ) $$
where $f : \underline{n} \to \underline{m}$ runs over all the maps.
Then the assignment $(n, m) \mapsto \mathsf{Cat}_{\mathfrak{O}} ( n, m )$ gives a $( \mathsf{Cat}_{\mathfrak{O}} , \mathsf{Cat}_{\mathfrak{O}})$-bimodule which we also denote by $\mathsf{Cat}_{\mathfrak{O}}$.

Let $\mathfrak{Ass}^u$ be the $\mathds{k}$-linear operad of unital associative algebras.
The module $\mathsf{Cat}_{\mathfrak{Ass}^u} ( n, m)$ is isomorphic to the free module generated by $m$-tuples $( f^{-1} (1) , o_1) \otimes ( f^{-1} (2) , o_2) \otimes \cdots ( f^{-1} (m) ,o_m)$ for maps $f : \underline{n} \to \underline{m}$ whose fibers $f^{-1} (i)$ are equipped with a total ordering $o_i$.
Fix a countable set of variables $\{ x_1, x_2, \cdots x_n , \cdots \}$.
For a finite set $S$ equipped with a total ordering $o$, we introduce the notation $x_{(S,o)} {:=} x_{a_1} x_{a_2} \cdots x_{a_k}$ (the {\it formal word}) if $S = \{ a_1, a_2, \cdots , a_k \}$ and the order $o$ is given by $a_i < a_{i+1}$.
If $S = \emptyset$, then $x_{(S,o)} {:=} 1$.
The replacement of $(f^{-1} (i) , o_i )$ by $x_{(f^{-1} (i) , o_i )}$ induces an isomorphism:
\begin{align}
\label{Korea202306142142}
\mathsf{Cat}_{\mathfrak{Ass}^u} ( n, m) \cong \bigoplus_{f: \underline{n} \to \underline{m}} \bigotimes_{i=1}^{m}  \mathds{k} x_{(f^{-1} (i) , o_i )} .
\end{align}

Let $A_n {:=} \mathds{k} \langle x_1 , \cdots , x_n \rangle$ be the free non-commutative polynomial algebra generated by $n$ variables $x_1 , \cdots , x_n$.
We define the $( \mathsf{Cat}_{\mathfrak{Ass}^u} , \mathsf{Cat}_{\mathfrak{Ass}^u})$-bimodule $\mathbb{A}$ by $\mathbb{A} ( n , m) {:=} A_n^{\otimes m}$.
To be precise, for each $n$, the left $\mathsf{Cat}_{\mathfrak{Ass}^u}$-action on $\mathbb{A}$ is induced by the algebra structure on $A_n$.
The right $\mathsf{Cat}_{\mathfrak{Ass}^u}$-action is induced by {\it substitution of variables by words}.
For example, $f = \mathbb{A} ( x_2 x_3 \otimes 1 \otimes x_4 x_1  , 1) : \mathbb{A} ( 3, 1) \to \mathbb{A} ( 4 , 1)$ is an algebra homomorphism characterized by $f(x_1) = x_2 x_3$, $f(x_2) = 1$ and $f(x_3) = x_4 x_1$; and $\mathbb{A} ( x_2 x_3 \otimes 1 \otimes x_4 x_1  , m) = f^{\otimes m}$ for arbitrary $m$.
The isomorphism in (\ref{Korea202306142142}) leads to a $( \mathsf{Cat}_{\mathfrak{Ass}^u} , \mathsf{Cat}_{\mathfrak{Ass}^u})$-bimodule embedding $\iota : \mathsf{Cat}_{\mathfrak{Ass}^u} \hookrightarrow \mathbb{A}$.

\subsection{The bimodule ${}_{\Delta} \mathsf{Cat}_{\mathfrak{Ass}^u}$}
\label{Korea202306141722}

In this section, we will see that using the embedding $\iota : \mathsf{Cat}_{\mathfrak{Ass}^u} \hookrightarrow \mathbb{A}$ we can define a $( \mathds{k} \mathsf{gr}^{\mathsf{op}} , \mathsf{Cat}_{\mathfrak{Lie}})$-bimodule structure on $\mathsf{Cat}_{\mathfrak{Ass}^u}$. 
For this, we start by constructing a $( \mathds{k} \mathsf{gr}^{\mathsf{op}} , \mathsf{Cat}_{\mathfrak{Lie}})$-bimodule $\mathbb{H}$ from the $( \mathsf{Cat}_{\mathfrak{Ass}^u} , \mathsf{Cat}_{\mathfrak{Ass}^u})$-bimodule $\mathbb{A}$.

Let $H_n$ be the Hopf algebra with $A_n$ the underlying algebra structure and the comultiplication given by the shuffle, i.e. $\Delta (x_i ) = x_i \otimes 1 + 1 \otimes x_i$.
In particular, $H_n$ is cocommutative.
We define the $( \mathds{k} \mathsf{gr}^{\mathsf{op}} , \mathsf{Cat}_{\mathfrak{Lie}})$-bimodule $\mathbb{H}$ by $\mathbb{H} ( n , m ) {:=} H_n^{\otimes m}$.
To be precise, for each $n$, the left $\mathds{k} \mathsf{gr}^{\mathsf{op}}$-action on $\mathbb{H}$ is induced by the Hopf algebra structure on $H_n$ (see Remark \ref{202207311722}).
The right $\mathsf{Cat}_{\mathfrak{Lie}}$-action is induced by {\it substitution of variables by Lie brackets}.
For example, $g = \mathbb{H} ( [ x_1 , x_2 ] \otimes x_3 \otimes \cdots \otimes x_n , 1 ) : H_{n-1} \to H_n$ is a Hopf homomorphism characterized by $g (x_1 ) = x_1 x_2 - x_2 x_1$ and $g(x_i) = x_{i+1}$ for $i \geq 2$.
The category $\mathsf{Cat}_{\mathfrak{Lie}}$ is generated by Lie brackets and permutations so that this example implies that $\mathbb{H}$ is well-defined.

The bimodule embedding $\iota : \mathsf{Cat}_{\mathfrak{Ass}^u} \hookrightarrow \mathbb{A}$ gives $\iota^\prime_{( n , m )} : {}_{\Delta} \mathsf{Cat}_{\mathfrak{Ass}^u} ( n , m ) \hookrightarrow \mathbb{H}( n , m )$ for $n, m \in \mathds{N}$ by definitions.
It turns out that the collection of $\iota^\prime_{( n , m )}$ is compatible with the $( \mathds{k} \mathsf{gr}^{\mathsf{op}} , \mathsf{Cat}_{\mathfrak{Lie}})$-bimodule structure on $\mathbb{H}$.
Indeed, the assignment $(n,m) \mapsto \iota^\prime_{n,m}$ preserves the left $\mathds{k} \mathsf{gr}^{\mathsf{op}}$-action (the right $\mathsf{Cat}_{\mathfrak{Lie}}$-action, resp.) since it is compatible with $\Delta, \nabla, \eta, \epsilon, s$ (the Lie brackets, resp.) and permutations (see section \ref{202208181550} for the notation).
We define the $( \mathds{k} \mathsf{gr}^{\mathsf{op}} , \mathsf{Cat}_{\mathfrak{Lie}})$-bimodule ${}_{\Delta} \mathsf{Cat}_{\mathfrak{Ass}^u}$ by the bimodule structure inherited from $\mathbb{H}$ via $\iota^\prime$.

\begin{prop}
\label{202312050902}
For $n \in \mathds{N}$, ${}_{\Delta}\mathsf{Cat}_{\mathfrak{Ass}^u} (n , - ) \in \mathcal{F} ( \mathsf{gr}^{\mathsf{op}} ; \mathds{k})$ has degree $n$.
\end{prop}
\begin{proof}
We first prove that $\mathrm{deg} ({}_{\Delta}\mathsf{Cat}_{\mathfrak{Ass}^u} (n , - )) \leq n$, i.e. $\mathrm{cr}^{n+1} ( {}_{\Delta}\mathsf{Cat}_{\mathfrak{Ass}^u} (n , - ) ) \cong 0$.
Recall the notation in section \ref{202208131840}.
It suffices to prove that $\chi_F ( [m_i]^{n+1}_{i=1} )= 0$ for an arbitrary sequence $[m_i]^{n+1}_{i=1}$ of objects in $\mathsf{gr}^{\mathsf{op}}$ and $F = {}_{\Delta}\mathsf{Cat}_{\mathfrak{Ass}^u} (n , - )$.
If $m_i = 0$ for some $i$, then it is clear that $\chi_F ( [m_i]^{n+1}_{i=1} )= 0$.
Suppose that $m_i >0$ for all $i$.
Let $N = \sum^{n+1}_{i=1} m_i $.
By definitions, the module ${}_{\Delta}\mathsf{Cat}_{\mathfrak{Ass}^u} (n , N )$ is generated by $\bigotimes^{N}_{j=1} w_j \in {}_{\Delta}\mathsf{Cat}_{\mathfrak{Ass}^u} (n , N )$ such that $w_{m_i+1} = w_{m_i+2} = \cdots = w_{m_{i+1}} = 1$ for some $i$ since $(n+1) \leq N$.
For such tensor products, we have $e_i \left( \bigotimes^{N}_{j=1} w_j \right) = 0$.
Hence, $\chi_F ( [m_i]^{n+1}_{i=1} )= 0$.

We now prove that $\mathrm{deg} ({}_{\Delta}\mathsf{Cat}_{\mathfrak{Ass}^u} (n , - )) \not\leq (n-1)$, equivalently, $\chi_F ( [ m_i ]^{n}_{i=1} ) \neq 0$ for some sequence $[ m_i ]^{n}_{i=1}$ of objects in $\mathsf{gr}^{\mathsf{op}}$.
Let $m_i = 1 \in \mathsf{gr}^{\mathsf{op}}$ for any $i$.
By definition, we have 
$$\left( \chi_F ( [ m_i ]^{n}_{i=1} ) \right) ( x_1 \otimes x_2 \otimes \cdots \otimes x_n ) = x_1 \otimes x_2 \otimes \cdots \otimes x_n . $$
\end{proof}

\subsection{Relation between $\mathsf{Cat}_{\mathfrak{Ass}^u}$ and ${}_{\Delta} \mathsf{Cat}_{\mathfrak{Ass}^u}$}
\label{Korea202306121601}

In this section, we study how the bimodules $\mathsf{Cat}_{\mathfrak{Ass}^u}$ and ${}_{\Delta} \mathsf{Cat}_{\mathfrak{Ass}^u}$ interact with each other.

We regard a formal word $w = x_{a_1} x_{a_2} \cdots x_{a_k}$ consisting of the variables $x_1, x_2, \cdots , x_n$ as an element $x_{a_1} x_{a_2} \cdots x_{a_k} \in \mathsf{F}_n$ when $k \geq 1$.
If the word $w$ is trivial, i.e. $w= 1$, then $\mathcal{E}$ assigns the unit $e \in \mathsf{F}_n$.
By extending this observation, we introduce the following functor:
\begin{Defn} \label{202401101407}
We define the linear functor 
$$\mathcal{E} : \mathsf{Cat}_{\mathfrak{Ass}^u} \hookrightarrow \mathds{k} \mathsf{gr}^{\mathsf{op}} .$$
It assigns $\mathcal{E} ( n ) {:=} n$ to an object $n \in \mathsf{Cat}_{\mathfrak{Ass}^u}$.
Let $w_1 \otimes w_2 \otimes \cdots \otimes w_m \in \mathsf{Cat}_{\mathfrak{Ass}^u} ( n, m)$ such that $w_1, w_2 , \cdots, w_m$ are formal words consisting of $x_1, x_2, \cdots x_n$ (see section \ref{202401101351}).
By using the notation of morphisms in $\mathsf{gr}^\mathsf{op}$ in section \ref{202208181550}, the morphism $\mathcal{E} ( w_1 \otimes w_2 \otimes \cdots \otimes w_m )$ from $n$ to $m$ in $\mathsf{gr}^{\mathsf{op}}$ is defined by
$$\mathcal{E} ( w_1 \otimes w_2 \otimes \cdots \otimes w_m ) {:=} [w_1| w_2| \cdots | w_m]_n$$
where we regard the formal words $w_1, w_2, \cdots, w_m$ as elements of $\mathsf{F}_n$ as above.
\end{Defn}

For example, we have $\mathcal{E} ( x_2 x_3 \otimes 1 \otimes x_4 x_1 ) = [x_2x_3 | e | x_4 x_1]_4$.

\begin{prop}
Definition \ref{202401101407} gives a well-defined faithful linear functor.
\end{prop}
\begin{proof}
We prove that the assignment $\mathcal{E}$ preserves the identity.
For an object $n \in \mathsf{Cat}_{\mathfrak{Ass}^u}$, the identity of $n$ is $x_1 \otimes x_2 \otimes \cdots \otimes x_n$ by definition.
The corresponding morphism $\mathcal{E} ( x_1 \otimes x_2 \otimes \cdots \otimes x_n ) = [x_1 | x_2 | \cdots | x_n]_n$ is the identity of $n$ in $\mathds{k} \mathsf{gr}^{\mathsf{op}}$.

We now prove the assignment $\mathcal{E}$ preserves the composition.
Let $f = (w_1 \otimes w_2 \otimes \cdots \otimes w_m) : n \to m$ and $f^\prime = (w_1^\prime \otimes w_2^\prime \otimes \cdots \otimes w_l^\prime) : m \to l$ be morphisms in the category $\mathsf{Cat}_{\mathfrak{Ass}^u}$; and put $f^\prime \circ f = w_1^{\prime\prime} \otimes w_2^{\prime\prime} \otimes \cdots \otimes w_l^{\prime\prime}$.
By definition, if we have $w_i^\prime = x_{a_1} x_{a_2} \cdots x_{a_k}$ for some $1 \leq i \leq l$, then we have $w_{a_1} w_{a_2} \cdots w_{a_k}  = w_i^{\prime\prime}$.
To prove that $\mathcal{E} ( f^\prime \circ f) = \mathcal{E} (f^\prime ) \circ \mathcal{E} (f)$, we compute $( \rho \circ \rho^\prime ) (x_i)$ for $1 \leq i \leq l$ where we regard $\rho {:=} \mathcal{E} (f)$ and $\rho^\prime {:=} \mathcal{E} (f^\prime )$ as group homomorphisms: if we have $w_i^\prime = x_{a_1} x_{a_2} \cdots x_{a_k}$, then we obtain
$$
( \rho \circ \rho^\prime ) (x_i) = \rho ( w_i^\prime ) = \rho ( x_{a_1} x_{a_2} \cdots x_{a_k} ) = w_{a_1} w_{a_2} \cdots w_{a_k} = w_i^{\prime\prime} .
$$
Hence, this proves $\mathcal{E} (f^\prime ) \circ \mathcal{E} (f) = \mathcal{E} ( f^\prime \circ f)$.

The functor $\mathcal{E}$ is faithful since it assigns the {\it formal word} basis of $\mathsf{Cat}_{\mathfrak{Ass}^u} ( n, m)$ to the basis $\mathsf{gr}^{\mathsf{op}} (n,m)$ of $\mathds{k}\mathsf{gr}^{\mathsf{op}} (n,m)$.
\end{proof}

On the one hand, the operad homomorphism $\mathfrak{Lie} \to \mathfrak{Ass}^u$ induces the functor:
$$\mathcal{B} : \mathsf{Cat}_{\mathfrak{Lie}} \to \mathsf{Cat}_{\mathfrak{Ass}^u}$$
By definitions, the diagram below commutes up to a natural isomorphism.
$$
\begin{tikzcd}
\mathsf{Cat}_{\mathfrak{Ass}^u}^{\mathsf{op}} \times \mathsf{Cat}_{\mathfrak{Ass}^u} \ar[r, "\mathsf{Cat}_{\mathfrak{Ass}^u}"] & \mathsf{Mod}_{\mathds{k}} \\
\mathsf{Cat}_{\mathfrak{Lie}}^{\mathsf{op}} \times \mathsf{Cat}_{\mathfrak{Ass}^u} \ar[u, "\mathcal{B} \times \mathrm{id}"] \ar[r, hookrightarrow, "\mathrm{id} \times \mathcal{E}"'] & \mathsf{Cat}_{\mathfrak{Lie}}^{\mathsf{op}} \times \mathds{k} \mathsf{gr}^{\mathsf{op}} \ar[u, "{}_{\Delta} \mathsf{Cat}_{\mathfrak{Ass}^u}"']
\end{tikzcd}
$$

\section{On a left ideal of $\mathds{kr}\mathsf{gr}^{\mathsf{op}}$}
\label{202312081800}

The functor $\mathcal{E} : \mathsf{Cat}_{\mathfrak{Ass}^u} \hookrightarrow \mathds{k} \mathsf{gr}^{\mathsf{op}}$ introduced in Definition \ref{202401101407} is not full.
For example, the morphism $\Delta : 1 \to 2$ is not in the image of $\mathcal{E}$.
The aim of this section is to define a certain left ideal $\mathcal{I}$ of $\mathds{kr}\mathsf{gr}^{\mathsf{op}}$ for which the composition $\mathsf{Cat}_{\mathfrak{Ass}^u} (n,m) \stackrel{\mathcal{E}}{\hookrightarrow} \mathds{k} \mathsf{gr}^{\mathsf{op}} (n,m)  \stackrel{\pi}{\to} \mathds{k} \mathsf{gr}^{\mathsf{op}} (n,m)   / \mathcal{I} (n,m) $ is surjective (see Theorem \ref{202312061920}).
In other words, all the morphisms in $\mathds{k}\mathsf{gr}^{\mathsf{op}}$ arises from $\mathcal{E}$ modulo $\mathcal{I}$.
We will use this to define the $\mathsf{Cat}_{\mathfrak{Lie}}$-module $\mathcal{P} (F)$ later (see Definition \ref{202312081139}).

In this section, we freely use the notations of morphisms in $\mathsf{gr}^{\mathsf{op}}$ following section \ref{202208181550}.
Note that we use the notation $[ \rho (x_1) | \cdots | \rho (x_m) ]_n$ to denote morphisms in the opposite category $\mathsf{gr}^{\mathsf{op}}$, not the category $\mathsf{gr}$.

In the following, we recall the notion of {\it semicategory}.
It is similar to that of category, so it has objects and morphisms with composition, but we do not require the identity morphisms.
\begin{Defn}
\label{202312062128}
Let $\theta: 1 \to 2$ be the morphism in $\mathds{k} \mathsf{gr}^{\mathsf{op}}$ defined by 
$$\theta {:=} [ x_1 | x_1 ]_1  - [ e | x_1 ]_1 - [ x_1 | e ]_1 .$$
We define $\mathcal{I} ( n , m)$ to be the submodule of $\mathds{k} \mathsf{gr}^{\mathsf{op}} ( n, m)$ generated by $f \circ (\mathrm{id}_{i-1} \ast \theta \ast \mathrm{id}_{n-i})$ for all $1 \leq i \leq n$ and $f \in \mathds{k} \mathsf{gr}^{\mathsf{op}} (n+1 , m)$.
We denote by $\mathcal{I}$ the subsemicategory of $\mathds{k} \mathsf{gr}^{\mathsf{op}}$ which consists of objects $\mathds{N}$ and morphisms $\mathcal{I} (n,m)$ for $n,m \in \mathds{N}$.
\end{Defn}

\begin{prop} \label{202312091726}
The subsemicategory $\mathcal{I}$ is a left ideal of $\mathds{k} \mathsf{gr}^{\mathsf{op}}$.
In other words, if $h \in \mathds{k} \mathsf{gr}^{\mathsf{op}} (m, l)$ and $g \in \mathcal{I} ( n, m)$, then $h \circ g \in \mathcal{I} ( n, l)$.
\end{prop}
\begin{proof}
The morphism $g$ is a linear combination of $f \circ (\mathrm{id}_{i-1} \ast \theta \ast \mathrm{id}_{n-i})$ for some $f$ and $i$.
The result follows from $(h \circ f) \circ (\mathrm{id}_{i-1} \ast \theta \ast \mathrm{id}_{n-i}) \in \mathcal{I} ( n, l)$.
\end{proof}

We investigate the quotient $\mathds{k} \mathsf{gr}^{\mathsf{op}} ( n, m) /\mathcal{I} (n,m)$ so as to define the primitive filtration in section \ref{202312091723}.
For $\rho,\rho^\prime \in \mathds{k} \mathsf{gr}^{\mathsf{op}} ( n,m)$, we write $\rho \equiv \rho^\prime$ if $\rho - \rho^\prime \in \mathcal{I} ( n,m)$.
In particular, for $\rho \in \mathds{k} \mathsf{gr}^{\mathsf{op}}(n+1, m)$, we have
\begin{align*}
    \rho \circ [x_1| \cdots |x_{i-1}|x_i |x_i | x_{i+1} | \cdots |x_n]_n &\equiv \\
    \rho \circ [x_1| \cdots |x_{i-1}|x_i |e | x_{i+1} | \cdots &|x_n]_n + \rho \circ [x_1| \cdots |x_{i-1}|e |x_i | x_{i+1} | \cdots |x_n]_n .
\end{align*}

\begin{Example}
For $[ x_1^2 ]_1 \in \mathds{k} \mathsf{gr}^{\mathsf{op}} ( 1, 1)$, we have $[x_1^2]_1 = [x_1 x_2]_2 \circ [x_1 | x_1]_1 \equiv [x_1 x_2]_2 \circ [e | x_1]_1 + [x_1 x_2]_2 \circ [ x_1| e]_1 = 2 [x_1]_1$.
Similarly, one may obtain $[x_1^k]_1 \equiv k [x_1]_1$.
\end{Example}

For $1 \leq i \leq n$, let $\mathsf{F}_n^{(i)}$ be the subgroup of $\mathsf{F}_n$ generated by $x_1, \cdots , x_{i-1} , x_{i+1} , \cdots , x_n$.

\begin{Lemma} \label{202307142220}
Let $n , m \in \mathds{N}$.
    \begin{enumerate}
        \item For $1 \leq i \leq n$, we have $[x_1|x_2| \cdots |x_{i-1} | x_{i+1} | \cdots | x_n]_n \equiv 0$.
        \item Let $w_1, w_2 , \cdots, w_m \in \mathsf{F}_n$.
        If there exists $1 \leq i \leq n$ such that $w_1, w_2, \cdots, w_m \in \mathsf{F}_n^{(i)}$, then $[w_1 | \cdots | w_m]_n \equiv 0$.
    \end{enumerate}
\end{Lemma}
\begin{proof}
Without loss of generality, we prove the first claim for $i=n$, i.e. $[x_1|x_2| \cdots | x_{n-1}]_n \equiv 0$.
This follows from 
\begin{align*}
& [x_1|x_2| \cdots | x_{n-1}]_n ,  \\
=& [x_1|x_2| \cdots |x_{n-1}]_{n+1} \circ [x_1 | x_2 | \cdots | x_{n-1} | x_n | x_n]_n , \\
\equiv& [x_1|x_2| \cdots |x_{n-1}]_{n+1} \circ [x_1 | x_2 | \cdots | x_{n-1} | x_n | e]_n + [x_1|x_2| \cdots |x_{n-1}]_{n+1} \circ [x_1 | x_2 | \cdots | x_{n-1} | e | x_n]_n , \\
=& [x_1|x_2| \cdots | x_{n-1}]_n + [x_1|x_2| \cdots | x_{n-1}]_n.
\end{align*}
Hence, we have $[x_1|x_2| \cdots | x_{n-1}]_n \equiv 0$.

We now prove the second claim.
Let $\rho : \mathsf{F}_m \to \mathsf{F}_n$ be the homomorphism such that $\rho (x_k) = w_k$ for $1 \leq k \leq m$; and $e_i : \mathsf{F}_{n-1} \to \mathsf{F}_n$ be the homomorphism such that
\begin{align*}
e_i (x_k) = \begin{cases}
		x_k & (k < i) , \\
		x_{k+1} & (k \geq i) . 
		\end{cases}
\end{align*}
By the assumption $w_1, w_2, \cdots, w_m \in \mathsf{F}_n^{(i)}$, the homomorphism $\rho$ factors as follows.
$$
\begin{tikzcd}
\mathsf{F}_m \ar[r, "\rho"] \ar[dr, "\exists \rho^\prime"'] & \mathsf{F}_n \\
& \mathsf{F}_{n-1} \ar[u, hookrightarrow, "e_i"']
\end{tikzcd}
$$
By definitions, this diagram is rephrased as 
$$[w_1 | \cdots | w_m ]_n = [\rho^\prime (x_1) | \cdots | \rho^\prime (x_m) ]_{n-1}  \circ [x_1|x_2| \cdots |x_{i-1} | x_{i+1} | \cdots | x_n]_n$$ 
in the category $\mathsf{gr}^{\mathsf{op}}$.
Hence, the second claim is immediate from the first claim and Proposition \ref{202312091726}.
\end{proof}

\begin{Lemma} \label{202307142238}
   Let $1 \leq i \leq n$.
   For $w_1, \cdots , w_k \in \mathsf{F}_n^{(i)}$ and $a_1, \cdots , a_k \in \mathds{Z}$, we have 
   $$[ w_1 x_i^{a_1} w_2 x_i^{a_2} \cdots w_k x_i^{a_k} ]_n \equiv \sum^{k}_{j=1} [ w_1 w_2 \cdots w_{j-1} ( w_j x_i^{a_j}) w_{j+1} w_{j+2} \cdots w_k ]_n .$$
\end{Lemma}
\begin{proof}
We prove the claim for $i=n$ by induction with respect to $k$.
Let $\rho = [ w_1 x_n^{a_1} w_2 x_{n+1}^{a_2} \cdots w_k x_{n+1}^{a_k} ]_{n+1}$.
\begin{align*}
[ w_1 x_n^{a_1} w_2 x_n^{a_2} \cdots w_k x_n^{a_k} ]_n &= \rho \circ [x_1|x_2| \cdots | x_{n-1} | x_n | x_n]_n , \\
&\equiv \rho \circ [x_1|x_2| \cdots | x_{n-1} | x_n | e]_n + \rho \circ [x_1|x_2| \cdots | x_{n-1} | e | x_n]_n , \\
&= [ w_1 x_n^{a_1} w_2 w_3 \cdots w_k ]_n + [ w_1 w_2 x_n^{a_2} w_3 x_n^{a_3} \cdots w_k x_n^{a_k} ]_n
\end{align*}
We now focus on the right summand $[ w_1 w_2 x_n^{a_2} w_3 x_n^{a_3} \cdots w_k x_n^{a_k} ]_n$.
We put $w^\prime_1 = w_1w_2$ and $w^\prime_{s} = w_{s+1}$ for $2 \leq s \leq (k-1)$; and apply the assumption of induction so as to complete the proof.
The proof for $i$ other than $n$ is straightforward by applying the permutation of variables.
\end{proof}

\begin{Lemma} \label{202307142237}
For  $\sigma \in \Sigma_n$, we have $[x_{\sigma(1)}^{a_1} x_{\sigma(2)}^{a_2} \cdots x_{\sigma(n)}^{a_n}]_n \equiv (a_1 a_2 \cdots a_n ) [x_{\sigma(1)} x_{\sigma(2)} \cdots x_{\sigma(n)}]_n$. 
\end{Lemma}
\begin{proof}
We consider the case that $\sigma$ is the identity, since the proof for the general case is similar.
Lemma \ref{202307142220} implies that, if $a_j = 0$ for some $j$, then $[x_1^{a_1} x_2^{a_2} \cdots x_n^{a_n}]_n \equiv 0$.
Let $1 \leq j \leq n$ such that $a_j \neq 0$.
For simplicity, let $j=n$.
Suppose $a_n > 0$.
For $w = x_1^{a_1} x_2^{a_2} \cdots x_{n-1}^{a_{n-1}}$, we have 
\begin{align*}
[wx_n^{a_n}]_n &= [w x_n^{a_n-1}x_{n+1}]_{n+1} \circ [x_1| \cdots |x_{n-1}|x_n|x_n]_n , \\
&\equiv [w x_n^{a_n-1}x_{n+1}]_{n+1} \circ [x_1| \cdots |x_{n-1}|x_n|e]_n + [w x_n^{a_n-1}x_{n+1}]_{n+1} \circ [x_1| \cdots |x_{n-1}|e|x_n]_n , \\
&= [wx_n^{a_n-1}]_n + [wx_n]_n
\end{align*}
Inductively, we obtain $[wx_n^{a_n}]_n = a_n [wx_n]_n$.
If $a_n < 0$, then we have $[wx_n^{a_n} x_n^{-a_n}]_n \equiv [wx_n^{a_n}]_n + [wx_n^{-a_n}]_n \equiv [wx_n^{a_n}]_n - a_n [w x_n]_n$.
By the second part of Lemma \ref{202307142220}, we have $[wx_n^{a_n}]_n = [w]_n \equiv 0$ so that the result follows.
\end{proof}

\begin{theorem}
\label{202312061920}
For $n,m \in \mathds{N}$, the composition of the maps below is surjective:
$$\mathsf{Cat}_{\mathfrak{Ass}^u} (n,m) \stackrel{\mathcal{E}}{\to} \mathds{k} \mathsf{gr}^{\mathsf{op}} (n,m) \stackrel{\pi}{\to} \mathds{k} \mathsf{gr}^{\mathsf{op}} (n,m) / \mathcal{I} (n,m) .$$
\end{theorem}
\begin{proof}
We first deal with the case $m=1$.
Let $\rho \in \mathds{k} \mathsf{gr}^{\mathsf{op}} ( n, 1)$.
Note that by the second part of Lemma \ref{202307142220} we have $\rho \equiv \sum^{r}_{j=1} a_j [w_j]_1$ where $w_j \in \mathsf{F}_n$ and $w_j \not\in \mathsf{F}_n^{(i)}$ for any $i$.
We iteratively apply Lemma \ref{202307142238} to the linear combination.
Then $\rho$ is a linear combination of $[x_{\sigma(1)}^{a_1} x_{\sigma(2)}^{a_2} \cdots x_{\sigma(n)}^{a_n}]_n$ up to $\mathcal{I} (n,1)$ where $\sigma$ is a permutation.
By Lemma \ref{202307142237}, such summand up to $\mathcal{I} (n,1)$ is induced by $\mathsf{Cat}_{\mathfrak{Ass}^u} (n,1)$ via $\mathcal{E}$.

We sketch the proof for general $m$.
Note that one can prove the formulae analogous to Lemma \ref{202307142238}, \ref{202307142237} even if there are some bars inserted between the words in the parenthesis $[~]$.
For example, by considering $m= 2$, we have
    \begin{align*}
        &[ w_1 x_i^{a_1} \cdots w_k x_i^{a_k} | w_{k+1} x_i^{a_{k+1}} \cdots w_l x_i^{a_l}]_n \\
        \equiv& \sum^{k}_{j=1} [ w_1 \cdots w_{j-1} ( w_j x_i^{a_j}) w_{j+1}  \cdots | \cdots w_l  ]_n + \sum^{l}_{j=k+1} [ w_1 \cdots | \cdots w_{j-1} ( w_j x_i^{a_j}) w_{j+1}  \cdots w_k ]_n .
    \end{align*}
This holds even if there are more than one bar;
the proof is parallel.
Similarly, for instance, we have
$$[x_1^{a_1} x_2^{a_2} \cdots x_k^{a_k} | x_{k+1}^{a_{k+1}} \cdots x_n^{a_n}]_n \equiv (a_1 a_2 \cdots a_n ) [x_1 x_2 \cdots x_k | x_{k+1} \cdots x_n]_n .$$
By using these extended formulae, one can prove our assertion as above.
\end{proof}

%%%%%%%%%%%%%%%
\section{Application to exponential $\mathsf{gr}^{\mathsf{op}}$-modules}
\label{202211081004}

The opposite category of finitely generated free groups $\mathsf{gr}^{\mathsf{op}}$ is a small category which satisfies the assumptions of $\mathcal{C}$.
Let $\mathds{k}$ be a commutative unital ring unless otherwise specified.
By \cite{Pirashvili} (see also \cite[Remark 2.1, Theorem 2.2]{MR3765469} and \cite{Hab}), the evaluation functor $\mathrm{ev}_{1} : \mathcal{F}^\mathsf{exp} ( \mathsf{gr^{\mathsf{op}}} ; \mathds{k} ) \to \mathsf{Hopf}^\mathsf{cc}_\mathds{k}$ induces an equivalence of categories (see Remark \ref{202207311722}).
In this section, we give an explicit description of some subcategories of exponential $\mathsf{gr}^{\mathsf{op}}$-modules under this equivalence.

\begin{remark}
\label{202207311722}
We sketch a construction of $\mathrm{ev}_{1}$ here.
Recall the notation in section \ref{202208181550}.
Let $F \in \mathcal{F}^\mathsf{exp} ( \mathsf{gr^{\mathsf{op}}} ; \mathds{k} )$.
Then the underlying module of the Hopf algebra $\mathrm{ev}_1 (F)$ is $F(1)$.
The multiplication, unit, comultiplication, counit and antipode are given as follows respectively:
\begin{align*}
\nabla_{\mathrm{ev}_1 (F)} : & F(1) \otimes F(1) \cong F(2) \stackrel{F(\nabla)}{\to} F(1) , \\ 
\eta_{\mathrm{ev}_1 (F)} : & \mathds{k} \cong F(0) \stackrel{F(\eta)}{\to} F(1) , \\
\Delta_{\mathrm{ev}_1 (F)} : & F(1) \stackrel{F(\Delta)}{\to} F(2) \cong F(1) \otimes F(1) , \\
\varepsilon_{\mathrm{ev}_1 (F)} : & F(1) \stackrel{F(\epsilon)}{\to} F(0) \cong \mathds{k} , \\
S_{\mathrm{ev}_1 (F)} : & F(1) \stackrel{F(\gamma)}{\to} F(1) .
\end{align*}

The quasi-inverse $\alpha : \mathsf{Hopf}^\mathsf{cc}_\mathds{k} \to \mathcal{F}^\mathsf{exp}( \mathsf{gr^{\mathsf{op}}} ; \mathds{k} )$ is given as follows.
For an object $H \in \mathsf{Hopf}^\mathsf{cc}_\mathds{k}$, we set $\left( \alpha (H) \right) ( n ) = U(H)^{\otimes n}$ where $U(H)$ is the underlying module of $H$.
The cocommutative Hopf algebra structure on $H$ induces a contravariant assignment of morphisms.
\end{remark}

\subsection{Analytic exponential $\mathsf{gr^{\mathsf{op}}}$-modules}
\label{202207211512}

In this section, we give an equivalence between conilpotent cocommutative Hopf algebras, and analytic exponential $\mathsf{gr}^{\mathsf{op}}$-modules.

Let $\mathsf{Hopf}^\mathsf{cc,conil}_\mathds{k}$ be the full subcategory of $\mathsf{Hopf}^\mathsf{cc}_\mathds{k}$ consisting of conilpotent cocommutative Hopf algebras over $\mathds{k}$.

Let $\mathcal{F}^\mathsf{exp}_\omega ( \mathsf{gr^{\mathsf{op}}} ; \mathds{k} )$ be the full subcategory of $\mathcal{F}^\mathsf{exp} ( \mathsf{gr^{\mathsf{op}}} ; \mathds{k} )$ consisting of analytic $\mathsf{gr^{\mathsf{op}}}$-modules.

\begin{theorem}
\label{202207212109}
The functor $\mathrm{ev}_{1}$ induces an equivalence of categories:
$$
\mathcal{F}^\mathsf{exp}_\omega ( \mathsf{gr^{\mathsf{op}}} ; \mathds{k} ) \simeq \mathsf{Hopf}^\mathsf{cc,conil}_\mathds{k} .
$$
\end{theorem}
\begin{proof}
We apply Corollary \ref{202207211158} to prove our claim.
By Corollary \ref{202207211158}, the restriction of $\mathrm{ev}_{1}$ to $\mathcal{F}^\mathsf{exp}_\omega ( \mathsf{gr^{\mathsf{op}}} ; \mathds{k})$ factors through $\mathsf{Hopf}^\mathsf{cc,conil}_\mathds{k} \hookrightarrow \mathsf{Hopf}^\mathsf{cc}_\mathds{k}$.
Denote the factorized functor by $\mathrm{ev}_{1}^\prime : \mathcal{F}^\mathsf{exp}_\omega ( \mathsf{gr^{\mathsf{op}}} ; \mathds{k} ) \to \mathsf{Hopf}^\mathsf{cc,conil}_\mathds{k}$.
Consider $\alpha$ introduced in Remark \ref{202207311722}.
On the one hand, the composition $\mathsf{Hopf}^\mathsf{cc,conil}_\mathds{k} \hookrightarrow \mathsf{Hopf}^\mathsf{cc}_\mathds{k} \stackrel{\alpha}{\to} \mathcal{F}^\mathsf{exp}( \mathsf{gr^{\mathsf{op}}} ; \mathds{k} )$ factors through the embedding $\mathcal{F}^\mathsf{exp}_\omega ( \mathsf{gr^{\mathsf{op}}} ; \mathds{k} ) \hookrightarrow \mathcal{F}^\mathsf{exp}( \mathsf{gr^{\mathsf{op}}} ; \mathds{k} )$.
In fact, the functor $\alpha (H)$ assigned to a conilpotent cocommutative Hopf algebra $H$ is analytic since $(\alpha (H))(n) = H^{\otimes n}$ is conilpotent for any $n$, by Proposition \ref{202207281554}.
Denote the factorized functor by $\alpha^\prime : \mathsf{Hopf}^\mathsf{cc,conil}_\mathds{k} \to \mathcal{F}^\mathsf{exp}_\omega ( \mathsf{gr^{\mathsf{op}}} ; \mathds{k} )$.
Then $\mathrm{ev}_{1}^\prime$ and $\alpha^\prime$ induce an equivalence of categories since so do the functors $\mathrm{ev}_{1}$ and $\alpha$.
\end{proof}

\subsection{Outer exponential $\mathsf{gr}^{\mathsf{op}}$-modules}

In this section, we prove an equivalence between bicommutative Hopf algebras and outer exponential $\mathsf{gr}^{\mathsf{op}}$-modules.
An {\it outer $\mathsf{gr}^{\mathsf{op}}$-module} is a functor $F \in \mathcal{F} (\mathsf{gr}^{\mathsf{op}} ; \mathds{k} )$ such that the induced action of inner automorphisms on $F ( n )$ is trivial for any $n$.
The notion of outer $\mathsf{gr}^{\mathsf{op}}$-module is introduced in \cite{PV}, motivated by the study of representations of $\mathrm{Out} (\mathsf{F}_n)$ coming from Hochschild homology associated to a wedge of circles.

Let $\mathsf{ab}$ be the category consisting of the free abelian groups $\mathds{Z}^n$ for $n \in \mathds{N}$ and group homomorphisms.
The evaluation $\mathrm{ev}_\mathds{Z} : \mathcal{F}^\mathsf{exp} ( \mathsf{ab} ; \mathds{k} ) \to \mathsf{Hopf}^\mathsf{bc}_\mathds{k}$ at $\mathds{Z} \in \mathsf{ab}$ induces an equivalence of categories \cite{touze}.
The assignment of the dual $X^\vee {:=} \mathrm{Hom} ( X , \mathds{Z} )$ to $X \in \mathsf{ab}$ induces an equivalence of categories $(-)^\vee : \mathsf{ab}^{\mathsf{op}} \to \mathsf{ab}$.
We denote by $\mathrm{ev}_\mathds{Z} : \mathcal{F}^\mathsf{exp} ( \mathsf{ab}^{\mathsf{op}} ; \mathds{k} ) \to \mathsf{Hopf}^\mathsf{bc}_\mathds{k}$ their composition with a slight abuse of notation.
Especially, $\mathrm{ev}_\mathds{Z}$ gives an equivalence of categories with a commutative diagram below where $\mathfrak{a} : \mathsf{gr} \to \mathsf{ab}$ is the abelianization functor:
\begin{equation}
\label{202211071428}
\begin{tikzcd}
\mathcal{F}^\mathsf{exp} ( \mathsf{ab}^{\mathsf{op}} ; \mathds{k} ) \ar[d, "\mathfrak{a}^\ast"] \ar[r, "\mathrm{ev}_\mathbb{Z}"] & \mathsf{Hopf}^\mathsf{bc}_\mathds{k} \ar[d, hookrightarrow]  \\
\mathcal{F}^\mathsf{exp} ( \mathsf{gr}^{\mathsf{op}} ; \mathds{k} ) \ar[r, "\mathrm{ev}_{1}"]  & \mathsf{Hopf}^\mathsf{cc}_\mathds{k}
\end{tikzcd}
\end{equation}

Let $\mathcal{F}^\mathsf{exp}_\mathsf{out} ( \mathsf{gr}^{\mathsf{op}} ; \mathds{k} )$ be the full subcategory of $\mathcal{F}^\mathsf{exp} ( \mathsf{gr}^{\mathsf{op}} ; \mathds{k} )$ consisting of outer $\mathsf{gr}^{\mathsf{op}}$-modules.
\begin{theorem}
\label{202207311733}
The functor $\mathfrak{a}^\ast$ induces an equivalence of categories:
$$
\mathcal{F}^\mathsf{exp} ( \mathsf{ab}^{\mathsf{op}} ; \mathds{k} )
\simeq
\mathcal{F}^\mathsf{exp}_\mathsf{out} ( \mathsf{gr}^{\mathsf{op}} ; \mathds{k} ) 
$$
In particular, the functor $\mathrm{ev}_{1}$ induces an equivalence of categories:
$$
\mathcal{F}^\mathsf{exp}_\mathsf{out} ( \mathsf{gr}^{\mathsf{op}} ; \mathds{k} ) 
\simeq
\mathsf{Hopf}^\mathsf{bc}_{\mathds{k}}.
$$
\end{theorem}
\begin{proof}
By (\ref{202211071428}), $\mathfrak{a}^\ast$ is a faithfully full embedding.
It is clear that the induced automorphism $\mathfrak{a} (g)$ is trivial for any inner automorphism $g$.
Hence, the essential image of $\mathfrak{a}^\ast$ is contained in $\mathcal{F}^\mathsf{exp}_\mathsf{out} ( \mathsf{gr}^{\mathsf{op}} ; \mathds{k} )$.
It suffices to prove that every $F \in \mathcal{F}^\mathsf{exp}_\mathsf{out} ( \mathsf{gr}^{\mathsf{op}} ; \mathds{k} )$ arises from a precomposition with $\mathfrak{a}$.
Equivalently, we prove that the cocommutative Hopf algebra $H = \mathrm{ev}_{1} (F)$ is commutative.
For the group presentation $\mathsf{F}_2 = \langle x_1 , x_2 \rangle$, consider the inner automorphism $\rho : \mathsf{F}_2 \to  \mathsf{F}_2$ induced by $x_1 \in \mathsf{F}_2$, i.e. $\rho ( x_i ) = x_1 x_i x_1^{-1}$.
We denote by $S$ the antipode on $H$, and freely use the Sweedler notation of comultiplication below.
For $x,y\in H$, one can calculate as follows by definition:
\begin{align*}
( F(\rho) )( x \otimes y ) 
= x_{(1)} \otimes x_{(2)} y S( x_{(3)} ) . 
\end{align*}
By the hypothesis that $F$ is outer, we have $( F(\rho) )( x \otimes y ) = x \otimes y$.
This gives $x \otimes y = x_{(1)} \otimes x_{(2)} y S( x_{(3)} )$.
Hence, for any $x,y \in H$, we have $\varepsilon (x) y = \varepsilon ( x_{(1)} ) x_{(2)} y S(x_{(3)} ) = x_{(1)} y S(x_{(2)} )$.
This result leads to the following calculation for $x,y \in H$:
\begin{align*}
yx &= \varepsilon  (x_{(1)} ) y x_{(2)} , \\
&= x_{(1)} y S(x_{(2)} ) x_{(3)} , \\
&= x_{(1)} y \varepsilon ( x_{(2)} ) = xy .
\end{align*}
\end{proof}

\subsection{Analytic $\mathsf{gr}^{\mathsf{op}}$-modules and Lie algebras}
\label{202301131738Japan}

In this section, we apply the previous results to prove an equivalence between some exponential $\mathsf{gr}^{\mathsf{op}}$-modules and Lie algebras.
The proof depends heavily on classical structure theorems for bialgebras (e.g. see  \cite[Theorem 1.3.4]{loday2012algebraic}).

In this section, we assume that the ground ring $\mathds{k}$ is a field of characteristic zero.
Let $\mathsf{Lie}_\mathds{k}$ be the category of Lie algebras over $\mathds{k}$ and Lie homomorphisms.
Then the assignment of primitive elements to Hopf algebras induce an equivalence of categories \cite[Section 5]{MM}:
\begin{equation}
\label{202211102107}
P : \mathsf{Hopf}^\mathsf{cc,conil}_\mathds{k} \stackrel{\simeq}{\longrightarrow} \mathsf{Lie}_\mathds{k}.
\end{equation}
The quasi-inverse is given by the universal enveloping algebra construction $U$.
This equivalence leads to the following corollary.

Let $\mathcal{F}^\mathsf{exp}_{\omega, \mathsf{out}} ( \mathsf{gr^{\mathsf{op}}} ; \mathds{k} )$ be the full subcategory of $\mathcal{F}^\mathsf{exp} ( \mathsf{gr}^{\mathsf{op}} ; \mathds{k} )$ consisting of analytic and outer $\mathsf{gr}^{\mathsf{op}}$-modules.

Let $\mathsf{Lie}^\mathsf{ab}_\mathds{k}$ be the full subcategory of $\mathsf{Lie}_{\mathds{k}}$ consisting of abelian Lie algebras.

\begin{Corollary}
\label{202211071547}
If the characteristic of the ground field $\mathds{k}$ is zero, then the composition of the evaluation on $1 \in \mathsf{gr}^{\mathsf{op}}$ with the functor $P$ induces equivalences of categories:
\begin{enumerate}
\item
$\mathcal{F}^\mathsf{exp}_\omega ( \mathsf{gr^{\mathsf{op}}} ; \mathds{k} ) \simeq \mathsf{Lie}_\mathds{k}$.
\item
$\mathcal{F}^\mathsf{exp}_{\omega, \mathsf{out}} ( \mathsf{gr^{\mathsf{op}}} ; \mathds{k} )  \simeq \mathsf{Lie}^\mathsf{ab}_\mathds{k} $.
\end{enumerate}
\end{Corollary}
\begin{proof}
The first claim is immediate from Theorem \ref{202207212109}.
The second one follows from the second part of Theorem \ref{202207311733}.
\end{proof}

\begin{remark}
Let $\mathfrak{Lie}$ be the Lie operad.
In \cite{powell2021analytic}, Powell constructs an equivalence between analytic $\mathsf{gr}^{\mathsf{op}}$-modules and $\mathfrak{Lie}$-modules of characteristic of zero.
In particular, it yields an embedding $\mathsf{Lie}_\mathds{k} \hookrightarrow \mathcal{F}_{\mathfrak{Lie}} \simeq \mathcal{F}_\omega ( \mathsf{gr}^{\mathsf{op}} ; \mathds{k})$.
The equivalences in Corollary \ref{202211071547} are compatible with the embedding.
\end{remark}

Let $\mathcal{F}^\mathsf{exp}_{< \infty} ( \mathsf{gr}^{\mathsf{op}} ; \mathds{k} )$ be the full subcategory of $\mathcal{F}^\mathsf{exp} ( \mathsf{gr}^{\mathsf{op}} ; \mathds{k} )$ consisting of $\mathsf{gr}^{\mathsf{op}}$-modules with finite degrees.

\begin{Corollary}
If the characteristic of the ground field $\mathds{k}$ is zero, then the category $\mathcal{F}^\mathsf{exp}_{< \infty} ( \mathsf{gr}^{\mathsf{op}} ; \mathds{k} )$ is equivalent with the one point category.
\end{Corollary}
\begin{proof}
Let $F$ be an exponential $\mathsf{gr}^{\mathsf{op}}$-module with finite degree $n \geq 0$.
In particular, the functor $F$ is analytic and thus, $F$ is induced by a Lie algebra $\mathfrak{g}$ by the first part of Corollary \ref{202211071547}.
By applying Theorem \ref{202207141636} to $X = 1$, we obtain $\left( P_n (F) \right) (X) \cong P_n ( U ( \mathfrak{g} ) )$.
We have $U ( \mathfrak{g} ) = P_n ( U ( \mathfrak{g} )) $ due to the degree assumption on $F$.
If $\mathfrak{g} \not\cong 0$, then we choose a nonzero $x \in \mathfrak{g}$.
By Poincar\'e-Birkhoff-Witt, we have $0 \neq x^{n+1} \in U ( \mathfrak{g} )$.
On the other hand, $\delta^{n+1} (x^{n+1}) = (n+1)! x^{\otimes (n+1)} \neq 0$ since the characteristic is zero.
Hence, $x \not\in P_n ( U ( \mathfrak{g} )) = U(\mathfrak{g})$ which leads to a contradiction so that we have $\mathfrak{g} \cong 0$.
Thus, the equivalence in the first part of Corollary \ref{202211071547} reduces to $\mathcal{F}^\mathsf{exp}_{< \infty} ( \mathsf{gr}^{\mathsf{op}} ; \mathds{k} ) \simeq \ast$.
\end{proof}

%%%%%%%%%%%%%%%%%%%%%%%%
\section{Primitive filtration of $\mathsf{gr}^{\mathsf{op}}$-modules}
\label{202211081022}

In this section, we introduce the primitive filtration of $\mathsf{gr}^{\mathsf{op}}$-modules.
For exponential $\mathsf{gr}^{\mathsf{op}}$-modules, it reproduces the primitive filtration of cocommutative Hopf algebras.
The primitive filtration is induced by the primitive part (see Definition \ref{202301131504}) which turns out to have an action from the Lie operad $\mathfrak{Lie}$.
It generalizes the Lie algebra structure on the primitive part of Hopf algebras.

\subsection{Primitive filtration}
\label{202312091723}

In this section, we define the {\it primitive filtration} of $\mathsf{gr}^{\mathsf{op}}$-modules.
Let $\mathds{k}$ be a commutative unital ring unless otherwise specified.
The primitive filtration of $F \in \mathcal{F} ( \mathsf{gr}^{\mathsf{op}} ; \mathds{k} )$ is constructed from a $\Sigma$-module $\mathcal{P} ( F) \in \mathcal{F} ( \Sigma ; \mathds{k} )$ which is induced by the first step of the polynomial filtration related with $F$.
Moreover, we compare the polynomial filtration in section \ref{202211251500} and the primitive filtration.

Recall the cross effect functor in section \ref{202208131840}.
By using the explicit construction of polynomial approximation in Definition \ref{202207202244}, we define 
$$\widetilde{P}_1 ( F ) {:=} P_1 ( \mathrm{cr}^1 (F) ) , ~ F \in \mathcal{F} ( \mathsf{gr}^{\mathsf{op}} ; \mathds{k} ) .$$

In the following, recall that we are identifying natural numbers with the objects of $\mathsf{gr}^{\mathsf{op}}$ so that the product in the category $\mathsf{gr}^{\mathsf{op}}$ is realized as the sum of natural numbers.

\begin{Defn}
Let $F \in \mathcal{F} ( \mathsf{gr}^{\mathsf{op}} ; \mathds{k} )$. 
\begin{itemize}
\item
Let $[X_i]^{n-1}_{i=1}$ be a sequence  of objects in $\mathsf{gr}^{\mathsf{op}}$.
For $1 \leq m \leq n$, we define $\hat{F}_m ( [X_i]^{n-1}_{i=1} ) \in \mathcal{F} ( \mathsf{gr}^{\mathsf{op}} ; \mathds{k} )$ by $\left( \hat{F}_m ( [X_i]^{n-1}_{i=1} ) \right) ( Y ) {:=} F( \sum^{m-1}_{i=1} X_i  + Y + \sum^{n-1}_{i=m} X_i )  $ for an object $Y$, and $\left( \hat{F}_m ( [X_i]^{n-1}_{i=1} ) \right) (f) {:=} F( \Asterisk^{m-1}_{i=1} \mathrm{id}_{X_i} \ast f \ast \Asterisk^{n-1}_{i=m} \mathrm{id}_{X_i} )$ for a morphism $f$.
\item
For a sequence $[X_i]^{n}_{i=1}$, we define $R_n (F) \in \mathcal{F} ( \left( \mathsf{gr}^{\mathsf{op}}\right)^{\times n} ; \mathds{k} )$ by
$$
\left( R_n (F) \right) \left( [X_i]^{n}_{i =1} \right) {:=} \bigcap^{n}_{i=1} \widetilde{P}_1 \left( \hat{F}_i ( X_1, \cdots, X_{i-1}, X_{i+1} , \cdots , X_n ) \right) (X_i ) .
$$
\end{itemize}
\end{Defn}

For $[X_i]^{n}_{i=1}$ a sequence  of objects in $\mathsf{gr}^{\mathsf{op}}$, the symmetry on $\mathsf{gr}^{\mathsf{op}}$ induces a $\Sigma_n$-action: $\sigma \in \Sigma_n$ induces an isomorphism
$$
F(X_1 + X_2 + \cdots + X_n ) \to F(X_{\sigma^{-1}(1)} + X_{\sigma^{-1}(2)} + \cdots + X_{\sigma^{-1}(n)} ) .
$$
The restriction to $\left( R_n (F) \right) \left( [X_i]^{n}_{i =1} \right) \subset F(X_1 + X_2 + \cdots + X_n )$ induces an isomorphism
$$
\left( R_n (F) \right) \left( [X_i]^{n}_{i =1} \right) \to \left( R_n (F) \right) \left( [X_{\sigma^{-1} (i)}]^{n}_{i =1} \right) .
$$
By substitution of $X_i = 1$ for any $i$, we define the following.

\begin{Defn}
\label{202301131504}
We define $\mathcal{P} ( F) \in \mathcal{F} ( \Sigma ; \mathds{k} )$ {\it the primitive part of $F$} by 
$$\left( \mathcal{P} ( F) \right) (n) := \left( R_n (F) \right) ( \overbrace{1, 1, \cdots, 1}^n )  . $$
\end{Defn}

\begin{remark}
It turns out that $\mathcal{P} ( F)$ can be enhanced to a $\mathsf{Cat}_{\mathfrak{Lie}}$-module (see section \ref{202212202131Japan}). 
\end{remark}

\begin{Example}
\label{202212160937}
Let $\mathds{k}$ be a field.
If $F$ is exponential, then there exists a cocommutative Hopf algebra $H$ such that $F(n) = H^{\otimes n}$ (explicitly, $F = \alpha (H)$ in Remark \ref{202207311722}).
In this case, we have a natural isomorphism $\left( \mathcal{P} (F) \right) (n) \cong \mathrm{Prim} (H)^{\otimes n}$.
This will be proved in Lemma \ref{Korea202306161806}.
\end{Example}

%
%\begin{Lemma}
%\label{202312050615}
%Let $F,G \in \mathcal{F} ( \mathsf{gr}^{\mathsf{op}} ; \mathds{k} )$.
%For $\Phi$ a natural transformation from $\mathcal{E}^\ast G$ to $\mathcal{E}^\ast F$, $\Phi$ gives a natural transformation from $G$ to $F$ if and only if we have $F (f) \circ \Phi_m = \Phi_l \circ G (f)$ where  $f : m \to l$ is one of the following morphisms:
%\begin{enumerate}
%\item 
%$\Delta \ast \mathrm{id}_{m-1} : m \to (m+1)$ for $m \in \mathds{N}$.
%\item
%$\epsilon \ast \mathrm{id}_{m-1} :m \to (m-1)$ for $m \in \mathds{N}$.
%\item
%$\gamma \ast \mathrm{id}_{m-1}: m \to m$ for $m \in \mathds{N}$.
%\end{enumerate}
%\end{Lemma}
%\begin{proof}
%It is obvious that if $\Phi$ gives a natural transformation from $G$ to $F$, then we have $F (f) \circ \Phi_n = \Phi_m \circ G (f)$ for the above $f$'s.
%We prove the converse.
%Note that the category $\mathsf{gr}^{\mathsf{op}}$ is generated by $\Delta, \nabla, \gamma, \eta, \epsilon$ and permutations under the free product of groups.
%Hence, we should check the naturality with respect to coproducts of such morphisms.
%By definitions, the morphisms $\nabla, \eta$ in $\mathsf{gr}^{\mathsf{op}}$ are induced by some morphisms of $\mathsf{Cat}_{\mathfrak{Ass}^u}$ via $\mathcal{E}$: $\nabla = \mathcal{E} (x_1 x_2)$, $\eta = \mathcal{E} (1)$.
%The permutations are also induced via $\mathcal{E}$.
%Hence, we only need to prove the naturality with respect to the above three types of morphisms (which are not induced via $\mathcal{E}$).
%\end{proof}

Recall the notation in Definition \ref{202312062128}.

\begin{prop}
\label{202312062129}
For $n \in \mathds{N}$, we have
$$
\left( \mathcal{P} ( F) \right) (n) = \bigcap^{n}_{i = 1} \mathrm{Ker} \left( F (\mathrm{id}_{i-1} \ast \theta \ast \mathrm{id}_{n-i}) \right) .
$$
\end{prop}
\begin{proof}
For $X_1 = X_2 = \cdots =X_n = 1 \in \mathsf{gr}^{\mathsf{op}}$, we have 
$$\widetilde{P}_1 \left( \hat{F}_i ( X_1, \cdots, X_{i-1}, X_{i+1} , \cdots , X_n ) \right) (X_i ) = \mathrm{Ker} \left( F (\mathrm{id}_{i-1} \ast \theta \ast \mathrm{id}_{n-i}) \right)$$
by definition.
\end{proof}

The category $\mathcal{F} ( \mathsf{gr}^{\mathsf{op}} ; \mathds{k} )$ is equivalent with the category of $\mathds{k}$-linear functors from $\mathds{k} \mathsf{gr}^{\mathsf{op}}$ to $\mathsf{Mod}_{\mathds{k}}$ via the linear extension.
Based on this observation, we regard $F \in \mathcal{F} ( \mathsf{gr}^{\mathsf{op}} ; \mathds{k} )$ as the corresponding linear functor $F : \mathds{k} \mathsf{gr}^{\mathsf{op}} \to \mathsf{Mod}_{\mathds{k}}$.
Recall the functor $\mathcal{E} : \mathsf{Cat}_{\mathfrak{Ass}^u} \hookrightarrow \mathds{k} \mathsf{gr}^{\mathsf{op}}$ in Definition \ref{202401101407}.
We define
$$\mathcal{E}^\ast F {:=} F \circ \mathcal{E} .$$ 
For $n \in \mathds{N}$, the structure of $\mathcal{E}^\ast F$ yields a $\mathds{k}$-linear natural transformation, 
$$
J_n : \mathsf{Cat}_{\mathfrak{Ass}^u} (n , - ) \otimes F(n) \to \mathcal{E}^\ast F 
;~ f \otimes v  \mapsto \left(( \mathcal{E} ^\ast F) (  f ) \right) (v) .
$$
By applying the tensor-hom adjunction, we obtain a homomorphism
$$
\check{J}_n : F(n) \to  \mathrm{Nat} \left( \mathsf{Cat}_{\mathfrak{Ass}^u} (n , - )  , \mathcal{E}^\ast F \right) .
$$

In the following statements, we recall the notation $(x_1 \otimes \cdots \otimes x_n) \in \mathsf{Cat}_{\mathfrak{Ass}^u} ( n , n )$ of the identity and the generators $\Delta, \nabla, \eta, \epsilon, \gamma$ of the category $\mathsf{gr}^{\mathsf{op}}$ (see section \ref{202212161644}).
\begin{theorem}
\label{202212161526}
Let $n \in \mathds{N}$.
For $F \in \mathcal{F} ( \mathsf{gr}^{\mathsf{op}} ; \mathds{k} )$, we consider $\check{J}_n$ defined as above.
The restriction of $\check{J}_n$ along the inclusion $\left( \mathcal{P} (F) \right)(n) \subset F (n)$ factors through $\mathrm{Nat}  \left( {}_\Delta \mathsf{Cat}_{\mathfrak{Ass}^u} (n , - )  ,  F \right)$.
In other words, there exists a homomorphism $\check{I}_n :  (\mathcal{P}(F)) (n) \to \mathrm{Nat}  \left( {}_\Delta \mathsf{Cat}_{\mathfrak{Ass}^u} (n , - )  ,  F \right)$ making the following diagram commutative:
$$
\begin{tikzcd}
 F(n) \ar[r, "\check{J}_n"] &  \mathrm{Nat} \left( \mathsf{Cat}_{\mathfrak{Ass}^u} (n , - )  , \mathcal{E}^\ast F \right)  \\
 (\mathcal{P}(F)) (n) \ar[r, "\check{I}_n"] \ar[u, hookrightarrow] & \mathrm{Nat}  \left( {}_\Delta \mathsf{Cat}_{\mathfrak{Ass}^u} (n , - )  ,  F \right) . \ar[u, hookrightarrow]
\end{tikzcd}
$$
\end{theorem}
\begin{proof}
Let $v \in \left( \mathcal{P} (F) \right)(n)$.
Consider $\Phi$ the $\mathds{k}$-linear natural transformation $\check{J}_n (v)$ from $\mathsf{Cat}_{\mathfrak{Ass}^u} (n , -)$ to $\mathcal{E}^\ast F$.
In particular, we have
$$\Phi_n (x_1 \otimes x_2 \cdots \otimes x_n) = v . $$
We prove that $\Phi$ gives a natural transformation from $K_n$ to $F$.
Put $K_n = {}_{\Delta} \mathsf{Cat}_{\mathfrak{Ass}^u} (n , -) \in \mathcal{F} ( \mathsf{gr}^{\mathsf{op}} ; \mathds{k} )$.
Note that $\mathcal{E}^\ast K_n =  \mathsf{Cat}_{\mathfrak{Ass}^u} (n , -)$.
Let $f : m \to l$ be a morphism in $\mathds{k} \mathsf{gr}^{\mathsf{op}}$.
It suffices to prove that the diagram below commutes:
$$
\begin{tikzcd}
K_n (m) \ar[r, "\Phi_m"] \ar[d, "K_n(f)"] & F(m) \ar[d, "F(f)"'] \\
K_n (l) \ar[r, "\Phi_l"] & F(l)
\end{tikzcd}
$$
Let $a \in K_n(m)$.
Below, we prove that $(\Phi_l \circ K_n (f) )(a) = \left( F(f) \circ \Phi_m \right) (a)$.
Note that we have 
$$a = \left( \mathcal{E}^\ast K_n (a ) \right) (x_1 \otimes x_2  \cdots \otimes x_n) .$$
where we regard $a$ on the right hand side as an element of $\mathsf{Cat}_{\mathfrak{Ass}^u} (n , m)$.
By Theorem \ref{202312061920}, there exists $g \in \mathsf{Cat}_{\mathfrak{Ass}^u} ( n,l)$ such that $\mathcal{E} (g) - f \circ \mathcal{E} (a) \in \mathcal{I} (n,l)$.
\begin{align*}
(\Phi_l \circ K_n (f) )(a) &= (\Phi_l \circ K_n (f) \circ \mathcal{E}^\ast K_n (a)) (x_1 \otimes x_2 \cdots \otimes x_n ) , \\
&= \left(\Phi_l \circ K_n (f \circ \mathcal{E} (a)) \right) (x_1 \otimes x_2 \cdots \otimes x_n ) , \\
&= ( \Phi_l \circ K_n ( \mathcal{E} (g) ) ) (x_1 \otimes x_2 \cdots \otimes x_n ) .
\end{align*}
where the final equality follows from the fact $\left( K_n ( \mathcal{I} (n,l) ) \right) ( x_1 \otimes x_2 \cdots \otimes x_n )= 0$, which follows from $\left( K_n (\mathrm{id}_{i-1} \ast \theta \ast \mathrm{id}_{n-i}) \right) ( x_1 \otimes x_2 \cdots \otimes x_n )= 0$.
The assumption $\Phi \in \mathrm{Nat} ( \mathcal{E}^\ast K_n , \mathcal{E}^\ast F)$ implies that
\begin{align*}
( \Phi_l \circ K_n ( \mathcal{E} (g) ) ) (x_1 \otimes x_2 \cdots \otimes x_n ) & = ( F( \mathcal{E} (g) ) \circ \Phi_n )(x_1 \otimes x_2 \cdots \otimes x_n ) , \\
&= \left( F( \mathcal{E} (g) ) \right) ( v ) .
\end{align*}
Note that we have $\left( F( \mathcal{I}( l , n) ) \right) (v) = 0$, since $v \in \left( \mathcal{P} (F) \right)(n)$ by the hypothesis and Proposition \ref{202312062129}.
Thus, we obtain $\left( F( \mathcal{E} (g) ) \right) ( v ) = \left( F ( f\circ \mathcal{E} (a) ) \right) (v)$.
By the assumption $\Phi \in \mathrm{Nat} ( \mathcal{E}^\ast K_n , \mathcal{E}^\ast F)$ again, we obtain
\begin{align*}
\left( F ( f\circ \mathcal{E} (a) ) \right) (v) &= \left( F(f) \circ \mathcal{E}^\ast F (a)  \circ \Phi_n \right) (x_1 \otimes x_2 \cdots \otimes x_n ) , \\
&= \left( F(f) \circ \Phi_m \circ \mathcal{E}^\ast K_n (a) \right) (x_1 \otimes x_2 \cdots \otimes x_n ) , \\
&= \left( F(f) \circ \Phi_m \right) (a) .
\end{align*}
\end{proof}

\begin{Defn}
\label{202312061102}
We denote by $I_n$ the induced natural transformation by applying the tensor-hom adjunction to $\check{I}_n$:
$$I_n : ~{}_{\Delta}\mathsf{Cat}_{\mathfrak{Ass}^u} (n , - ) \otimes \left( \mathcal{P} (F) \right)(n) \to F . $$
\end{Defn}

Then the following is well-defined since $I_k$'s are $\mathsf{gr}^{\mathsf{op}}$-homomorphisms by Theorem \ref{202212161526}.

\begin{Defn}
\label{202211251229}
For $F \in \mathcal{F} ( \mathsf{gr}^{\mathsf{op}} ; \mathds{k} )$, we define {\it the primitive filtration} of $F$:
$$
\begin{tikzcd}
& & F  & \\
 \cdots \ar[r, hookrightarrow] \ar[urr, shift left, hookrightarrow] & Q_n (F) \ar[r, hookrightarrow] \ar[ur, hookrightarrow] & Q_{n+1} (F) \ar[r, hookrightarrow] \ar[u, hookrightarrow] & \cdots \ar[ul, hookrightarrow]
\end{tikzcd}
$$
where $Q_n (F) \in \mathcal{F} ( \mathsf{gr}^{\mathsf{op}} ; \mathds{k} )$ is the image of $\sum_{k \leq n} I_k :  \bigoplus_{k \leq n} \left( {}_{\Delta} \mathsf{Cat}_{\mathfrak{Ass}^u} ( k, - ) \otimes \left( \mathcal{P} (F)  \right) (k) \right) \to F$.
The $\mathsf{gr}^{\mathsf{op}}$-module $F$ is {\it primitive} if $F \cong \mathrm{colim}_{n} ~ Q_n(F)$.
\end{Defn}

\begin{prop}
\label{202212160131}
Let $F \in \mathcal{F} ( \mathsf{gr}^{\mathsf{op}} ; \mathds{k} )$.
We have $Q_n(F) \subset P_n (F)$.
In particular, every primitive $\mathsf{gr}^{\mathsf{op}}$-module is analytic.
\end{prop}
\begin{proof}
Recall $I_n$ in Theorem \ref{202212161526}.
We apply the polynomial approximation functor $P_n : \mathcal{F} ( \mathsf{gr}^{\mathsf{op}} ; \mathds{k} ) \to \mathcal{F}_{\leq n} ( \mathsf{gr}^{\mathsf{op}} ; \mathds{k} )$ to the natural transformation $\sum_{k \leq n} I_k$.
By Proposition \ref{202312050902}, the application yields a natural transformation,
$$\bigoplus_{k \leq n} \left( {}_{\Delta} \mathsf{Cat}_{\mathfrak{Ass}^u} ( k, -) \otimes \left( \mathcal{P} (F)  \right) (k) \right) \to P_n (F) .$$
The definition of $Q_n(F)$ implies that $Q_n(F) \subset P_n (F)$.
\end{proof}

\subsection{Representability of the induced $\mathsf{Cat}_{\mathfrak{Lie}}$-module}
\label{202212202131Japan}

In this section, we show that the $\Sigma$-module structure on $\mathcal{P} (F)$ extends to a left $\mathsf{Cat}_{\mathfrak{Lie}}$-module structure.
Especially, we have a functor $\mathcal{P} : \mathcal{F} ( \mathsf{gr}^{\mathsf{op}} ; \mathds{k} ) \to \mathcal{F}_{\mathfrak{Lie}}$.
Furthermore, we go further by proving that the $( \mathds{k} \mathsf{gr}^{\mathsf{op}} , \mathsf{Cat}_{\mathfrak{Lie}})$-bimodule ${}_{\Delta} \mathsf{Cat}_{\mathfrak{Ass}^u}$ {\it represents} the functor $\mathcal{P} : \mathcal{F} ( \mathsf{gr}^{\mathsf{op}} ; \mathds{k} ) \to \mathcal{F}_{\mathfrak{Lie}}$.

Let $\mathcal{B} : \mathsf{Cat}_{\mathfrak{Lie}} \to \mathsf{Cat}_{\mathfrak{Ass}^u}$ be the functor induced by the operad homomorphism $\mathfrak{Lie} \to \mathfrak{Ass}^u$.
\begin{Lemma}
\label{202312061047}
For $n,m \in \mathds{N}$, the functor $\mathcal{B}$ induces a linear map 
$$\widetilde{\mathcal{B}}_{n,m} : \mathsf{Cat}_{\mathfrak{Lie}} (n, m) \to \mathcal{P} \left( {}_\Delta \mathsf{Cat}_{\mathfrak{Ass}^u} (n , - )  \right) (m) . $$
\end{Lemma}
\begin{proof}
Recall the notation in Definition \ref{202312062128}.
Let $f_r = \mathrm{id}_{r-1} \ast \theta \ast \mathrm{id}_{m-r} : m \to (m+1)$ and put $L_r = \mathrm{Ker} \left( {}_\Delta \mathsf{Cat}_{\mathfrak{Ass}^u} (n , f_r ) \right)$ for $1 \leq r \leq m$.
We have $\bigcap^{m}_{r=1} L_r = \mathcal{P} \left( {}_\Delta \mathsf{Cat}_{\mathfrak{Ass}^u} (n , - )  \right) (m)$ by Proposition \ref{202312062129}, so that it suffices to give a factorization of $\mathcal{B}$ as $ \mathsf{Cat}_{\mathfrak{Lie}} (n, m) \to L_r$ for arbitrary $r$.
Recall that in section \ref{Korea202306141722} we introduce ${}_\Delta \mathsf{Cat}_{\mathfrak{Ass}^u} (n , m )$ as the submodule of $\mathbb{H}(n,m)$ the tensor product of noncommutative polynomial bialgebras.
Let $M_r$ be the kernel of $\mathbb{H} (n,f_r ) :\mathbb{H} (n,m) \to \mathbb{H} (n,m+1)$.
By definition of ${}_\Delta \mathsf{Cat}_{\mathfrak{Ass}^u}$ (in section \ref{Korea202306141722}), we obtain the following commutative diagram whose rows and columns are all exact:
$$
\begin{tikzcd}
& 0 \ar[d] & 0 \ar[d] & 0 \ar[d] \\
0 \ar[r]  & L_r \ar[r] \ar[d]& {}_\Delta \mathsf{Cat}_{\mathfrak{Ass}^u} (n , m )\ar[r, "{}_\Delta \mathsf{Cat}_{\mathfrak{Ass}^u} (n {,} f_r )"]  \ar[d] & {}_\Delta \mathsf{Cat}_{\mathfrak{Ass}^u} (n , m+1 )  \ar[d] \\
0 \ar[r] & M_r \ar[r] & \mathbb{H} (n,m) \ar[r, "\mathbb{H} (n {,}  f_r)"] & \mathbb{H} (n, m+1) .
\end{tikzcd}
$$
One may apply this diagram to prove that $L_r$ is the pullback of the inclusions ${}_\Delta \mathsf{Cat}_{\mathfrak{Ass}^u} (n , m ) \to \mathbb{H}(n,m)$ and $M_r \to \mathbb{H} (n,m)$.
Hence, it suffices to construct $h : \mathsf{Cat}_{\mathfrak{Lie}} ( n ,m) \to M_r$ such that the diagram below commutes:
$$
\begin{tikzcd}
\mathsf{Cat}_{\mathfrak{Lie}} ( n ,m)  \ar[dr, dashrightarrow] \ar[drr, bend left, "\mathcal{B}"] \ar[ddr, bend right] & & \\
& L_r  \ar[d] \ar[r] & {}_\Delta \mathsf{Cat}_{\mathfrak{Ass}^u} (n , m ) \ar[d] \\
& M_r \ar[r] & \mathbb{H} (n,m) .
\end{tikzcd}
$$

The iterated Lie brackets of variables $x_1, x_2, \cdots , x_n$ lie in $\mathrm{Prim} (H_n)$.
Hence, we obtain a linear map $h^\prime : \mathsf{Cat}_{\mathfrak{Lie}} ( n ,m) \to \mathrm{Prim} (H_n)^{\otimes m}$.
On the one hand, the inclusion $i : \mathrm{Prim} (H_n) \to H_n$ induces $\mathrm{Prim} (H_n)^{\otimes m} \to M_r$ due to the following commutative diagram and the exactness of the row:
$$
\begin{tikzcd}
0 \ar[r] & M_r \ar[r] & \mathbb{H}(n,m) \ar[r, "\mathbb{H}(n{,} f_r)"] & \mathbb{H} (n, m+1) \\
& & \mathrm{Prim} (H_n)^{\otimes m} \ar[ur, "0"'] \ar[u] \ar[ul, dashrightarrow, "h^{\prime\prime}"] &
\end{tikzcd}
$$
Therefore, $h = h^{\prime\prime} \circ h^\prime : \mathsf{Cat}_{\mathfrak{Lie}} ( n ,m) \to  M_r$ is the linear map which satisfies our requirement.
\end{proof}

The inclusion $\mathcal{P}\left( {}_\Delta \mathsf{Cat}_{\mathfrak{Ass}^u} (n , - ) \right) \to  {}_\Delta \mathsf{Cat}_{\mathfrak{Ass}^u} (n ,  -)$ induces a natural transformation $\mathcal{P}\left( {}_\Delta \mathsf{Cat}_{\mathfrak{Ass}^u} (n , - ) \right) \otimes M \to  {}_\Delta \mathsf{Cat}_{\mathfrak{Ass}^u} (n , -) \otimes M$ for a $\mathds{k}$-module $M$.
By definition, this factors through the inclusion $\mathcal{P} \left(  {}_\Delta \mathsf{Cat}_{\mathfrak{Ass}^u} (n , - ) \otimes M \right) \to {}_\Delta \mathsf{Cat}_{\mathfrak{Ass}^u} (n , - ) \otimes M$, and we obtain the factorization:
$$\mathcal{U}^{M}_{n} : \mathcal{P}\left( {}_\Delta \mathsf{Cat}_{\mathfrak{Ass}^u} (n , - ) \right) \otimes M \to \mathcal{P} \left(  {}_\Delta \mathsf{Cat}_{\mathfrak{Ass}^u} (n , - ) \otimes M \right) .$$
By using this observation with $\widetilde{\mathcal{B}}_{n,m}$ in Lemma \ref{202312061047} and $I_n$ in Definition \ref{202312061102}, we introduce the following.
\begin{Defn} \label{202312081139}
For $n, m \in \mathds{N}$, we define
 $\mathcal{L}_{n,m} : \mathsf{Cat}_{\mathfrak{Lie}} (n,m) \otimes (\mathcal{P} (F) )(n) \to  \left( \mathcal{P} ( F) \right) (m)$ by the composition below.
$$
\begin{tikzcd}
\mathsf{Cat}_{\mathfrak{Lie}} (n,m) \otimes (\mathcal{P} (F) )(n) 
\ar[r, "\widetilde{\mathcal{B}}_{n{,}m}\otimes \mathrm{id}"] &  \mathcal{P} \left( {}_\Delta \mathsf{Cat}_{\mathfrak{Ass}^u} (n , - )  \right) (m) \otimes (\mathcal{P} (F) )(n)  \ar[d, "\mathcal{U}^{(\mathcal{P} (F) )(n) }_{n}"] \\
\left( \mathcal{P} (F) \right) (m)&  \mathcal{P} \left(  {}_\Delta \mathsf{Cat}_{\mathfrak{Ass}^u} (n , - ) \otimes (\mathcal{P} (F) )(n)  \right) (m) \ar[l, "\mathcal{P} (I_n)"] .
\end{tikzcd}
$$
\end{Defn}

The definition of $\mathcal{L}_{n,m}$ is unpacked as follows.
Let $a \in \mathsf{Cat}_{\mathfrak{Lie}} ( n,m)$ and $x \in \left( \mathcal{P} (F) \right) (n)$.
By definitions, we have
\begin{align}
\label{202312081627}
\mathcal{L}_{n,m} ( a \otimes x ) = J_n ( \mathcal{B} ( a) \otimes x ) .
\end{align}

\begin{prop}
\label{202212161700}
Let $F \in \mathcal{F} ( \mathsf{gr}^{\mathsf{op}} ; \mathds{k} )$.
The family $\mathcal{L}_{n,m}$ for $n,m \in \mathds{N}$ gives a $\mathsf{Cat}_{\mathfrak{Lie}}$-module structure on $\mathcal{P} (F)$, which is a sub $\mathsf{Cat}_{\mathfrak{Lie}}$-module of $\mathcal{B}^\ast \mathcal{E}^\ast (F)$.
\end{prop}
\begin{proof}
It suffices to prove that the diagram below commutes:
$$
\begin{tikzcd}
\mathsf{Cat}_{\mathfrak{Lie}} (m, l) \otimes \mathsf{Cat}_{\mathfrak{Lie}} (n, m) \otimes \left( \mathcal{P} (F) \right) (n) \ar[d, "\mathrm{id} \otimes \mathcal{L}_{n{,}m}"] \ar[r, "\circ \otimes \mathrm{id}"] & \mathsf{Cat}_{\mathfrak{Lie}} (n, l) \otimes  \left( \mathcal{P} (F) \right) (n) \ar[d, "\mathcal{L}_{n{,}l}"]  \\
\mathsf{Cat}_{\mathfrak{Lie}} (m , l) \otimes  \left( \mathcal{P} (F) \right) (m) \ar[r, "\mathcal{L}_{m{,}l}"] &  \left( \mathcal{P} (F) \right) (l) .
\end{tikzcd}
$$
It follows from 
\begin{align*}
\mathcal{L}_{m,l} ( b \otimes \mathcal{L}_{n,m} (a \otimes x ) ) &= J_m ( \mathcal{B} (b) \otimes  J_n ( \mathcal{B} ( a)  \otimes x)) , \\
&= J_n \left(  (\mathcal{B} (b) \circ \mathcal{B} ( a) ) \otimes x \right) , \\
&= J_n \left( \mathcal{B} (b \circ a) \otimes x \right) , \\
&= \mathcal{L}_{n,l} (  (b\circ a) \otimes x ) ,
\end{align*}
where $a \in \mathsf{Cat}_{\mathfrak{Lie}} ( n, m) , b \in \mathsf{Cat}_{\mathfrak{Lie}} ( m, l), x \in \left( \mathcal{P} ( F) \right) (n)$.
Furthermore, $\mathcal{P}(F)$ is a sub $\mathsf{Cat}_{\mathfrak{Lie}}$-module of $\mathcal{B}^\ast \mathcal{E}^\ast (F)$ since the $\mathsf{Cat}_{\mathfrak{Lie}}$-action is inherited from that of $\mathcal{B}^\ast \mathcal{E}^\ast (F)$ as we see in (\ref{202312081627}).
\end{proof}

For the category $\mathsf{Cat}_{\mathfrak{Ass}^u}$, the Yoneda embedding is given by $\mathsf{Cat}_{\mathfrak{Ass}^u}^{\mathsf{op}} \to \mathcal{F}_{\mathfrak{Ass}^u} ; n \mapsto \mathsf{Cat}_{\mathfrak{Ass}^u} (n , - )$ where $\mathcal{F}_{\mathfrak{Ass}^u}$ is the category of $\mathsf{Cat}_{\mathfrak{Ass}^u}$-modules and linear natural transformations.
For $G \in \mathcal{F}_{\mathfrak{Ass}^u}$, the Yoneda lemma implies that the evaluation on the identity $(x_1 \otimes \cdots \otimes x_n) \in \mathsf{Cat}_{\mathfrak{Ass}^u} ( n , n )$ (see section \ref{202212161644} for the notation) induces an isomorphism:
\begin{align}
\label{202212141625}
\mathrm{ev}_{x_1 \otimes \cdots \otimes x_n} : \mathrm{Nat} ( \mathsf{Cat}_{\mathfrak{Ass}^u} (n , - ) , G ) \to G (n) .
\end{align}
This isomorphism is natural with respect to $n \in \mathsf{Cat}_{\mathfrak{Ass}^u}$.

Let $F \in \mathcal{F} ( \mathsf{gr}^{\mathsf{op}} ; \mathds{k} )$.
For $G = \mathcal{E}^\ast F$ (see Definition \ref{202401101407} for $\mathcal{E}$), note that $\mathrm{Nat} ( {}_{\Delta} \mathsf{Cat}_{\mathfrak{Ass}^u} ( n , - ) , F )$ is naturally embedded into $\mathrm{Nat} ( \mathsf{Cat}_{\mathfrak{Ass}^u} (n , - ) , G )$.
We interpret the subspace via the Yoneda isomorphism in Theorem \ref{202212160932}.

In the following statements, we use the notation of the generators $\Delta, \nabla, \eta, \epsilon, s$ of the category $\mathsf{gr}^{\mathsf{op}}$ (see section \ref{Korea202306141722}).

\begin{Lemma}
\label{202212142110}
Let $F \in \mathcal{F} ( \mathsf{gr}^{\mathsf{op}} ; \mathds{k})$ and $n \in \mathds{N}$.
For $\xi \in \mathrm{Nat} ( {}_{\Delta} \mathsf{Cat}_{\mathfrak{Ass}^u} ( n , - ) , F )$, we have,
$$\xi (x_1 \otimes \cdots \otimes x_n) \in \left( \mathcal{P} (F) \right) ( n ) . $$
\end{Lemma}
\begin{proof}
It suffices to prove that $\xi ( x_1 \otimes \cdots \otimes x_n ) \in \mathrm{Ker} \left( F (\mathrm{id}_{m-1} \ast \theta \ast \mathrm{id}_{n-m}) \right)$ for any $1 \leq m \leq n$ by Proposition \ref{202312062129}.
For simplicity, let $K_n {:=} ~{}_{\Delta} \mathsf{Cat}_{\mathfrak{Ass}^u} ( n , - )$ be the $\mathds{k} \mathsf{gr}^{\mathsf{op}}$-module.
Since $\xi$ preserves $\mathds{k} \mathsf{gr}^{\mathsf{op}}$-actions, we have,
$$( F ( \Delta \ast \mathrm{id}_{n-1} ) ) ( \xi ( x_1 \otimes \cdots \otimes x_n ) ) = \xi ( (K_n ( \Delta \ast \mathrm{id}_{n-1} ) ) ( x_1 \otimes \cdots \otimes x_n ) ) . $$
By definition, we have 
$$(K_n ( \Delta \ast \mathrm{id}_{n-1} ) ) ( x_1 \otimes \cdots \otimes x_n ) = x_1 \otimes 1 \otimes x_2 \otimes \cdots x_n + 1 \otimes x_1 \otimes x_2 \otimes \cdots \otimes x_n , $$
which leads to the above claim for $m =1$.
One can proceed similarly for any $m$.
\end{proof}

We now prove that the $( \mathds{k} \mathsf{gr}^{\mathsf{op}} , \mathsf{Cat}_{\mathfrak{Lie}})$-bimodule ${}_{\Delta} \mathsf{Cat}_{\mathfrak{Ass}^u}$ {\it represents} the functor $\mathcal{P} : \mathcal{F} ( \mathsf{gr}^{\mathsf{op}} ; \mathds{k} ) \to \mathcal{F}_{\mathfrak{Lie}}$.
In the following statement, recall that $\mathcal{P} (F)$ is a left $\mathsf{Cat}_{\mathfrak{Lie}}$-module by Proposition \ref{202212161700}.
The assignment $n \mapsto \mathrm{Nat} ( {}_{\Delta} \mathsf{Cat}_{\mathfrak{Ass}^u} ( n , - ) , F )$ gives a left $\mathsf{Cat}_{\mathfrak{Lie}}$-module arising from the right $\mathsf{Cat}_{\mathfrak{Lie}}$-module structure of ${}_{\Delta} \mathsf{Cat}_{\mathfrak{Ass}^u}$.

\begin{theorem}
\label{202212160932}
Let $F \in \mathcal{F} ( \mathsf{gr}^{\mathsf{op}} ; \mathds{k})$ and $n \in \mathds{N}$.
The evaluation $\mathrm{ev}_{x_1 \otimes \cdots \otimes x_n}$ in (\ref{202212141625}) factors through an isomorphism which is natural with respect to $n \in \mathsf{Cat}_{\mathfrak{Lie}}$:
$$\mathrm{Nat} ( {}_{\Delta} \mathsf{Cat}_{\mathfrak{Ass}^u} ( n , - ) , F ) \cong \mathcal{P} (F) ( n ) . $$
\end{theorem}
\begin{proof}
The map $\mathrm{ev}_{x_1 \otimes \cdots \otimes x_n}$ factors into $h_n : \mathrm{Nat} (\mathsf{Cat}_{\mathfrak{Ass}^u} ( n , - ) , F ) \to \mathcal{P} (F) ( n )$ due to Lemma \ref{202212142110}.
Moreover, $h_n$ is injective since $\mathrm{ev}_{x_1 \otimes \cdots \otimes x_n}$ is an isomorphism.
The naturality of $h_n$ with respect to $n \in \mathsf{Cat}_{\mathfrak{Lie}}$ follows from that of (\ref{202212141625}).
Recall the map $I_n$ in Theorem \ref{202212161526}.
By definitions, we have $h_n \circ I_n$ is the identity.
Hence, $h_n$ is surjective.
\end{proof}

\begin{Corollary}
For $\mathds{k}$ a field of characteristic zero, a $\mathsf{gr}^{\mathsf{op}}$-module is primitive if and only if it is analytic.
\end{Corollary}
\begin{proof}
Let $F$ be a $\mathsf{gr}^{\mathsf{op}}$-module.
By Proposition \ref{202212160131}, if $F$ is primitive, then it is analytic.
Conversely, if $F$ is analytic, then the evaluation gives ${}_{\Delta} \mathsf{Cat}_{\mathfrak{Ass}^u} \otimes_{\mathsf{Cat}_{\mathfrak{Lie}}} \mathrm{Nat} ( {}_{\Delta} \mathsf{Cat}_{\mathfrak{Ass}^u} , F ) \cong F$ by \cite{powell2021analytic} since the characteristic is zero.
We take a composition with the isomorphism in Theorem \ref{202212160932} to deduce that $\sum_{k \in \mathds{N}} I_k$ is surjective.
Therefore, $F$ is primitive.
\end{proof}

\begin{remark}
For the ground field $\mathds{k}$ of positive characteristic, there exist analytic and non-primitive $\mathsf{gr}^{\mathsf{op}}$-modules.
Some examples arise by combining Theorem \ref{202207212109} and Remark \ref{2022112921241}.
\end{remark}

\section{Application to primitive exponential $\mathsf{gr}^{\mathsf{op}}$-modules}
\label{202312081805}

In this section, we prove an equivalence between primitive cocommutative Hopf algebras and primitive exponential $\mathsf{gr}^{\mathsf{op}}$-modules (see Definition \ref{202211251229} for the definition of primitive $\mathsf{gr}^{\mathsf{op}}$-modules).

\subsection{Primitive filtration and coradical filtration}
\label{202211301655}

In this section, we give a comparison of the coradical filtration and the primitive filtration of bialgebras which is reviewed below.

Let $\mathds{k}$ be a commutative unital ring.
For a bialgebra $H$ over $\mathds{k}$, its underlying coaugmented coalgebra induces the coradical filtration $P_\bullet (H)$.
The algebra structure gives a submodule $P_1 (H)^n \subset H$ generated by multiples of $n$ elements in $P_1 (H)$.
It yields {\it the primitive filtration} of $H$:
$$
\mathds{k} = P_1 (H)^0 \subset \cdots \subset P_1 (H)^n \subset P_1 (H)^{n+1} \subset \cdots \subset H .
$$
The bialgebra $H$ is {\it primitive} (or {\it primitively generated}) if $H = \bigcup_{n \in \mathds{N}} P_1 (H)^n$.

We first prove that the coradical filtration includes the primitive filtration, properly in general.
We also discuss some conditions implying that the two filtrations coincide.
In particular, for $\mathds{k}$ a field of characteristic zero, if the bialgebra is either cocommutative, primitive or finite-dimensional, then the filtrations are equal.

\begin{prop}
\label{202211271234}
Let $H$ be a bialgebra over $\mathds{k}$.
The coradical filtration contains the primitive filtration of $H$:
$$
P_1(H)^n \subset P_n (H) .
$$
\end{prop}
\begin{proof}
The result follows from an iterated application of Proposition \ref{202401101549}.
\end{proof}

We say that a bialgebra is {\it good} if its coradical filtration is equal to the primitive filtration.

\begin{Example}
We suppose that $m \cdot 1_{\mathds{k}} \in \mathds{k}$ is not a zero divisor for any nonzero integer $m$.
An application of the results in section \ref{202211071544} implies that polynomial bialgebras are good.
\end{Example}

\begin{Example}
\label{202211301553}
A bialgebra is not good in general.
Consider the shuffle bialgebra $H = \mathrm{Sh} (V)$ (see section \ref{202212111901}).
Note that for $v_1, v_2, \cdots, v_k \in V \subset \mathrm{Sh} (V)$, the multiplication $v_1 \shuffle v_2 \shuffle \cdots \shuffle v_k$ lies in the $k$-th symmetric tensor module contained in $T(V) = \mathrm{Sh} (V)$, i.e. the invariant subspace under the symmetric group $\Sigma_k$ action.
Hence, if $V$ is a free $\mathds{k}$-module with $\dim V \geq 2$, then $P_1 (H)^n \neq P_n (H)$ for $n \geq 2$ by Proposition \ref{202211291552}.
Furthermore, one can prove directly that, for $\mathds{k}$ a field of positive characteristic, $\mathrm{Sh} (V)$ is not good even if $\dim V =1$.
\end{Example}

\begin{Example}
If $\mathds{k}$ is a field of characteristic zero, then finite-dimensional bialgebras are good by Theorem \ref{202211292056}.
\end{Example}

We now suppose that the ground ring $\mathds{k}$ is a field.
Below, we give general properties satisfied by good bialgebras (see Theorem \ref{202208302145}).

\begin{prop}
\label{202312080622}
If $\mathds{k}$ is a field, then we have $\delta_H^n ( P_n (H)) \subset \mathrm{Prim} (H)^{\otimes n}$. 
\end{prop}
\begin{proof}
Recall the notations in section \ref{202212191823Japan}.
By definitions, we have $(\mathrm{id}_{H}^{\otimes (i-1)} \otimes \varepsilon_H \otimes \mathrm{id}_{H}^{\otimes (n-i)} ) \circ \delta_H^n = 0$ for $1 \leq i \leq n$.
This leads to the following by applying Lemma \ref{202207282250}:
$$\delta_H^n ( P_n (H)) \subset \bigcap_{i=1}^{n} \mathrm{Ker} (\mathrm{id}_{H}^{\otimes (i-1)} \otimes \varepsilon_H \otimes \mathrm{id}_{H}^{\otimes (n-i)} ) =  \overline{H}^{\otimes n}$$
Hence, to prove our claim, it suffices to show that 
$(\mathrm{id}_{\overline{H}}^{\otimes (i-1)} \otimes \overline{\Delta}_H  \otimes \mathrm{id}_{\overline{H}}^{\otimes (n-i)} ) \circ \delta_H^n (x)= 0$ for $x \in P_n (H)$ by Lemma \ref{202207282250}.
Note that Proposition \ref{202208302218} implies that 
$$(\mathrm{id}_{\overline{H}}^{\otimes (i-1)} \otimes \overline{\Delta}_H  \otimes \mathrm{id}_{\overline{H}}^{\otimes (n-i)} ) \circ \delta_H^n = \delta_H^{n+1} . $$
We now use the assumption $x \in P_n(H)$, equivalently $\delta_H^{n+1} (x) = 0$.
\end{proof}
By Proposition \ref{202312080622}, the map $\delta_H^n : H \to H^{\otimes n}$ factors through $f_n : P_n (H) \to \mathrm{Prim} (H)^{\otimes n}$.
By Proposition \ref{202211271234}, the iterated multiplication $\nabla^{(n)} : H^{\otimes n} \to H$ induces a linear map $g_n : P_1 (H)^{\otimes n} \to P_n (H)$.
Let $h_n : \mathrm{Prim} (H)^{\otimes n} \to P_n (H)$ be the restriction of $g_n$.
Let $s_n : \mathrm{Prim} (H)^{\otimes n} \to \mathrm{Prim} (H)^{\otimes n}$ be the linear extension of $s_n ( a_1 \otimes a_2 \otimes \cdots \otimes a_n  ) {:=} \sum_{\sigma \in \Sigma_n} a_{\sigma (1)} \otimes a_{\sigma (2)} \otimes \cdots \otimes a_{\sigma (n)}$.

\begin{Lemma}
\label{202209011150}
For $n \in \mathds{N}$, we have $f_n \circ h_n  = s_n$.
\end{Lemma}
\begin{proof}
Let $a_1, a_2 , \cdots , a_n  \in \mathrm{Prim} (H)$.
\begin{align*}
(f_n \circ h_n) ( a_1 \otimes a_2 \otimes \cdots \otimes a_n ) &= \delta^n ( a_1 a_2 \cdots a_n ) , \\
&= \delta_H^1 (a_1) \shuffle \delta_H^1 ( a_2 ) \shuffle \cdots \shuffle \delta_H^1 (a_n) ~ (\mathrm{by}~\mathrm{Lemma}~\ref{202209010909} ) , \\
&= a_1 \shuffle a_2 \shuffle \cdots \shuffle a_n ~ (\mathrm{by} ~\varepsilon (a_j) = 0) ,
\end{align*}
By definitions, we have $a_1 \shuffle a_2 \shuffle \cdots \shuffle a_n = \sum_{\sigma \in \Sigma_n} a_{\sigma (1)} \otimes a_{\sigma (2)} \otimes \cdots \otimes a_{\sigma (n)}$.
\end{proof}

In the following, consider the idempotent $e_H$ on $P_1 (H)$ defined by $e_H (x) = x - \varepsilon_H (x) \cdot 1_H$.

\begin{Lemma} \label{202312080706}
Let $W$ be the kernel of $e_H^{\otimes n}$.
Then we have $g_n ( W ) \subset P_{n-1} (H)$.
\end{Lemma}
\begin{proof}
Note that $W$ is the subspace generated by $a_1 \otimes a_2 \otimes \cdots \otimes a_n \in P_1 (H)^{\otimes n}$ such that $a_j \in \mathds{k} 1_H$ for some $1 \leq j \leq n$.
Hence, by Proposition \ref{202211271234}, we obtain $g_n ( W ) \subset P_{n-1} (H)$ .
\end{proof}

\begin{theorem}
\label{202208302145}
Let $H$ be a bialgebra over a field $\mathds{k}$.
Then the first part below implies the second part.
\begin{itemize}
\item The bialgebra $H$ is good.
\item For $n \in\mathds{N}$, we have $\delta_H^n \left( P_n (H) \right) \subset \left( H^{\otimes n} \right)^{\Sigma_n}$, i.e. the subspace $\delta_H^n \left( P_n (H) \right)$ is invariant under the permutation action of  $\Sigma_n$.
\end{itemize}
Moreover, if the characteristic of $\mathds{k}$ is zero, then the above properties are equivalent.
\end{theorem}
\begin{proof}
Suppose that $H$ is good.
Let $n \in \mathds{N}$.
We prove that, for  $a \in P_n (H)$, $\delta_H^n ( a ) \in H^{\otimes n}$ is fixed under the $\Sigma_n$-action.
The hypothesis implies that $g_n$ is surjective.
We choose $b \in P_1 (H)^{\otimes n}$ such that $a = g_n (b)$.
Recall the notations in Lemma \ref{202312080706}.
The idempotent $e^{\otimes n}_H$ on $P_1(H)^{\otimes n}$ induces a direct decomposition $P_1(H)^{\otimes n} \cong \mathrm{Prim}(H)^{\otimes n} \oplus W$.
Let $c \in \mathrm{Prim} (H)^{\otimes n}$ and $c^\prime \in W$ be the components of $b$ with respect to this decomposition.
By Lemma \ref{202312080706}, we have $\delta_H^n ( g_n (c^\prime) ) = 0$.
Hence, we obtain $\delta_H^n ( g_n (b) ) = \delta_H^n ( g_n (c)) = f_n ( h_n (c))$ which leads to $\delta_H^n (a) = s_n (c)$, by Lemma \ref{202209011150}.
Thus $\delta_H^n ( a ) \in H^{\otimes n}$ is fixed under the $\Sigma_n$-action.

We now consider $\mathds{k}$ a field of characteristic zero.
We start from the second part of our statement and prove that $P_1(H)^n = P_n (H)$ by induction on $n$.
It is obvious for $n=1$.
Suppose that $P_1(H)^{n-1} = P_{n-1} (H)$ for $n \geq 2$.
It suffices to prove that $P_n (H) \subset P_1 (H)^n$ by Proposition \ref{202211271234}.
Let $a \in P_n (H)$.
By the hypothesis $\delta_H^n \left( P_n (H) \right) \subset \left( H^{\otimes n} \right)^{\Sigma_n}$, we have 
$$s_n ( \delta_H^n ( a ) ) = \sum_{\sigma \in \Sigma_n} \sigma ( \delta_H^n (a ) ) = n!~ \delta_H^n (a) , $$ 
or equivalently, $s_n \circ f_n = n! ~ f_n$ .
By Lemma \ref{202209011150}, we obtain
\begin{align} \label{202312081023}
f_n \circ h_n \circ f_n = n! ~ f_n .
\end{align}
Let $\pi_n : P_n (H) \to P_n (H) / P_{n-1} (H)$ be the projection.
Note that $f_n$ induces an injection $\overline{f_n} : P_n (H) / P_{n-1} (H) \to \mathrm{Prim} (H)^{\otimes n}$ such that $\overline{f_n} \circ \pi_n = f_n$ since $P_{n-1} (H)$ is the kernel of $\delta^n$.
By (\ref{202312081023}), we have $\overline{f_n} \circ \pi_n \circ h_n \circ f_n = n! ~ \overline{f_n} \circ \pi_n$ which gives $\pi_n \circ h_n \circ f_n = n!~ \pi_n$ since $\overline{f_n}$ is injective.
Hence, we have $a - (n!)^{-1} \cdot h_n (f_n (a)) \in P_{n-1} (H)$ (here, we use the assumption on the characteristic).
By the hypothesis of the induction, we have $a -(n!)^{-1} \cdot h_n (f_n (a)) \in P_1 (H)^{n-1}$.
Note that $(n!)^{-1} \cdot h_n (f_n (a)) \in P_1(H)^n$ by the definition of $h_n$.
Therefore, we have $a \in P_1 (H)^n$, since $P_1 (H)^{n-1} \subset P_1 (H)^n$.
\end{proof}

\begin{Corollary}
\label{202208302300}
Let $\mathds{k}$ be a field of characteristic zero.
Every cocommutative bialgebra over $\mathds{k}$ is good.
In particular, a primitive bialgebra over $\mathds{k}$ is good.
\end{Corollary}
\begin{proof}
It is obvious that $\delta^2_H : H \to H \otimes H$ is fixed under the $\Sigma_2$-action, since $H$ is cocommutative.
Moreover, $\delta_H^n$ is equal to the $n$-times iterated application of $\delta^2_H$, and thus, $\delta^n_H$ is fixed under the $\Sigma_n$-action.
Hence, the converse in Theorem \ref{202208302145} implies the claim.
The last claim follows from the fact that a primitive bialgebra is cocommutative.
\end{proof}

\begin{remark}
If $\mathds{k}$ is a field of positive characteristic, there exist cocommutative bialgebras which are not good.
For example, the shuffle bialgebra $Sh (\mathds{k})$ is cocommutative but not good (see Example \ref{202211301553}).
\end{remark}

\begin{Corollary}
\label{202212121035}
Let $\mathds{k}$ be a field of characteristic zero.
A bialgebra over $\mathds{k}$ is primitive if and only if it is conilpotent and good.
\end{Corollary}
\begin{proof}
A primitive bialgebra $H$ is conilpotent since $H = \bigcup_{n \in \mathds{N}} P_1 (H)^n \subset \bigcup_{n \in \mathds{N}} P_n (H) \subset H$.
By Corollary \ref{202208302300}, the only if part is proved.
Conversely, for a conilpotent and good bialgebra $H$, we have $\bigcup_{n \in \mathds{N}} P_1 (H)^n = \bigcup_{n \in \mathds{N}} P_n (H)$ since it is good. 
The conilpotency of $H$ implies $\bigcup_{n \in \mathds{N}} P_1 (H)^n = H$.
\end{proof}

\begin{remark}
\label{202212121042}
Consider $\mathds{k}$ a field of characteristic zero. 
One can deduce from a classical structure theorem for bialgebras (e.g. see \cite[Theorem 1.3.4]{loday2012algebraic}) that conilpotency is equivalent to primitivity for a cocommutative bialgebra.
We remark that Corollary \ref{202208302300} and \ref{202212121035} lead to the same conclusion without applying the structure theorem.
\end{remark}

\subsection{Primitive exponential $\mathsf{gr}^{\mathsf{op}}$-modules}
\label{202211251536}

In this section, we give a relation between primitive cocommutative Hopf algebras and primitive exponential $\mathsf{gr}^{\mathsf{op}}$-modules.
In particular, we prove an equivalence of them if the ground ring $\mathds{k}$ is a field.

Let $\mathds{k}$ be a commutative unital ring.
Recall the functor $\mathrm{ev}_{1} : \mathcal{F}^\mathsf{exp} ( \mathsf{gr^{\mathsf{op}}} ; \mathds{k} ) \to \mathsf{Hopf}^\mathsf{cc}_\mathds{k}$ in Remark \ref{202207311722}.
In the following statements, let $H = \mathrm{ev}_{1} (F)$ be the induced cocommutative Hopf algebra for $F \in \mathcal{F}^\mathsf{exp} ( \mathsf{gr}^{\mathsf{op}} ; \mathds{k} )$.
Note that the exponentiality of $F$ induces the isomorphism $H^{\otimes m} \cong F(m)$.

For a Lie algebra $\mathfrak{g}$, the assignment of $\mathfrak{g}^{\otimes n}$ to $n \in \mathds{N}$ extends to a left $\mathsf{Cat}_{\mathfrak{Lie}}$-module which we denote by $\mathfrak{g}^{\otimes}$.
This construction is natural so that if we have a Lie algebra homomorphism $f : \mathfrak{g} \to \mathfrak{h}$, then it induces a natural transformation $\mathfrak{g}^{\otimes} \to \mathfrak{h}^{\otimes}$.
For a Hopf algebra $H$, we are mainly interested in the Lie algebras $\mathrm{Prim} (H)$ and $H$ itself, which are Lie algebras equipped with the commutator.

\begin{Lemma} \label{202312081633}
Let $H$ be a cocommutative Hopf algebra over $\mathds{k}$.
The inclusion $\mathrm{Prim} (H) \hookrightarrow H$ induces the natural transformation 
$$
\mathrm{Prim} (H)^{\otimes} \to \mathcal{B}^\ast \mathcal{E}^\ast \left( \alpha ( H) \right)
$$
where $\mathsf{Cat}_{\mathfrak{Lie}} \stackrel{\mathcal{B}}{\to} \mathsf{Cat}_{\mathfrak{Ass}^u} \stackrel{\mathcal{E}}{\to} \mathds{k} \mathsf{gr}^{\mathsf{op}}$ are the canonical functors introduced in section \ref{Korea202306121601}.
\end{Lemma}
\begin{proof}
The inclusion $\mathrm{Prim}(H) \to H$ is a Lie algebra homomorphism.
Hence, this induces a natural transformation $\mathrm{Prim}(H)^{\otimes} \to H^{\otimes}$ where $H^{\otimes} : \mathsf{Cat}_{\mathfrak{Lie}} \to \mathsf{Mod}_{\mathds{k}} ; n \mapsto H^{\otimes n}$ is the functor induced by the Lie algebra structure on $H$.
Note that we have $H^{\otimes} = \mathcal{B}^\ast \mathcal{E}^\ast \left( \alpha ( H) \right)$ by definitions.
\end{proof}

Let $F \in \mathcal{F}^{\mathsf{exp}} ( \mathsf{gr}^{\mathsf{op}} ; \mathds{k} )$.
Applying Lemma \ref{202312081633} to $\mathrm{ev}_1 (F)$ gives a natural transformation $\mathrm{Prim} ( \mathrm{ev}_1 (F) )^{\otimes} \to \mathcal{B}^\ast \mathcal{E}^\ast \left( \alpha ( \mathrm{ev}_1 (F) ) \right) \cong  \mathcal{B}^\ast \mathcal{E}^\ast ( F )$.

Recall that we have a natural inclusion $\mathcal{P} ( F) \hookrightarrow \mathcal{B}^\ast \mathcal{E}^\ast ( F )$ of $\mathsf{Cat}_{\mathfrak{Lie}}$-modules by Proposition \ref{202212161700}.
\begin{Lemma}
\label{Korea202306161806}
Let $F \in \mathcal{F}^{\mathsf{exp}} ( \mathsf{gr}^{\mathsf{op}} ; \mathds{k} )$.
The natural transformation $\mathrm{Prim} ( \mathrm{ev}_1 (F) )^{\otimes} \to \mathcal{B}^\ast \mathcal{E}^\ast ( F )$ in Lemma \ref{202312081633} factors through $\mathcal{P} ( F) \hookrightarrow \mathcal{B}^\ast \mathcal{E}^\ast ( F )$:
$$
\begin{tikzcd}
\mathrm{Prim} ( \mathrm{ev}_1 (F) )^{\otimes} \ar[rr] \ar[dr, "\exists \beta"'] & &\mathcal{B}^\ast \mathcal{E}^\ast ( F ) \\
& \mathcal{P} ( F) \ar[ur, hookrightarrow] . &
\end{tikzcd}
$$
If $\mathds{k}$ is a field, then the factorized natural transformation $\beta$ is an isomorphism.
\end{Lemma}
\begin{proof}
Let $f_n: \mathrm{Prim} ( \mathrm{ev}_1 (F) )^{\otimes n} \to \mathrm{ev}_1 (F)^{\otimes n} = F(1)^{\otimes n} \cong F(n)$ be the $n$-fold tensor product of the inclusions.
It suffices to show the objectwise factorization below since $\mathcal{P} ( F)$ is a sub $\mathsf{Cat}_{\mathfrak{Lie}}$-module of $\mathcal{B}^\ast \mathcal{E}^\ast ( F )$:
$$
\begin{tikzcd}
F(n) \ar[r, equal] & F ( n) \\
\mathrm{Prim} ( \mathrm{ev}_1 (F) )^{\otimes n} \ar[r, "\exists \beta_n"] \ar[u, "f_n"] & ( \mathcal{P} ( F) ) (n)  .  \ar[u, hookrightarrow] 
\end{tikzcd}
$$
By Proposition \ref{202312062129}, it suffices to prove that the image of $f_n$ lies in $L_i = \mathrm{Ker} \left( F (\mathrm{id}_{i-1} \ast \theta \ast \mathrm{id}_{n-i}) \right)$ for every $1 \leq i \leq n$.
In fact, we have $\mathrm{Prim} (\mathrm{ev}_1 (F)) = \mathrm{Ker} ( F(\theta))$ by definition, so that we have $F (\mathrm{id}_{i-1} \ast \theta \ast \mathrm{id}_{n-i}) \circ f_n = 0$.
$$
\begin{tikzcd}
0 \ar[r] & L_i \ar[r] & F(n) \ar[r, "F (\mathrm{id}_{i-1} \ast \theta \ast \mathrm{id}_{n-i})"] & F(n+1) \\
&& \mathrm{Ker} ( F(\theta))^{\otimes n} \ar[ul, dashrightarrow] \ar[u, "f_n"] \ar[ur, "0"] . &
\end{tikzcd}
$$

We now assume that $\mathds{k}$ is a field, 
By applying Lemma \ref{202207282250}, we see that $L_i = F(1)^{\otimes (i-1)} \otimes \mathrm{Ker} (F(\theta)) \otimes F(1)^{\otimes (n-i)} $.
In particular, Proposition \ref{202312062129} implies that $\left( \mathcal{P} (F) \right) (n) = \bigcap^{n}_{i=1} L_i = \mathrm{Ker} (F(\theta))^{\otimes n}$ so that $\beta_n$ gives an isomorphism.
\end{proof}

Let $\mathsf{Hopf}^\mathsf{cc,prim}_\mathds{k} $ be the full subcategory of $\mathsf{Hopf}^\mathsf{cc}_\mathds{k}$ consisting of primitive Hopf algebras.

Let $\mathcal{F}^\mathsf{exp}_\mathsf{prim} ( \mathsf{gr^{\mathsf{op}}} ; \mathds{k} )$ be the full subcategory of $\mathcal{F}^\mathsf{exp} ( \mathsf{gr^{\mathsf{op}}} ; \mathds{k} )$ consisting of primitive $\mathsf{gr^{\mathsf{op}}}$-modules.

In the following, recall the functor $\alpha : \mathsf{Hopf}^\mathsf{cc}_\mathds{k} \to \mathcal{F}^\mathsf{exp}( \mathsf{gr^{\mathsf{op}}} ; \mathds{k} )$ in Remark \ref{202207311722}, which gives an equivalence of categories.
\begin{theorem}
\label{202312081733}
The functor $\alpha$ induces a fully faithful functor:
$$
\mathsf{Hopf}^\mathsf{cc,prim}_\mathds{k} \to \mathcal{F}^\mathsf{exp}_\mathsf{prim} ( \mathsf{gr^{\mathsf{op}}} ; \mathds{k} )  .
$$
If $\mathds{k}$ is a field, then it gives an equivalence of categories.
\end{theorem}
\begin{proof}
We first prove that, if $H$ is a cocommutative primitive Hopf algebra, then $F = \alpha (H)$ is primitive.
Let $m \in \mathds{N}$.
Recall the notations in Theorem \ref{202212161526}.
For $k \in \mathds{N}$, consider the following diagram:
\begin{equation}
\label{Korea202306161832}
\begin{tikzcd}
H^{\otimes m} \ar[r, "\cong"] & F( m ) \\
\mathsf{Cat}_{\mathfrak{Ass}^u} ( k , m ) \otimes H^{\otimes k} \ar[u, "T_k"] \ar[r, "\cong"] & \mathsf{Cat}_{\mathfrak{Ass}^u} ( k , m ) \otimes  F ( k) \ar[u, "J_k"] \\
{}_{\Delta} \mathsf{Cat}_{\mathfrak{Ass}^u} ( k , m ) \otimes \mathrm{Prim} ( H )^{\otimes k} \ar[r, "\mathrm{id} \otimes \beta_k"] \ar[u] &  {}_{\Delta} \mathsf{Cat}_{\mathfrak{Ass}^u} ( k , m ) \otimes  \left( \mathcal{P} ( F) \right) ( k) . \ar[u] 
\end{tikzcd}
\end{equation}
The map $T_k$ is induced by the underlying algebra structure of $H$.
Note that ${}_{\Delta} \mathsf{Cat}_{\mathfrak{Ass}^u} ( k , m ) = \mathsf{Cat}_{\mathfrak{Ass}^u} ( k , m )$.
The top square commutes due to the definition of $J_k$: it is induced by the $\mathsf{Cat}_{\mathfrak{Ass}^u}$-module structure on $\mathcal{E}^\ast F$.
The square below commutes due to the definition of $\beta$ in Lemma \ref{Korea202306161806}.
Let $\widetilde{T}_k$ be the composition of the left vertical maps.
By definition, if $H$ is primitive, then $H^{\otimes m}$ is equal to the image of $\Sigma_{k \in \mathds{N}} \widetilde{T}_k$.
Hence, if $H$ is primitive, then $F( m )$ is equal to the image of $\Sigma_{k \in \mathds{N}} \widetilde{J}_k$, i.e. the $\mathsf{gr}^{\mathsf{op}}$-module  $F = \alpha (H)$ is primitive.
Therefore, the composition $\mathsf{Hopf}^\mathsf{cc,prim}_\mathds{k}  \hookrightarrow \mathsf{Hopf}^\mathsf{cc}_\mathds{k} \stackrel{\mathrm{ev}_1}{\to} \mathcal{F}^\mathsf{exp}( \mathsf{gr^{\mathsf{op}}} ; \mathds{k} )$ factors through $\mathcal{F}^\mathsf{exp}_\mathsf{prim} ( \mathsf{gr^{\mathsf{op}}} ; \mathds{k} )$.
Moreover, the factorized functor is faithfully full since so is $\mathrm{ev}_1$.

We now suppose that $\mathds{k}$ is a field.
It suffices to prove that an exponential $\mathsf{gr}^{\mathsf{op}}$-module $F$ is primitive only if the Hopf algebra $H = \mathrm{ev}_{1} ( F )$ is primitive.
By the last statement in Lemma \ref{Korea202306161806}, the horizontal map $\mathrm{id} \otimes \beta_k$ in (\ref{Korea202306161832}) is an isomorphism for any $k$.
Thus, if $F$ is primitive, then $H^{\otimes m}$ is equal to the image of $\Sigma_{k \in \mathds{N}} \widetilde{T}_k$, by the commutative diagram (\ref{Korea202306161832}).
By applying this to $m= 1$, it proved that $H$ is primitive.
\end{proof}

\begin{remark}
\label{2022112921241}
Every primitive bialgebra is conilpotent for any ground ring $\mathds{k}$ (see Proposition \ref{202211271234}).
For a cocommutative bialgebra over a field of characteristic zero, its conilpotency is equivalent to primitivity (see Remark \ref{202212121042}).
Hence, if $\mathds{k}$ is a field of characteristic zero, the equivalences of categories in Theorem \ref{202207212109} and \ref{202312081733} are the same.
In positive characteristic, however, $\mathsf{Hopf}^\mathsf{cc,prim}_\mathds{k} \subset \mathsf{Hopf}^\mathsf{cc,conil}_\mathds{k}$ is an (essentially) proper inclusion.
For example, the shuffle bialgebra $\mathrm{Sh} ( \mathds{k} )$ associated with the one-dimensional vector space is conilpotent but non-primitive (see section \ref{202212111901}).
In particular, the inclusion $\mathcal{F}^{\mathsf{exp}}_{\mathsf{prim}} ( \mathsf{gr}^{\mathsf{op}} ; \mathds{k}) \hookrightarrow \mathcal{F}^{\mathsf{exp}}_\omega ( \mathsf{gr}^{\mathsf{op}} ; \mathds{k})$ (by Proposition \ref{202212160131}) is (essentially) proper.
\end{remark}

\subsection{Equivalence with restricted Lie algebras}
In this section, we suppose that the ground ring $\mathds{k}$ is a field of positive characteristic $p$.
Then the functors $P : \mathsf{Hopf}^\mathsf{cc,conil}_\mathds{k} \to \mathsf{Lie}_\mathds{k}$ and $U : \mathsf{Lie}_\mathds{k} \to \mathsf{Hopf}^\mathsf{cc,conil}_\mathds{k}$ in section \ref{202301131738Japan} do not induce an equivalence of categories.
We give some counterexamples below.

\begin{Example}
The functor $P$ is not full.
Consider the conilpotent Hopf algebra $\mathds{k} [t]$ in section \ref{202211071544}.
By Theorem \ref{202211071539}, $P_1 (\mathds{k} [t])$ is infinite-dimensional since $L_p (i) =1$ if and only if $i = p^k$ for some $k \in \mathds{N}$.
Hence, $P ( \mathds{k} [t] )$ is an infinite-dimensional abelian Lie algebra, and thus, the map $P : \mathrm{End}_{\mathsf{Hopf}_\mathds{k}} (\mathds{k} [t]) \to \mathrm{End}_{\mathsf{Lie}_\mathds{k}} (P(\mathds{k} [t]))$ is not surjective.
\end{Example}
\begin{Example}
The functor $U$ is not essentially surjective.
For example, the conilpotent Hopf algebra $H$ in Example \ref{202212121114} is not isomorphic to any universal enveloping algebra.
\end{Example}

There is a well-known modification of (\ref{202211102107}) in positive characteristic (see \cite[Theorem 6.11]{MM}).
In fact, we have a well-behaved subcategory $\mathsf{Hopf}^\mathsf{cc,prim}_\mathds{k} \subset \mathsf{Hopf}^\mathsf{cc,conil}_\mathds{k}$ which is (essentially) proper by Remark \ref{2022112921241}.
The assignment of the induced restricted Lie algebra $\bar{P} (H)$ to a primitive cocommutative Hopf algebra $H$ gives an equivalence of categories:
$$
\bar{P} : \mathsf{Hopf}^\mathsf{cc,prim}_\mathds{k} \stackrel{\simeq}{\longrightarrow}  \mathsf{ResLie}_\mathds{k}
$$
where $\mathsf{ResLie}_\mathds{k}$ is the category of restricted Lie algebras.

In the following statement, we denote by $ \mathsf{ResLie}^\mathsf{ab}_\mathds{k} $ the category of abelian restricted Lie algebras.
\begin{Corollary}
\label{202211301650}
If $\mathds{k}$ is a field with positive characteristic, then the composition of the functor $\mathrm{ev}_1$ with $\bar{P}$ induces equivalences of categories:
\begin{enumerate}
\item
$\mathcal{F}^\mathsf{exp}_\mathsf{prim} ( \mathsf{gr^{\mathsf{op}}} ; \mathds{k} ) \simeq \mathsf{ResLie}_\mathds{k}$.
\item
$\mathcal{F}^\mathsf{exp}_{\mathsf{prim}, \mathsf{out}} ( \mathsf{gr^{\mathsf{op}}} ; \mathds{k} )  \simeq \mathsf{ResLie}^\mathsf{ab}_\mathds{k} $.
\end{enumerate}
\end{Corollary}
\begin{proof}
The first claim is immediate from Theorem \ref{202312081733}.
The second one is proved by combining Theorem \ref{202207311733} and Theorem \ref{202312081733}.
\end{proof}

\bibliography{reference}{}

\begin{thebibliography}{10}

\bibitem{D}
Aur\'{e}lien Djament.
\newblock D\'{e}composition de {H}odge pour l'homologie stable des groupes
  d'automorphismes des groupes libres.
\newblock {\em Compos. Math.}, 155(9):1794--1844, 2019.

\bibitem{DV}
Aur\'{e}lien Djament and Christine Vespa.
\newblock Sur l'homologie des groupes d'automorphismes des groupes libres \`a
  coefficients polynomiaux.
\newblock {\em Comment. Math. Helv.}, 90(1):33--58, 2015.

\bibitem{EML}
Samuel Eilenberg and Saunders Mac~Lane.
\newblock On the groups {$H(\Pi,n)$}. {II}. {M}ethods of computation.
\newblock {\em Ann. of Math. (2)}, 60:49--139, 1954.

\bibitem{Hab}
Kazuo Habiro.
\newblock {On the category of finitely generated free groups}.
\newblock {\em ArXiv:1609.06599}, September 2016.

\bibitem{HM}
Kazuo Habiro and Gw\'{e}na\"{e}l Massuyeau.
\newblock The {K}ontsevich integral for bottom tangles in handlebodies.
\newblock {\em Quantum Topol.}, 12(4):593--703, 2021.

\bibitem{HPV}
Manfred Hartl, Teimuraz Pirashvili, and Christine Vespa.
\newblock Polynomial functors from algebras over a set-operad and nonlinear
  {M}ackey functors.
\newblock {\em Int. Math. Res. Not. IMRN}, (6):1461--1554, 2015.

\bibitem{MR3765469}
Martin Kassabov and Sasha Patotski.
\newblock Character varieties as a tensor product.
\newblock {\em J. Algebra}, 500:569--588, 2018.

\bibitem{Katada2}
Mai Katada.
\newblock Actions of automorphism groups of free groups on spaces of {J}acobi
  diagrams. {II}.
\newblock {\em Journal of the Institute of Mathematics of Jussieu}, pages
  1--69, 2022.

\bibitem{Katada1}
Mai Katada.
\newblock Actions of automorphism groups of free groups on spaces of {J}acobi
  diagrams. {I}.
\newblock {\em Annales de l'Institut Fourier}, 73(4):1489--1532, 2023.

\bibitem{KV}
Nariya Kawazumi and Christine Vespa.
\newblock On the wheeled {PROP} of stable cohomology of ${A}ut({F}_n)$ with
  bivariant coefficients.
\newblock arXiv:2105.14497, To appear in Algebraic \& Geometric Topology.

\bibitem{loday2012algebraic}
Jean-Louis Loday and Bruno Vallette.
\newblock {\em Algebraic operads}, volume 346 of {\em Grundlehren der
  mathematischen Wissenschaften [Fundamental Principles of Mathematical
  Sciences]}.
\newblock Springer, Heidelberg, 2012.

\bibitem{MM}
John~W. Milnor and John~C. Moore.
\newblock On the structure of {H}opf algebras.
\newblock {\em Ann. of Math. (2)}, 81:211--264, 1965.

\bibitem{Pirashvili}
Teimuraz Pirashvili.
\newblock On the {PROP} corresponding to bialgebras.
\newblock {\em Cah. Topol. G\'eom. Diff\'er. Cat\'eg.}, 43(3):221--239, 2002.

\bibitem{powell2021analytic}
Geoffrey Powell.
\newblock On analytic contravariant functors on free groups.
\newblock arXiv:2110.01934.

\bibitem{P-21}
Geoffrey Powell.
\newblock Outer functors and a general operadic framework.
\newblock arXiv:2201.13307.

\bibitem{PV}
Geoffrey Powell and Christine Vespa.
\newblock Higher {H}ochschild homology and exponential functors.
\newblock arXiv:1802.07574.

\bibitem{radford1979natural}
David~E. Radford.
\newblock A natural ring basis for the shuffle algebra and an application to
  group schemes.
\newblock {\em J. Algebra}, 58(2):432--454, 1979.

\bibitem{touze}
Antoine Touz\'{e}.
\newblock On the structure of graded commutative exponential functors.
\newblock {\em Int. Math. Res. Not. IMRN}, (17):13305--13415, 2021.

\bibitem{V-Jacobi}
Christine Vespa.
\newblock On the functors associated with beaded open {J}acobi diagrams.
\newblock arXiv:2202.10907.

\end{thebibliography}
\bibliographystyle{plain}

\end{document}